\documentclass[reqno,oneside,11pt]{amsart}

\usepackage[a4paper, top=3.5cm, bottom=3.5cm, left = 3.5cm, right = 3.5cm]{geometry}

\usepackage[utf8]{inputenc}
\usepackage[T1]{fontenc}
\usepackage[svgnames,hyperref]{xcolor}
\usepackage{stmaryrd,lmodern,amssymb}
\usepackage{MnSymbol}
\usepackage{ctable}

\usepackage[cmtip]{xy}
\xyoption{ps}
\xyoption{color}\UseCrayolaColors
\xyoption{dvips}
\xyoption{all}

\newcommand{\mygraph}[1]{\xybox{\xygraph{#1}}}
\usepackage[shortlabels]{enumitem}
\setlist[enumerate]{label=\rm{(\arabic*)}}
\setlist[enumerate,2]{label=\rm({\it\roman*})}
\setlist[itemize]{label=\raisebox{0.25ex}{\tiny$\bullet$}}

\theoremstyle{plain}
 
\newtheorem{theorem}{Theorem}[section]
\newtheorem{corollary}[theorem]{Corollary}
\newtheorem{proposition}[theorem]{Proposition}
\newtheorem{lemma}[theorem]{Lemma}
\newtheorem{IH}[theorem]{Induction Hypothesis}
\newtheorem{fact}[theorem]{Fact}

\theoremstyle{definition}

\newtheorem{example}[theorem]{Example}
\newtheorem{nonexample}[theorem]{Non Example}
\newtheorem{remark}[theorem]{Remark}
\newtheorem{setup}[theorem]{Set-Up}

\makeatletter
\let\c@equation\c@theorem
\makeatother

\newcommand{\Q}{\mathbb{Q}}
\newcommand{\Z}{\mathbb{Z}}
\newcommand{\N}{\mathbb{N}}
\newcommand{\A}{{\mathbb{A}}}
\newcommand{\p}{\mathbb{P}}

\newcommand{\TT}{\mathbb{T}} % tensor algebra

\newcommand{\K}{\mathbf{k}}

\renewcommand{\LL}{\mathcal{L}}

\DeclareMathOperator{\GL}{GL}

\DeclareMathOperator{\Aut}{Aut}

\DeclareMathOperator{\Stab}{Stab}

\DeclareMathOperator{\topdeg}{topdeg}

\DeclareMathOperator{\Tame}{Tame}

\DeclareMathOperator{\car}{char}

\newcommand{\PF}{\text{$\circlearrowleft$\!\,}}

\newcommand{\Comp}{\mathcal{C}}

\renewcommand{\phi}{\varphi}
\newcommand{\id}{\text{\rm id}}
\newcommand{\TA}{\Tame(\A^3)}
\renewcommand{\ne}{\between}
\newcommand{\dvirt}{\deg_{\text{\rm virt}}}
\newcommand{\llb}{\llbracket}
\newcommand{\rrb}{\rrbracket}
\newcommand{\lines}[1]{\p^2(#1)}
\newcommand{\linesd}[1]{\hat\p^2(#1)}
\newcommand{\parti}[2]{\tfrac{\partial #2}{\partial x_{#1}}}
\newcommand{\topcomp}[2]{\overline{#1}^{#2}}

\renewcommand{\le}{\leqslant}
\renewcommand{\ge}{\geqslant}

\newcommand{\typeone}{\xy*\cir<2pt>{}\endxy}
\newcommand{\typetwo}{\bullet}
\newcommand{\typethree}{\rule[.2ex]{0.9ex}{0.9ex}}

\newcolumntype{L}{>{$}l<{$}}
\newcolumntype{C}{>{$}c<{$}}

\newcommand{\mycolor}{Maroon}
\usepackage[pdfauthor={S. Lamy}, colorlinks, linktocpage, citecolor = \mycolor, linkcolor = \mycolor, urlcolor = \mycolor]{hyperref}
\usepackage[all]{hypcap} % needed to help hyperlinks direct correctly;

\DeclareRobustCommand{\SkipTocEntry}[5]{}

\title[Combinatorics of the tame automorphism group]{Combinatorics of the tame automorphism group}
\author{Stéphane Lamy}
\date{\today}
\thanks{This research was partially supported by ANR Grant ``BirPol''  ANR-11-JS01-004-01.}
\address{Institut de Math\'ematiques de Toulouse, Universit\'e Paul Sabatier, 118 route de Narbonne, 31062 Toulouse Cedex 9, France}
\email{slamy@math.univ-toulouse.fr}

\begin{document}

\begin{abstract}
We study the group $\TA$ of tame automorphisms of the 3-dimensional affine space, over a field of characteristic zero.
We recover, in a unified way, previous results of Kuroda, Shestakov, Umirbaev
and Wright, about the theory of reduction and the relations in $\TA$.
The novelty in our presentation is the emphasis on a simply connected  2-dimensional simplicial complex on which $\TA$ acts by isometries.
\end{abstract}

\maketitle

\setcounter{tocdepth}{2}
\tableofcontents

\section*{Introduction}

Let $\K$ be a field, and let $\A^n = \A_\K^n$ be the affine space over $\K$.
We are interested in the group $\Aut(\A^n)$ of algebraic automorphisms of the affine space.
Concretely, we choose once and for all a coordinate system $(x_1, \dots, x_n)$ for $\A^n$.
Then any element $f \in\Aut(\A^n)$ is a map of the form
$$ f\colon(x_1, \dots, x_n) \mapsto (f_1(x_1, \dots, x_n), \dots, f_n(x_1, \dots, x_n)),$$
where the $f_i$ are polynomials in $n$ variables, such that there exists a map $g$ of the same form satisfying $f \circ g = \id$.
We shall abbreviate this situation by writing $f = (f_1, \dots, f_n)$, and $g = f^{-1}$.
Observe a slight abuse of notation here, since we are really
speaking about polynomials, and not about the associated polynomial
functions.
For instance over a finite base field, the group $\Aut(\A^n)$ is
infinite (for $n \ge 2$) even if there is only a finite number of induced bijections on the
finite number of $\K$-points of $\A^n$.

The group $\Aut(\A^n)$ contains the following natural subgroups.
First we have the affine group $A_n = \GL_n(\K) \ltimes \K^n$.
Secondly we have the group $E_n$ of elementary automorphisms, which have the form
$$ f\colon(x_1, \dots, x_n) \mapsto (x_1 + P(x_2, \dots, x_n), x_2, \dots, x_n),$$
for some choice of polynomial $P$ in $n-1$ variables.
The subgroup
$$\Tame(\A^n) = \langle A_n, E_n \rangle$$
generated by the affine and elementary automorphisms is called the subgroup of \textbf{tame automorphisms}.

A natural question is whether the inclusion $\Tame(\A^n) \subseteq \Aut(\A^n)$ is in fact an equality.
It is a well-known result, which goes back to Jung (see e.g. \cite{LJung} for a review of some of the many proofs available in the literature), that the answer is \textit{yes} for $n = 2$ (over any base field), and it is a result by Shestakov and Umirbaev \cite{SU:main} that the answer is \textit{no} for $n = 3$, at least when $\K$ is a field of characteristic zero.

The main purpose of the present paper is to give a self-contained reworked proof of this last result: see  Theorem \ref{thm:reducibility} and Corollary \ref{cor:nagata}.
We follow closely the line of argument by Kuroda \cite{Ku:main}.
However, the novelty in our approach is the emphasis on a 2-dimensional simplicial complex $\Comp$ on which $\Tame(\A^3)$ acts by isometries.
In fact, this construction is not particular to the 3-dimensional case: In
\S\ref{sec:simplicial complex} we introduce, for any $n \ge 2$ and over any base
field, a $(n-1)$-dimensional simplicial complex on which $\Tame(\A^n)$ naturally
acts.

We now give an outline of the main notions and results of the paper.
Since the paper is quite long and technical, we hope that this informal outline will serve as a guide for the reader, even if we cannot avoid being somewhat imprecise at this point.

\begin{itemize}[wide]
\item The 2-dimensional complex $\Comp$ contains three kinds of vertices, corresponding to three orbits under the action of $\TA$.
In this introduction we shall focus on the so-called type 3 vertices, which correspond to a tame automorphism 
$(f_1,f_2,f_3)$ up to post-composition by an affine automorphism.
Such an equivalence class is denoted $v_3 = [f_1, f_2, f_3]$ (see \S\ref{sec:Complex} and Figure \ref{fig:complex}).

\item Given a vertex $v_3$ one can always choose a representative $(f_1,f_2,f_3)$ such that the top monomials of the $f_i$ are pairwise distinct. Such a good representative is not unique, but the top monomials are. This allows to define a degree function (with values in $\N^3$) on vertices, with the identity automorphism corresponding to the unique vertex of minimal degree (see \S\ref{sec:degrees}). 

\item By construction of the complex $\Comp$, two type 3 vertices $v_3$ and $v_3'$ are at distance 2 in $\Comp$ if and only if they admit representatives of the form $v_3 = [f_1, f_2, f_3]$ and $v_3' = [f_1 + P(f_2, f_3), f_2, f_3]$, that is, representatives that differ by an elementary automorphism.
If moreover $\deg v_3 > \deg v_3'$, we say that $v_3'$ is an elementary reduction of $v_3$.
The whole idea is that it is `almost' true that any vertex admits a sequence of elementary reductions to the identity. However a lot of complications lie in this `almost', as we now discuss.

\item Some particular tame automorphisms are triangular automorphisms, of the form (up to a permutation of the variables) $(x_1 + P(x_2, x_3), x_2 + Q(x_3), x_3)$. There are essentially two ways to decompose such an automorphism as a product of elementary automorphisms, namely
\begin{align*}
(x_1 + P(x_2, x_3), x_2 &+ Q(x_3), x_3) \\
&= (x_1, x_2 + Q(x_3), x_3 ) \circ (x_1 + P(x_2, x_3), x_2, x_3)\\
&= (x_1 + P(x_2-Q(x_3), x_3), x_2, x_3) \circ (x_1, x_2 + Q(x_3), x_3).
\end{align*}
This leads to the presence of squares in the complex $\Comp$, see Figure \ref{fig:square}.
Conversely, the fact that a given vertex admits two distinct elementary
reductions often leads to the presence of such a square, in particular if one of
the reductions corresponds to a polynomial that depends only on one variable
instead of two, like $Q(x_3)$ above. We call `simple' such a particular
elementary reduction.

\item When $v_3' = [f_1 + P(f_2, f_3), f_2, f_3]$ is an elementary reduction of $v_3=[f_1, f_2, f_3]$, we shall encounter the following situations (see Corollary \ref{cor:minimal or K}):
\begin{enumerate}[(i), wide]
\item The most common situation is when the top monomial of $f_1$ is the dominant one, that is, the largest among the top monomials of the $f_i$.

\item 
Another (non-exclusive) situation is when $P$ depends only on one variable.
As mentioned before this is typically the case when $v_3$ admits several elementary reductions.
This corresponds to the fact that the top monomial of a component is a power of another one, and we name `resonance' such a coincidence (see definition in \S\ref{sec:degrees}).

\item
Finally another situation is when $f_1$ does not realize the dominant monomial, but $f_2, f_3$ nevertheless satisfy a kind of minimality condition via looking at the degree of the 2-form $df_2 \wedge df_3$.
We call this last case an elementary $K$-reduction (see \S\ref{sec:K} for the definition, Corollary \ref{cor:center of a K red} for the characterization in terms of the minimality of $\deg df_2 \wedge df_3$, and \S\ref{sec:examples} for examples).
This case is quite rigid (see Proposition \ref{pro:K stability}), and \textit{at posteriori} it forbids the
existence of any other elementary reduction from $v_3$ (see Proposition \ref{pro:no K}).
\end{enumerate}

\item Finally we define (see \S\ref{sec:K} again) an exceptional case, under the terminology  `normal proper $K$-reduction', that corresponds to moving from a vertex $v_3$ to a neighbor vertex $w_3$ of the same degree, and then realizing an elementary $K$-reduction from $w_3$ to another vertex $u_3$.
The fact that $v_3$ and $w_3$ share the same degree is not part of the technical definition, but again is true only \textit{at posteriori} (see Corollary \ref{cor:v3 w3 same degree}).

\item Then the main result (Reducibility Theorem \ref{thm:reducibility}) is that we can go from any vertex to the vertex corresponding to the identify by a finite sequence of elementary reductions or normal proper $K$-reductions.

\item The proof proceeds by a double induction on degrees which is quite involved.
The Induction Hypothesis is precisely stated on page \pageref{IH}.
The analysis is divided into two main branches:
\begin{enumerate}[(i), wide]
\item \S\ref{sec:SP} where the slogan is ``a vertex that admits an elementary $K$-reduction does not admit any other reduction'' (Proposition \ref{pro:slide} is an intermediate technical statement, and Proposition \ref{pro:no K} the final one);

\item \S\ref{sec:4C} where the slogan is ``a vertex that admits several
elementary reductions must admit some resonance'' (see in particular Lemmas \ref{lem:outer resonance} and \ref{lem:easy}).
\end{enumerate}
\end{itemize}

For readers familiar with previous works on the subject, we now give a few word about terminology. 
In the work of Kuroda \cite{Ku:main}, as in the original work of Shestakov and Umirbaev \cite{SU:main}, elementary reductions are defined with respect to one of the three coordinates of a fixed coordinate system.
In contrast, as explained above, we always work up to an affine change of
coordinates.
Indeed our simplicial complex $\Comp$ is designed so that two tame
automorphisms correspond to two vertices at distance 2 in the complex if and
only if they differ by the left composition of an automorphism of the form $a_1 e
a_2$, where $a_1, a_2$ are affine and $e$ is elementary.
This slight generalization of the definition of reduction allows us to absorb the so-called ``type I'' and ``type II'' reductions of Shestakov and Umirbaev in the class of elementary reductions: In our terminology they become ``elementary $K$-reductions'' (see \S \ref{sec:K}).
On the other hand, the ``type III'' reductions, which are technically difficult to handle, are still lurking around.
One can suspect that such reductions do not exist (as the most intricate ``type IV'' reductions which were excluded by Kuroda \cite{Ku:main}), and an ideal proof would settle this issue.
Unfortunately we were not able to do so, and these hypothetical reductions still appear in our text under the name of ``normal proper $K$-reduction''.
See Example \ref{nonexample} for more comments on this issue.

One could say that the theory of Shestakov, Umirbaev and Kuroda consists in understanding the relations inside the tame group $\TA$.
This was made explicit by Umirbaev \cite{U}, and then it was proved by Wright
\cite{wright} that this can be rephrased in terms of an amalgamated product
structure over three subgroups (see Corollary \ref{cor:product}).
In turn, it is known that such a structure is equivalent to the action of the group on a 2-dimensional simply connected simplicial complex, with fundamental domain a simplex.
Our approach allows to recover a more transparent description of the relations in $\TA$.
After stating and proving the Reducibility Theorem \ref{thm:reducibility} in \S\ref{sec:reducibility theorem}, we directly show in \S\ref{sec:1-connected} that the natural complex on which $\TA$ acts is simply connected (see Proposition \ref{pro:1connected}), by observing that the reduction process of  \cite{Ku:main, SU:main} corresponds to local homotopies.

We should stress once more that this paper contains no original result, and consists only in a new presentation of previous works by the above cited authors.
In fact, for the sake of completeness we also include in Section \ref{sec:prelim} some preliminary results where we only slightly differ from the original articles \cite{Ku:ineq, SU:ineq}.

Our motivation for reworking this material is to prepare the way for new results about $\TA$, such as the linearizability of finite subgroups, the Tits alternative or the acylindrical hyperbolicity.
From our experience in related settings (see \cite{BFL, CL, Lonjou, Martin}), such results should follow from some non-positive curvature properties of the simplicial complex.
We plan to explore these questions in some follow-up papers (see \cite{LP} for a first step).

\addtocontents{toc}{\SkipTocEntry}
\section*{Acknowledgement}
In August 2010 I was assigned the paper \cite{Ku:main} as a reviewer for
MathSciNet.
I took this opportunity to understand in full detail his beautiful proof, and
in November of the same year, at the invitation of Adrien Dubouloz, I gave some lectures in Dijon on the subject.
I thank all the participants of this study group ``Automorphismes des Espaces
Affines'', and  particularly Jérémy Blanc, Eric Edo, Pierre-Marie
Poloni and Stéphane Vénéreau who gave me useful feedback on some early version of these notes.

A few years later I embarked on the project to rewrite this proof focusing on
the related action on a simplicial complex.
At the final stage of the present version, when I was somehow loosing
faith in my endurance, I greatly benefited from the encouragements and numerous
 suggestions of Jean-Philippe Furter.

Finally, I am of course very much indebted to Shigeru Kuroda.
Indeed without his
work the initial proof of \cite{SU:main} would have
remained a mysterious black box to me.

\section{Simplicial complex} \label{sec:simplicial complex}

We define a $(n-1)$-dimensional simplicial complex on which the  tame
automorphism group of $\A^n$ acts.
This construction makes sense in any dimension $n \ge 2$, over any base
field $\K$.

\subsection{General construction}

For any $1 \le r \le n$, we call \textbf{$r$-tuple of components} a morphism
\begin{align*}
f\colon\A^n &\to \A^r \\
x = (x_1, \dots, x_n) &\mapsto \left( f_1(x), \dots, f_r(x) \right)
\end{align*}
that can be extended as a tame automorphism $f = (f_1,\dots, f_n)$ of $\A^n$.
One defines $n$ distinct types of vertices, by considering $r$-tuple of
components modulo composition by an affine automorphism on the range, $r = 1,
\dots, n$.
We use a bracket notation to denote such an equivalence class:
\begin{equation*}
[f_1, \dots,f_r] := A_r (f_1, \dots, f_r) = \{ a \circ (f_1, \dots, f_r) ; a \in A_r\}
\end{equation*}
where $A_r = \GL_r(\K) \ltimes \K^r$ is the $r$-dimensional affine group.
We say that $v_r = [f_1, \dots,f_r]$ is a vertex of \textbf{type} $\boldsymbol r$, and that $(f_1, \dots, f_r)$ is a representative of $v_r$.
We shall always stick to the convention that the index corresponds to the type of a vertex: for instance $v_r, v'_r, u_r, w_r, m_r$ will all be possible notation for a vertex of type $r$.

Now given $n$ vertices $v_1, \dots, v_n$ of type $1, \dots, n$, we attach a standard Euclidean $(n-1)$-simplex on these vertices if there exists a tame automorphism $(f_1,\dots, f_n) \in \Tame(\A^n)$ such that for all $ i \in \{1, \dots, n\}$:
$$v_i = [f_1, \dots, f_i ].$$
We obtain a $(n-1)$-dimensional simplicial complex $\Comp_n$ on which the tame group acts by isometries, by the formulas
$$g\cdot [f_1,\dots,f_r] := [f_1 \circ g^{-1},\dots, f_r \circ g^{-1}].$$

\begin{lemma}
The group $\Tame(\A^n)$ acts on $\Comp_n$ with fundamental domain the simplex
$$[x_1],\, [x_1, x_2],\, \dots,\, [x_1, \dots, x_n].$$
In particular the action is transitive on vertices of a given type.
\end{lemma}

\begin{proof}
Let $v_1, \dots, v_n$ be the vertices of a simplex (recall that the index corresponds to the type).
By definition there exists $f =(f_1, \dots f_n) \in \Tame(\A^n)$ such that $v_i = [f_1, \dots, f_i]$ for each $i$.
Then
\begin{align*}
[x_1, \dots, x_i] &= [(f_1, \dots, f_i ) \circ f^{-1}] = f\cdot v_i. \qedhere
\end{align*}
\end{proof}

\begin{remark}
\begin{enumerate}[wide]
\item
One could make a similar construction by working with the full automorphism group $\Aut(\A^n)$ instead of the tame group.
The complex $\Comp_n$ we consider is the gallery connected component of the standard simplex
$[x_1]$, $[x_1, x_2]$, $\dots,\, [x_1, \dots, x_n]$
in this bigger complex.
See \cite[\S 6.2.1]{BFL} for more details.
\item
When $n = 2$, the previous construction yields a graph $\Comp_2$.
It is not difficult to show (see \cite[\S 2.5.2]{BFL}) that $\Comp_2$ is isomorphic to the classical  Bass-Serre tree of $\Aut(\A^2) = \Tame(\A^2)$.
\end{enumerate}
\end{remark}

\subsection{Degrees} \label{sec:degrees}

We shall compare polynomials in $\K[x_1, \dots,x_n]$ by using the graded lexicographic order on monomials.
We find it more convenient to work with an additive notation, so we introduce the \textbf{degree} function, with value in $\N^n \cup \{-\infty\}$, by taking
$$\deg x_1^{a_1} x_2^{a_2} \dots x_n^{a_n} = (a_1, a_2, \dots, a_n)$$
and by convention $\deg 0 = -\infty$.
We extend this order to $\Q^n \cup \{-\infty\}$, since sometimes it is convenient to consider difference of degrees, or degrees multiplied by a rational number.
The \textbf{top term} of $g \in \K[x_1,\dots,x_n]$ is the uniquely defined $\bar g = c x_1^{d_1}\dots x_n^{d_n}$ such that
$$(d_1,\dots,d_n) = \deg g  > \deg (g - \bar g).$$
Observe that two polynomials $f,g \in \K[x_1,\dots,x_n]$ have the same degree if and only if their top terms $\bar f, \bar g$ are proportional. 
If $f = (f_1, \dots, f_r)$ is a $r$-tuple of components,
we call \textbf{top degree} of $f$ the maximum of the degree of the $f_i$:
$$\topdeg f := \max \deg f_i \in \N^n.$$

\begin{lemma} \label{lem:stratified degree}
Let $f = (f_1, \dots, f_r)$
be a $r$-tuple of components, and consider $V \subset \K[x_1, \dots, x_n]$ the vector space generated by the $f_i$.
Then
\begin{enumerate}
\item The set $H$ of elements $g \in V$ satisfying $\topdeg f > \deg g$ is a hyperplane in $V$;
\item There exist a sequence of degrees $\delta_r > \dots > \delta_1$ and a flag of subspaces $V_1 \subset \dots \subset V_r = V$ such that $\dim V_i  = i$ and $\deg g = \delta_i$ for any $g \in V_i \smallsetminus V_{i-1}$.
\end{enumerate}
\end{lemma}

\begin{proof}
\begin{enumerate}[wide]
\item
Up to permuting the $f_i$ we can assume $\topdeg f = \deg f_r$.
Then for each $i = 1, \dots, r-1$ there exists a unique $c_i \in \K$ such that
$\deg f_r > \deg (f_i + c_i f_r)$.
The conclusion follows from the observation that an element of $V$ is in $H$ if and only if it is a linear combination of the $f_i + c_i f_r$, $i = 1, \dots, r-1$.
\item
Immediate, by induction on dimension. \qedhere
\end{enumerate}
\end{proof}

Using the notation of the lemma, we call $r$-$\deg f = (\delta_1, \dots, \delta_r)$ the $\boldsymbol{r}$-\textbf{degree} of $f$, and $\deg f = \sum_{i = 1}^r \delta_i \in \N^n$ the \textbf{degree} of $f$.
Observe that for any affine automorphism $a \in A_r$ we have $r\text{-}\deg f =
r\text{-}\deg (a\circ f)$, so we get a well-defined notion of $r$-degree and
degree for any vertex of type $r$.

If $v_r = [f_1,\dots,f_r] \in \Comp_n$ with
the $\deg f_i$ pairwise distinct, we say that $f$ is a
\textbf{good representative} of $v_r$ (we do not ask $\deg f_r > \dots > \deg
f_2 > \deg f_1$).
We use a double bracket notation such as $v_2 = \llb
f_1,f_2 \rrb$ or $v_3 = \llb f_1,f_2,f_3 \rrb$, to indicate that we are using a
good representative.

\begin{lemma} \label{lem:rep of a simplex}
Let $v_1, \dots, v_n$ be a $(n-1)$-simplex in $\Comp_n$.
Then there exists $f = (f_1, \dots, f_n) \in \Tame(\A^n)$ such that
$v_i = \llb f_1, \dots, f_i \rrb$ for each $n  \ge i \ge 1$.
\end{lemma}

\begin{proof}
We pick $f_1$ such that $v_1= \llb f_1 \rrb$, and we define the other $f_i$ by
induction as follows.
If the $i$-degree of $v_i = \llb f_1, \dots, f_i \rrb$ is $(\delta_1, \dots,
\delta_i)$ (recall that by definition the $\delta_j$ are equal to the degrees of
the $f_j$ only up to a permutation), then there exist $\delta \in \N^3$ and
$i+1 \ge s \ge 1$ such that the $(i+1)$-degree of $v_{i+1}$ is $(\delta_1,
\dots, \delta_{s-1}, \delta, \delta_{s}, \dots, \delta_i)$.
That exactly means that there exists $f_{i+1}$ of degree $\delta$ such that
$v_{i+1} = \llb f_1, \dots, f_{i+1} \rrb$.
\end{proof}

\subsection{The complex in dimension 3} \label{sec:Complex}

Now we specialize the general construction to the dimension $n= 3$, which is our main interest in this paper.
We drop the index and simply denote by $\Comp$ the 2-dimensional simplicial complex associated to $\Tame(\A^3)$.
To get a first feeling of the complex one can draw pictures such as Figure \ref{fig:complex}, where we use the following convention for vertices: $\typeone$, $\typetwo$ or $\typethree$ corresponds respectively to a vertex of type 1, 2 and 3.
However one should keep in mind, as the following formal discussion makes it clear, that the complex is not locally finite.
A first step in understanding the geometry of the complex $\Comp$ is to understand the link of each type of vertex.
In fact, we will now see that if the base field $\K$ is uncountable, then the link of any vertex or any edge also has uncountably many vertices.

\begin{figure}[t]
$$
\mygraph{
!{<0cm,0cm>;<1.7cm,0cm>:<0cm,1.5cm>::}
!{(-2,0)}*{\typethree}="Q"
!{(0,2)}*{\typethree}="id"
!{(0,-2)}*{\typethree}="v3''"
!{(2,0)}*{\typethree}="u3"
!{(0,0)}*{\typeone}="x2"
!{(-1,1)}*{\typetwo}="x2x3"
!{(-1,-1)}*-{\typetwo}="v2''"
!{(1,1)}*{\typetwo}="u2'"
!{(1,-1)}*-{\typetwo}="x1x2"
!{(-1.5,3)}*{\typeone}="x3"
!{(1.5,3)}*{\typeone}="x1"
!{(0,3.5)}*{\typetwo}="x1x3"
!{(1.8,-.6)}*{\typethree}="behind3"
!{(.4,-1)}*{\typetwo}="behind2"
!{(-.8,3.5)}*{\typeone}="behind1"
"id"-_<(.1){[x_1, x_2, x_3]\,}_>{[x_2,x_3]\;\;\;}"x2x3"-_>{[x_1 + Q(x_2),x_2,x_3]\quad}"Q"-_>{[x_1 + Q(x_2),x_2]}"v2''"-_>{[x_1 + Q(x_2),x_2, x_3+P(x_1,x_2)]\quad}"v3''"-"x1x2"-_<{[x_2, x_3+P(x_1,x_2)]}"u3"-_<{\quad[x_1,x_2, x_3+P(x_1,x_2)]}"u2'"-_<{\;\;\;[x_1,x_2]}"id"
"x1x2"-"x2"-"v2''"
"v3''"-"x2"-"id"
"u3"-_>(.81){[x_2]}"x2"-"Q"
"x2x3"-"x2"-"u2'"
"x2x3"-"x3"-^<{[x_3]}^>(1.1){[x_1,x_3]}"x1x3"-^>{[x_1]}"x1"-"u2'"
"id"-"x3" "id"-"x1x3" "id"-"x1"
"Q"-"x3" "u3"-"x1"
"Q"-@{.}_>{[x_1+x_3+Q(x_2),x_2]}"behind2"-@{.}_>(1.45){\quad[x_1,x_2,x_3+Q(x_2)]}"behind3"-@{.}"u2'"
"behind2"-@{.}"x2"-@{.}"behind3"
"x1x3"-_>{[x_1 +x_3]\quad}"behind1"-@{.}"id"
"behind3"-@{.}"x1"
}
$$
\caption{A few simplexes of the complex $\Comp$.}\label{fig:complex}
\end{figure}

Consider first the link $\LL(v_3)$ of a vertex of type $3$.
By transitivity of the action of $\Tame(\A^3)$, it is sufficient to describe the link $\LL([\id])$.
A vertex of type 1 at distance 1 from $[\id]$ has the form $[a_1x_1 + a_2x_2 + a_3x_3]$ where the $a_i \in \K$ are uniquely defined up to a common multiplicative constant.
In other words, vertices of type 1 in $\LL([\id])$ are parametrized by $\p^2$.
We denote by $\p^2(v_3)$ this projective plane of vertices of type $1$ in the link of $v_3$.
Similarly, vertices of type 2 in $\LL(v_3)$ correspond to lines in $\p^2(v_3)$, that is, to points in the dual projective space $\linesd{v_3}$.
The edges in $\LL(v_3)$ correspond to incidence relations (``a point belongs to a line'').
We will often refer to a vertex of type 2 as a ``line in $\lines{v_3}$''.
In the same vein, we will sometimes refer to a vertex of type 1 as being ``the intersection of two lines in $\lines{v_3}$'', or we will express the fact that $v_1$ and $v_2$ are joined by an edge in $\Comp$ by saying ``the line $v_2$ passes through $v_1$''.

Now we turn to the description of the link of a vertex $v_2$ of type $2$.
By transitivity we can assume $v_2 = [x_1,x_2]$, and one checks that vertices of type 1 in $\LL(v_2)$ are parametrized by $\p^1$ and are of the form
$$[a_1 x_1 + a_2 x_2],\; (a_1 : a_2) \in \p^1.$$
On the other hand vertices  of type 3 in $\LL(v_2)$ are of the form
$$[x_1, x_2, x_3 + P(x_1,x_2)],\; P \in \K[y,z].$$
Precisely by taking the $P$ without constant or linear part we obtain a complete
set of representatives for such vertices of type 3 in $\LL(v_2)$.
Using the transitivity of the action of $\TA$ on vertices of type 2, the following lemma and its corollary are then immediate:

\begin{lemma}
The link $\LL(v_2)$ of a vertex of type $2$ is the complete bipartite graph between vertices of type $1$ and $3$ in the link.
\end{lemma}

\begin{corollary} \label{cor:type 3 neighbors}
Let $v_2 =[f_1,f_2]$ and $v_3 = [f_1,f_2,f_3]$ be vertices of type 2 and 3.
Then any vertex $u_3$ distinct from $v_3$ such that $v_2 \in \linesd{u_3}$ has
the form
$$u_3  =[f_1,f_2, f_3 + P(f_1,f_2)]$$
where $P \in \K[y,z]$ is a non-affine polynomial in two variables (that is,
not of the form $P(y,z) = ay + bz + c$).
In particular, $v_2$ is the unique type 2 vertex in $\linesd{v_3} \cap \linesd{u_3}$.
\end{corollary}

The link of a vertex of type 1 is more complicated.
Let us simply mention without proof, since we won't need it in this paper (but
see Lemma \ref{lem:link1connected} for a partial result, and also
\cite[\S3]{LP}), that in contrast with the case of vertices of type 2 or 3, the
link of a vertex of type 1 is a connected unbounded graph, which admits a
projection to an unbounded tree.

\section{Parachute Inequality and Principle of Two Maxima} \label{sec:prelim}

We recall here two results from \cite{Ku:ineq} (in turn they were adaptations from \cite{SU:ineq,SU:main}).
The Parachute Inequality is the most important; we also recall some direct consequences.
From now on $\K$ denotes a field of characteristic zero.

\subsection{Degree of polynomials and forms} \label{sec:degree polynomials}

Recall that we define a \textbf{degree} function on $\K[x_1,x_2,x_3]$ with value in $\N^3 \cup \{-\infty\}$ by taking
$\deg x_1^{a_1} x_2^{a_2} x_3^{a_3} = (a_1, a_2, a_3)$
and by convention $\deg 0 = -\infty$.
We compare degrees using the graded lexicographic order.

We now introduce the notion of \textbf{virtual degree} in two distinct situations, which should be clear by context.

Let $g \in \K[x_1,x_2,x_3]$, and $\phi = \sum_{i \in I} P_i y^i \in  \K[x_1,x_2,x_3][y]$ where $P_i \neq 0$ for all $i \in I$, that is, $I$ is the support of $\phi$.
We define the virtual degree of $\phi$ with respect to $g$ as
$$\dvirt \phi(g) := \max_{i \in I} (\deg P_i g^i) = \max_{i \in I} (\deg P_i + i \deg g).$$
Denoting by $\bar I \subseteq I$ the subset of indexes $i$ that realize the maximum, we also define the \textbf{top component} of $\phi$ with respect to $g$ as
$$\topcomp{\phi}{g} := \sum_{i \in \bar I} \bar P_i y^i.$$

Similarly if $g,h \in \K[x_1,x_2,x_3]$, and $\phi = \sum_{(i,j) \in S} c_{i,j} y^i z^j \in \K[y,z]$ with support $S$, we define the virtual degree of $\phi$ with respect to $g$ and $h$ as
$$\dvirt \phi(g,h) := \max_{(i,j) \in S} \deg g^ih^j = \max_{(i,j) \in S} (i \deg g + j \deg h).$$

Observe that $\phi(g,h)$ can be seen either as an element coming from 
$\phi_h(y) := \phi(y,h) \in \K[h][y] \subset \K[x_1,x_2,x_3][y]$ or from $\phi(y,z) \in \K[y,z]$, and that the two possible notions of virtual degree coincide:
$$\dvirt \phi_h(g) =\dvirt \phi(g,h).$$

\begin{example}
In general we have $\dvirt \phi(g) \ge \deg \phi(g)$ and $\dvirt \phi(g,h) \ge \deg \phi(g,h)$.
We now give two simple examples where these inequalities are strict.
\begin{enumerate}
\item Let $\phi = x_3^2 y - x_3 y^2$, and $g= x_3$.
Then $\phi(g) = 0$, but
$$\dvirt \phi(g) = \deg x_3^3 = (0,0,3).$$
\item Let $\phi = y^2 - z^3$, and $g = x_1^3$, $h = x_1^2$.
Then $\phi(g,h) = 0$, but
$$\dvirt \phi(g,h) = \deg x_1^6 = (6,0,0).$$
\end{enumerate}
\end{example}

We extend the notion of degree to algebraic differential forms.
Given
$$\omega = \sum f_{i_1,\cdots,i_k} dx_{i_1} \wedge \cdots \wedge dx_{i_k}$$
where $k = 1,2$ or $3$ and $f_{i_1,\cdots,i_k} \in \K[x_1,x_2,x_3]$,
we define
$$\deg \omega := \max \{ \deg f_{i_1,\cdots,i_k} x_{i_1} \cdots x_{i_k}\} \in \N^3 \cup \{-\infty \}.$$

We gather some immediate remarks for future reference (observe that here we use the assumption $\car  \K = 0$).

\begin{lemma} \label{lem:basics forms}
If $\omega, \omega'$ are forms, and $g$ is a non constant polynomial, we have
\begin{align*}
\deg \omega + \deg \omega' &\ge \deg \omega \wedge \omega' ; \\
\deg g &= \deg dg; \\
\deg g \omega &= \deg g + \deg \omega.
\end{align*}
\end{lemma}

\subsection{Parachute Inequality} \label{sec:parachute}

If $\phi \in \K[x_1, x_2, x_3][y]$, we denote by $\phi^{(n)} \in \K[x_1, x_2,
x_3][y] $ the $n$th derivative of $\phi$ with respect to $y$.
We simply write $\phi'$ instead of $\phi^{(1)}$.

\begin{lemma} \label{lem:pre multroot}
Let $\phi \in \K[x_1, x_2, x_3][y]$ and $g \in \K[x_1,x_2,x_3]$.
Then:
\begin{enumerate}
\item \label{pre multroot:1} If $\deg_y \topcomp{\phi}{g} \ge 1$, then $\dvirt
\phi'(g) = \dvirt \phi(g) - \deg g$.

\item \label{pre multroot:2} If $\deg_y \topcomp{\phi}{g} \ge j \ge 1$,
then $\topcomp{\phi^{(j)}}{g} = (\topcomp{\phi}{g})^{(j)}.$
\end{enumerate}
\end{lemma}

\begin{proof}
We note as before
\begin{equation*}
\phi = \sum_{i \in I} P_i y^i,
\qquad \topcomp{\phi}{g} = \sum_{i \in \bar I} \bar P_i y^i
\qquad \text{and} \qquad \phi' = \sum_{i \in I \smallsetminus \{0\}} iP_i y^{i-1},
\end{equation*}
where $I$ is the support of $\phi$, and $\bar I \subseteq I$ is the subset of
indexes $i$ that realize the maximum $\max_{i \in I} (\deg P_i + i \deg g).$

Now if $\deg_y \topcomp{\phi}{g} \ge 1$, that is, if $\bar I \neq \{0\}$, then
the indexes in $\bar I \smallsetminus \{0\}$ are precisely those that realize
the maximum
$\max_{i \in I \smallsetminus \{0\}} (\deg P_i + (i-1) \deg g).$
Thus we get assertion \ref{pre multroot:1},
and $\topcomp{\phi'}{g} = (\topcomp{\phi}{g})'.$
Assertion \ref{pre multroot:2} for $j \ge 2$ follows by induction.
\end{proof}

\begin{lemma} \label{lem:multroot}
Let $\phi \in \K[x_1, x_2, x_3][y]$ and $g \in \K[x_1,x_2,x_3]$.
Then, for $m \ge 0$, the following two assertions are equivalent:
\begin{enumerate}
\item \label{multroot:1} For $j = 0, \dots, m-1$ we have $\dvirt \phi^{(j)}(g)
> \deg \phi^{(j)}(g)$, but
$$\dvirt \phi^{(m)}(g) = \deg \phi^{(m)}(g).$$
\item \label{multroot:2} There exists $\psi \in \K[x_1, x_2, x_3][y]$ such that $\psi(\bar g) \neq 0$ and
$$\topcomp{\phi}{g} = (y - \bar g)^m \cdot \psi.$$
\end{enumerate}
\end{lemma}

\begin{proof}
Observe that we have the equivalences
\begin{equation} \label{eq:2 equivalences}
\dvirt \phi(g) > \deg \phi(g) \;\Longleftrightarrow\; \topcomp{\phi}{g}(\bar
g) = 0 \;\Longleftrightarrow\; y-\bar g \text{ divides } \topcomp{\phi}{g}.
\end{equation}

First we prove $\ref{multroot:2} \Rightarrow \ref{multroot:1}$.
Assuming $\ref{multroot:2}$, by Lemma \ref{lem:pre multroot}\ref{pre
multroot:2} we have $\topcomp{\phi^{(j)}}{g} = (\topcomp{\phi}{g})^{(j)}$ for
$j = 0, \dots, m-1$.
The second equivalence in (\ref{eq:2 equivalences}) yields
$(\topcomp{\phi}{g})^{(j)}(g) = 0$ for $j = 0, \dots, m-1$, and
$(\topcomp{\phi}{g})^{(m)}(g) \neq 0$, and then the first equivalence gives the
result.

To prove $\ref{multroot:1} \Rightarrow \ref{multroot:2}$, it is sufficient to
show that if $\dvirt \phi^{(j)}(g) > \deg \phi^{(j)}(g)$ for $j = 0, \dots,
k-1$, then $\topcomp{\phi}{g} = (y - \bar g)^k \cdot \psi_k$ for some $\psi_k
\in \K[x_1, x_2, x_3][y]$.
The remark (\ref{eq:2 equivalences}) gives it for $k = 1$.
Moreover, by Lemma \ref{lem:pre multroot}\ref{pre multroot:2},  if
$\topcomp{\phi}{g}$ depends on $y$ then
$\topcomp{\phi'}{g} = (\topcomp{\phi}{g})'$, hence the result by induction.
\end{proof}

In the situation of Lemma \ref{lem:multroot}, we call the integer $m$ the \textbf{multiplicity} of $\phi$ with respect to $g$, and we denote it by $m(\phi,g)$.
In other words, the top term $\bar g$ is a multiple root of $\topcomp{\phi}{g}$ of order $m(\phi,g)$.

Following Vénéreau \cite{V}, where a similar inequality is proved, we call the next result a ``Parachute Inequality''.
Indeed its significance is that the real degree cannot drop too much with respect to the virtual degree.
However we follow Kuroda for the proof.

Recall that (over a field $\K$ of characteristic zero) some polynomials $f_1,
\cdots, f_r \in \K[x_1, \dots, x_n]$ are algebraically independent if and only
if
$df_1 \wedge \cdots \wedge df_r \neq 0$.
Indeed this is equivalent to asking that the map $f = (f_1, \dots, f_r)$ from
$\A^n$ to $\A^r$ is dominant, which in turn is equivalent to saying that the
differential of this map has maximal rank on an open set of $\A^n$
(for details see for instance \cite[Theorem III p.135]{HP}).

\begin{proposition}[Parachute Inequality, see {\cite[Theorem 2.1]{Ku:ineq}}] \label{pro:parachute}
Let $r = 2$ or $3$, and let $f_1,\cdots,f_r \in \K[x_1,x_2,x_3]$ be algebraically independent.
Let $\phi \in \K[f_2,\cdots, f_r][y] \smallsetminus \{0\}$.
Then
$$\deg \phi(f_1) \ge \dvirt \phi(f_1) - m(\phi,f_1) (\deg \omega + \deg f_1 - \deg df_1 \wedge \omega).$$
where $\omega =  df_2$ if $r= 2$, or $\omega = df_2 \wedge df_3$ if $r = 3$.
\end{proposition}

\begin{proof}
Denoting as before $\phi'$ the derivative of $\phi$ with respect to $y$, we have
$$d (\phi(f_1)) = \phi'(f_1) df_1 + \text{other terms involving }df_2 \text{ or } df_3.$$
So we obtain $  d( \phi(f_1)) \wedge \omega = \phi'(f_1)  df_1 \wedge \omega $.
Using Lemma \ref{lem:basics forms} this yields
\begin{multline*}
\deg \phi(f_1) + \deg \omega = \deg d( \phi(f_1)) + \deg \omega  \ge \deg   d( \phi(f_1)) \wedge \omega\\
= \deg  \phi'(f_1)   df_1 \wedge \omega = \deg \phi'(f_1) + \deg   df_1 \wedge \omega,
\end{multline*}
which we can write as
\begin{equation}
\label{eq:degphi}
- \deg     df_1 \wedge \omega + \deg \omega + \deg \phi(f_1) \ge \deg \phi'(f_1).
\end{equation}

Now we are ready to prove the inequality of the statement, by induction on $m(\phi,f_1)$.

If $m(\phi,f_1) = 0$, that is, if $\deg \phi(f_1) = \dvirt \phi(f_1)$, there is nothing to do.

If $m(\phi,f_1) \ge 1$, it follows from Lemma \ref{lem:multroot} that
$m(\phi',f_1) = m(\phi,f_1) - 1$.
Moreover, the condition $m(\phi,f_1) \ge 1$ implies that $\topcomp{\phi}{f_1}$
does depend on $y$, hence by Lemma \ref{lem:pre multroot}\ref{pre multroot:1}
we have
$$\dvirt \phi'(f_1) = \dvirt \phi(f_1) - \deg f_1.$$
By induction hypothesis, we have
\begin{align*}
\deg \phi'(f_1) &\ge \dvirt \phi'(f_1) - m(\phi',f_1) (\deg \omega + \deg f_1 - \deg df_1 \wedge \omega) \\
&= \dvirt \phi(f_1) - \deg f_1 - (m(\phi,f_1) - 1) (\deg \omega + \deg f_1
-\deg df_1 \wedge \omega ) \\
&= \dvirt \phi(f_1) - m(\phi,f_1)(\deg \omega + \deg f_1 -\deg df_1 \wedge \omega ) \\
& \qquad-\deg df_1 \wedge \omega + \deg \omega.
\end{align*}
Combining with (\ref{eq:degphi}), and canceling the terms $- \deg   df_1 \wedge \omega + \deg \omega$ on each side, one obtains the expected inequality.
\end{proof}

\subsection{Consequences}

We shall use the Parachute Inequality \ref{pro:parachute} mostly when $r = 2$, and when we have a strict inequality $\dvirt \phi(f_1,f_2) > \deg \phi(f_1,f_2)$.
In this context the following easy lemma is crucial.
Ultimately this is here that lies the difficulty when one tries to extend the theory in dimension 4 (or more!).

\begin{lemma}
\label{lem:p and q}
Let $f_1,f_2 \in \K[x_1,x_2,x_3]$ be algebraically independent, and $\phi \in \K[y,z]$ such that $\dvirt \phi(f_1,f_2) > \deg \phi(f_1,f_2)$.
Then:
\begin{enumerate}
\item \label{p and q:1} There exist coprime $p,q \in \N^*$such that
$$p \deg f_1 = q \deg f_2.$$
In particular, there exists $\delta \in \N^3$ such that $\deg f_1 = q \delta$,
$\deg f_2 = p\delta$, so that the top terms of $f_1^p$ and $f_2^q$ are equal
up to a constant: there exists $c \in \K$ such that $\bar f_1^p = c \bar
f_2^q$.

\item \label{p and q:2} Considering $\phi(f_1,f_2)$ as coming from $\phi(y,f_2) \in \K[x_1,x_2,x_3][y]$, we have
$$\topcomp{\phi}{f_1} 
= (y^p - \bar f_1^p)^{m(\phi,f_1)}\cdot \psi
= (y^p - c\bar f_2^q)^{m(\phi,f_1)}\cdot \psi$$
for some $\psi \in \K[x_1,x_2,x_3][y]$.
\end{enumerate}
\end{lemma}

\begin{proof}
\begin{enumerate}[wide]
\item
We write $\phi(f_1,f_2) = \sum c_{i,j}f_1^if_2^j$.
Since $\dvirt \phi(f_1,f_2) > \deg \phi(f_1,f_2)$, there exist distinct $(a,b)$ and $(a',b')$ such that
$$\deg f_1^a f_2^b = \deg f_1^{a'}f_2^{b'} = \dvirt \phi(f_1,f_2).$$
Moreover we can assume that $a, a'$ are respectively maximal and minimal for this property.
We obtain
$$(a-a')\deg f_1 = (b' - b) \deg f_2.$$
Dividing by $m$, the GCD of $a-a'$ and $b'-b$, we get the expected relation.

\item With the same notation, we have $a = a' + pm$ where $m \ge 1$, and in particular $\deg_y \phi(y,f_2) \ge p$.
So if $p > \deg_y P(y,f_2)$ for some $P \in \K[f_2][y]$, we have $\dvirt P(f_1,f_2) = \deg P(f_1,f_2)$.
By the first assertion, there exists $c \in \K$ such that $\deg f_1^p > \deg \left(f_1^p - cf_2^q\right)$.
By successive Euclidean divisions in $\K[f_2][y]$ we can write:
$$\phi(y,f_2) = \sum R_i(y) \left(y^p - c f_2^q\right)^i$$
with $p > \deg_y R_i$ for all $i$.
Denote by $I$ the subset of indexes such that
\begin{equation} \label{eq:barphi}
\topcomp{\phi}{f_1}  = \sum_{i \in I} \topcomp{R_i}{f_1} \left(y^p - c \bar f_2^q\right)^i.
\end{equation}
Let $i_0$ be the minimal index in $I$.
We want to prove that $i_0 \ge m(\phi,f_1)$.
By contradiction, assume that $m(\phi,f_1) > i_0$.
Since $y - \bar f_1$ is a simple factor of $(y^p - c \bar f_2^q) = (y^p - \bar f_1^p)$, and is not a factor of any $\topcomp{R_i}{f_1}$, we obtain that $(y - \bar f_1)^{i_0+1}$ divides all summands of (\ref{eq:barphi}) except $\topcomp{R_{i_0}}{f_1} (y^p - c \bar f_2^q)^{i_0}$.
In particular $(y - \bar f_1)^{i_0+1}$, hence also  $(y - \bar f_1)^{m(\phi,f_1)}$, do not divide $\topcomp{\phi}{f_1}$: This is a contradiction with Lemma \ref{lem:multroot}.
\qedhere
\end{enumerate}
\end{proof}

We now list some consequences of the Parachute Inequality \ref{pro:parachute}.

\begin{corollary}\label{cor:parachute}
Let $f_1, f_2 \in \K[x_1,x_2,x_3]$ be algebraically independent with $ \deg f_1 \ge \deg f_2$, and $\phi \in \K[y,z]$ such that $\dvirt \phi(f_1,f_2) > \deg \phi(f_1,f_2)$.
Following Lemma $\ref{lem:p and q}$, we write $p \deg f_1 = q \deg f_2$ where $p,q \in \N^*$ are coprime.
Then:
\begin{enumerate}[wide]
\item \label{parachute-i} $\deg \phi(f_1,f_2) \ge p\deg f_1 - \deg f_1 - \deg f_2 + \deg df_1\wedge df_2 $;
\item \label{parachute-ii} If $\deg f_1 \not\in \N \deg f_2$, then $\deg \phi(f_1,f_2) > \deg df_1 \wedge df_2$;
\item \label{parachute-iii} Assume $\deg f_1 \not\in \N \deg f_2$ and $\deg f_1 \ge \deg \phi(f_1,f_2)$.
Then $p = 2$, $q \ge 3$ is odd, and
$$\deg \phi(f_1,f_2) \ge \deg f_1 - \deg f_2 + \deg df_1 \wedge df_2.$$
If moreover $\deg f_2 \ge \deg \phi(f_1,f_2)$, then $q=3$.
\item \label{parachute-iv} Assume $\deg f_1 \not\in \N \deg f_2$ and $\deg f_1 \ge \deg \phi(f_1,f_2)$.
Then
$$\deg d(\phi(f_1,f_2)) \wedge df_2  \ge \deg f_1 + \deg df_1 \wedge df_2.$$
\end{enumerate}
\end{corollary}

\begin{proof}
\begin{enumerate}[wide]
\item By Lemma \ref{lem:p and q}\ref{p and q:2}, we have
\begin{equation}
\label{eq:virtphi}
\dvirt \phi(f_1,f_2) \ge  m(\phi,f_1) p \deg f_1.
\end{equation}
On the other hand the Parachute Inequality \ref{pro:parachute} applied to $\phi(y, f_2) \in \K[f_2][y]$ yields
$$\deg \phi(f_1,f_2) \ge \dvirt \phi(f_1,f_2) - m(\phi,f_1)(\deg f_1 + \deg f_2 -\deg df_1\wedge df_2).$$
Combining with (\ref{eq:virtphi}), and remembering that $m(\phi,f_1) \ge 1$, we obtain
$$\deg \phi(f_1,f_2) \ge \frac{\deg \phi(f_1,f_2)}{m(\phi,f_1)}
\ge p \deg f_1 - \deg f_1 - \deg f_2 + \deg df_1\wedge df_2.$$

\item From Lemma \ref{lem:p and q} we have $\deg f_1 = q\delta$  and  $\deg f_2 = p\delta$ for some $\delta \in \N^3$.
The inequality \ref{parachute-i} gives
\begin{multline} \label{eq:pq}
\deg \phi(f_1,f_2) \ge p\deg f_1 - \deg f_1 - \deg f_2 + \deg df_1 \wedge df_2 = \\
(pq - p - q)\delta + \deg df_1 \wedge df_2.
\end{multline}
The assumption $\deg f_1 \not\in \N \deg f_2$ implies $q > p \ge 2$.
Thus $pq - p - q >0$, and finally
$$\deg \phi(f_1,f_2) > \deg df_1 \wedge df_2.$$

\item Again the assumptions imply $q > p \ge 2$.
Since $q\delta = \deg f_1 \ge \deg \phi(f_1,f_2)$, we get from (\ref{eq:pq}) that $q > pq - p - q$.
This is only possible  if $p = 2$, and so $q \ge 3$ is odd.
Replacing $p$ by $2$ in (\ref{eq:pq}), we get the inequality.

If $\deg f_2 \ge \deg \phi(f_1,f_2)$, we obtain $2\delta > (q-2)\delta$, hence $q = 3$.\\

\item Denote $\phi(y,z) = \sum c_{i,j}y^i z^j$, and consider the partial derivatives
\begin{align*}
\phi_y'(y,z) &= \sum i c_{i,j} y^{i-1}z^j; \\
\phi_z'(y,z) &= \sum j c_{i,j} y^{i}z^{j-1}.
\end{align*}
We have $d (\phi(f_1,f_2)) = \phi_y'(f_1,f_2) df_1 + \phi_z'(f_1,f_2) df_2$.
In particular $d (\phi(f_1,f_2)) \wedge df_2 =  \phi_y'(f_1,f_2) df_1 \wedge
df_2$, and
$$\deg d (\phi(f_1,f_2)) \wedge df_2 = \deg \phi_y'(f_1,f_2) + \deg df_1 \wedge
df_2.$$
Now we consider $\phi_y'(f_1,f_2)$ as coming from $\phi_y'(y,f_2) \in
\K[f_2][y]$, and we simply write $\phi'(f_1)$ instead of $\phi_y'(f_1,f_2)$, in
accordance with the convention for derivatives introduced at the beginning of
\S\ref{sec:parachute}.
We want to show $ \deg \phi'(f_1) \ge \deg f_1$.
Recall that by (\ref{eq:virtphi}), $\dvirt \phi(f_1) \ge 2 m(\phi,f_1)\deg f_1$,
and so, using also Lemma \ref{lem:pre multroot}\ref{pre multroot:1}:
$$\dvirt \phi'(f_1) = \dvirt \phi(f_1) - \deg f_1 \ge 2 (m(\phi,f_1) - 1)\deg f_1 + \deg f_1.$$
The Parachute Inequality \ref{pro:parachute} then gives (for the last
inequality recall that $\deg f_1 \ge \deg f_2$ by assumption):
\begin{align*}
\deg \phi'(f_1) &\ge \dvirt \phi'(f_1) - m(\phi',f_1)(\deg f_1 + \deg f_2 -\deg df_1\wedge df_2) \\
 &\ge  2(m(\phi,f_1) - 1)\deg f_1 + \deg f_1 \\
 & \hspace{3cm} - (m(\phi,f_1) - 1)(\deg f_1 + \deg f_2 -\deg df_1\wedge df_2) \\
 &= (m(\phi,f_1) - 1)(\deg f_1 - \deg f_2 +\deg df_1\wedge df_2) + \deg f_1\\
 &\ge \deg f_1 .\qedhere
\end{align*}
\end{enumerate}
\end{proof}

\begin{corollary}\label{cor:parachute h + phi}
Let $f_1,f_2,f_3 \in \K[x_1,x_2,x_3]$ be algebraically independent, and $\phi \in \K[y,z]$ such that
\begin{align*}
\dvirt \phi(f_1,f_2) &> \deg \phi(f_1,f_2),\\
\dvirt \phi(f_1,f_2) &> \deg f_3.
\end{align*}
Following Lemma $\ref{lem:p and q}$, we write $p \deg f_1 = q \deg f_2$ where $p,q \in \N^*$ are coprime.
Then
$$\deg (f_3 + \phi(f_1,f_2))  > p \deg f_1  - \deg df_2\wedge df_3 - \deg f_1,$$
\end{corollary}

\begin{proof}
The Parachute Inequality \ref{pro:parachute} applied to $\psi = f_3 + \phi(y,f_2) \in \K[f_2,f_3][y]$ gives
\begin{multline}
\label{eq:psi}
\deg (f_3 + \phi(f_1,f_2)) \ge \dvirt \psi(f_1) \\
- m(\psi,f_1)(\deg df_2\wedge df_3 + \deg f_1 - \deg df_1\wedge df_2\wedge df_3 ).
\end{multline}
By assumption $\dvirt\phi(f_1,f_2) > \deg f_3$.
Thus not only $\dvirt \psi(f_1) = \dvirt \phi(f_1)$, but also 
$\topcomp{\psi}{f_1} = \topcomp{\phi}{f_1}$, hence $m(\psi,{f_1}) = m(\phi,{f_1}) \ge 1$.
By Lemma \ref{lem:p and q}\ref{p and q:2}, we obtain
$$\dvirt \psi(f_1) \ge m(\phi,f_1) p \deg f_1 = m(\psi,f_1) p \deg f_1$$
Replacing in (\ref{eq:psi}), and dividing by $m(\psi,f_1)$, we get the result.
\end{proof}

\subsection{Principle of Two Maxima}

The proof of the next result, which we call the ``Principle of Two Maxima'', is one of the few places where the formalism of Poisson brackets used by Shestakov and Umirbaev seems to be more transparent (at least for us) than the formalism of differential forms used by Kuroda.
In this section we propose a definition that encompasses the two points of view, and then we recall the proof following \cite[Lemma 5]{SU:ineq}.

\begin{proposition}[Principle of Two Maxima, {\cite[Theorem 5.2]{Ku:ineq} and \cite[Lemma 5]{SU:ineq}}]
\label{pro:ptm}
Let $(f_1, f_2, f_3)$ be an automorphism of $\A^3$.
Then the maximum between the following three degrees is realized at least twice:
$$\deg f_1 + \deg df_2 \wedge df_3,\quad
\deg f_2 + \deg df_1 \wedge df_3, \quad
\deg f_3 + \deg df_1 \wedge df_2. $$
\end{proposition}

Let $\Omega$ be the space of algebraic 1-forms $\sum f_idg_i$ where $f_i, g_i \in \K[x_1, x_2, x_3]$.
We consider $\Omega$ as a free module of rank three over $\K[x_1, x_2, x_3]$, with basis $dx_1, dx_2, dx_3$, and we denote by
$$\TT = \bigoplus_{p = 0}^\infty \Omega^{\otimes p}$$
the associative algebra of tensorial powers of $\Omega$, where as usual $\Omega^{\otimes 0} = \K[x_1, x_2, x_3]$.
The degree function on $\Omega$ extends naturally to a degree function on $\TT$.
Recall that $\TT$ has a natural structure of Lie algebra: For any $\omega, \mu  \in \TT$, we define their bracket as
$$[\omega, \mu] := \omega \otimes \mu -  \mu \otimes \omega.$$
In particular, if $df,dg \in \Omega$ are 1-forms, we have
$$[df,dg] = df \otimes dg - dg \otimes df = df \wedge dg.$$
It is easy to check that the bracket satisfies the Jacobi identity:
For any $\alpha, \beta, \gamma \in \TT$, we have
\begin{align*}
\left[ [\alpha, \beta], \gamma \right] &+\left[ [\beta, \gamma], \alpha \right] + \left[ [\gamma, \alpha], \beta \right] \\
&= \,\alpha \otimes \beta \otimes \gamma
 - \beta \otimes \alpha \otimes \gamma
 - \gamma \otimes \alpha \otimes \beta
 + \gamma \otimes \beta \otimes \alpha \\
& \; + \beta \otimes \gamma \otimes \alpha
 - \gamma \otimes \beta \otimes \alpha
 - \alpha \otimes \beta \otimes \gamma
 + \alpha \otimes \gamma \otimes \beta \\
& \; + \gamma \otimes \alpha \otimes \beta
 - \alpha \otimes \gamma \otimes \beta
 - \beta \otimes \gamma \otimes \alpha
 + \beta \otimes \alpha \otimes \gamma \\
&= \, 0
\end{align*}
since each one of the six possible permutations appears twice, with different
signs.

\begin{lemma} \label{lem:independent}
The nine elements $\left[dx_i \wedge  dx_j,dx_k\right]$, for $1\le i <
j \le 3$ and $1 \le k \le 3$ generate a 8-dimensional free submodule in $\TT$,
the only relation between them being the Jacobi identity:
$$[dx_1 \wedge  dx_2, dx_3] + [dx_2 \wedge  dx_3, dx_1] - [dx_1 \wedge  dx_3, dx_2] = 0.$$
\end{lemma}

\begin{proof}
We work inside the 27-dimensional free sub-module of $\TT$ generated by the $dx_i \otimes dx_j \otimes dx_k$ for $1 \le i,j,k \le 3$.
We compute, for $i < j$:
\begin{align*}
\left[dx_i \wedge  dx_j,dx_i\right]
&= \left[dx_i \otimes dx_j - dx_j \otimes dx_i ,dx_i\right] \\
&= 2 dx_i \otimes dx_j \otimes dx_i - dx_j \otimes dx_i \otimes dx_i - dx_i
\otimes dx_i \otimes dx_j, \\
\left[dx_i \wedge  dx_j,dx_j\right]
&= \left[dx_i \otimes dx_j - dx_j \otimes dx_i ,dx_j\right] \\
&= -2 dx_j \otimes dx_i \otimes dx_j + dx_i \otimes dx_j \otimes dx_j + dx_j
\otimes dx_j \otimes dx_i.
\end{align*}
This shows that the elements $\left[dx_i \wedge  dx_j,dx_i\right]$ and $\left[dx_i \wedge  dx_j,dx_j\right]$, for $i < j$, generate a 6-dimensional free submodule.
On the other hand, for $\{i,j,k\} = \{1,2,3\}$:
\begin{multline*}
\left[dx_i \wedge  dx_j,dx_k\right] = \left[dx_i \otimes dx_j - dx_j \otimes dx_i ,dx_k\right] \\
= dx_i \otimes dx_j \otimes dx_k - dx_j \otimes dx_i \otimes dx_k
- dx_k \otimes dx_i \otimes dx_j + dx_k \otimes dx_j \otimes dx_i,
\end{multline*}
so that $[dx_1 \wedge  dx_2, dx_3]$ and $[dx_2 \wedge  dx_3, dx_1]$ are
independent, and together with the above family  they generate a 8-dimensional
free submodule.
\end{proof}

The proof of the Principle of Two Maxima \ref{pro:ptm} now follows from the observation:

\begin{lemma} \label{lem:ptmbracket}
Let $f,g,h \in \K[x_1,x_2,x_3]$, with $h$ non-constant.
Then
$$\deg \left[ \left[df,dg\right], dh \right] = \deg h + \deg df \wedge dg.$$
\end{lemma}

\begin{proof}
We have
$$ dh = \sum_{1 \le k \le 3} \parti{k}{h} dx_k$$
and
$$[df,dg] = df \wedge dg = \sum_{1\le i < j \le 3} \left( \parti{i}{f}\parti{j}{g} - \parti{j}{f}\parti{i}{g}\right)  dx_i \wedge  dx_j.$$
Thus
$$\left[ [df,dg], dh \right] = \sum_{1 \le k \le 3} \sum_{1\le i < j \le 3} \parti{k}{h}\left( \parti{i}{f}\parti{j}{g} - \parti{j}{f}\parti{i}{g}\right)  \left[dx_i \wedge  dx_j,dx_k\right].$$
If the degree of $dh$ is realized by at most two of the terms  $\parti{k}{h} dx_k$, $k =1,2,3$, then by Lemma \ref{lem:independent} the terms realizing the maximum  of the degrees
\begin{equation}\label{eq:deg terms}
\deg \left( \parti{k}{h}\left( \parti{i}{f}\parti{j}{g} - \parti{j}{f}\parti{i}{g}\right)  \left[dx_i \wedge  dx_j,dx_k\right] \right)
\end{equation}
are independent (because at most two of them occur in the Jacobi relation), hence the result since $\deg \left[dx_i \wedge  dx_j,dx_k\right] = \deg x_ix_jx_k$.

On the other hand if the three terms  $\parti{k}{h} dx_k$ have the same degree, then among the indexes $(i,j,k)$ that realize the maximum of the degrees in (\ref{eq:deg terms}), we must find some with $k = i$ or $k=j$, hence again we get the conclusion since by Lemma \ref{lem:independent} such terms cannot cancel each other.
\end{proof}

\begin{proof}[{Proof of the Principle of Two Maxima $\ref{pro:ptm}$}]
Since by the Jacobi identity
$$\left[ [df_1,df_2], df_3 \right] + \left[ [df_2,df_3], df_1 \right] + \left[ [df_3,df_1], df_2 \right] = 0,$$
the dominant terms must cancel each other.
In particular the maximum of the degrees, which are computed in Lemma \ref{lem:ptmbracket}, is realized at least twice: This is the five lines proof of the Principle of Two Maxima by Shestakov and Umirbaev!
\end{proof}

\section{Geometric theory of reduction} \label{sec:theory reduction}

In this section we mostly follow Kuroda \cite{Ku:main}, but we reinterpret his theory of reduction in a combinatorial way, using the complex $\Comp$.
Recall that $\K$ is a field of characteristic zero.

\subsection{Degree of automorphisms and vertices} \label{sec:degrees}

Recall that in \S \ref{sec:degrees} we defined a notion of degree for an automorphism $f = (f_1, f_2, f_3) \in \TA$.
The point is that we want a degree that is adapted to the theory of reduction of Kuroda, so for instance taking the maximal degree of the three components of an automorphism is not good, because we would not detect a reduction of the degree on one of the two lower components (such reductions do exist, see \S \ref{sec:examples}).
We also want a definition that is adapted to working on the complex $\Comp$, so directly taking the sum of the degree of the three components is no good either, since it would not give a degree function on vertices of $\Comp$.

Recall that the \textbf{$3$-degree} of $f  \in \TA$, or of the vertex $v_3 =
[f]$, is the triple $(\delta_1, \delta_2, \delta_3)$ given by Lemma
\ref{lem:stratified degree}, where in particular $\delta_3 > \delta_2 >
\delta_1$.
By definition the \textbf{top degree} of $v_3$ is  $\delta_3 \in \N^3$,
and the \textbf{degree} of $v_3$ is the sum
$$\deg v_3 := \delta_1 + \delta_2 + \delta_3 \in \N^3.$$
Similarly we have a $2$-degree $(\nu_1, \nu_2)$ associated with any vertex $v_2$
of type 2, a top degree equal to $\nu_2$ and a degree $\deg v_2 := \nu_1 +
\nu_2$.
Finally for a vertex of type 1 the notions of $1$-degree, top degree and degree
coincide.

\begin{lemma} \label{lem:id is min}
Let $v_3$ be a vertex of type 3.
Then
$$\deg v_3 \ge (1,1,1)$$
with equality if and only if $v_3 = [\id]$.
\end{lemma}

\begin{proof}
If $v_3 = [f]$ with $\deg v_3 = (1,1,1)$, then the $3$-degree $(\delta_1,  \delta_2, \delta_3)$ of $v_3$ must be equal to $\big( (0,0,1), (0,1,0), (0,0,1) \big)$, hence the result.
\end{proof}

Let $v_3$ be a vertex with $3$-degree $(\delta_1, \delta_2, \delta_3)$.
The unique $m_1 \in \p^2(v_3)$ such that $\deg m_1 = \delta_1$ is called the \textbf{minimal vertex} in $\lines{v_3}$,
and the unique $m_2 \in \linesd{v_3}$ such that $\deg m_2 = (\delta_1, \delta_2)$ is called the \textbf{minimal line} in $\lines{v_3}$.
If $v_2\in \linesd{v_3}$ has $2$-degree  $(\nu_1, \nu_2)$, there is a unique degree $\delta$ such that $v_3$ has degree $\nu_1 + \nu_2 + \delta$.
We denote this situation by $\delta := \deg (v_3 \smallsetminus v_2)$.
Observe that, by definition, $\deg (v_3 \smallsetminus v_2) = \deg v_3 - \deg v_2$.
There is also a unique $v_1$ such that $v_2$ passes through $v_1$ and $\deg v_1 = \nu_1$: we call $v_1$ the minimal vertex of $v_2$.
Observe that if $v_2 = m_2 \in \linesd{v_3}$, then the minimal vertex of $v_2$ coincides with the minimal vertex of $v_3$, and if $v_2 \neq m_2$, then $v_1$ is the intersection of $v_2$ with $m_2$.

We call \textbf{triangle} $T$ in $\lines{v_3}$ the data of three
non-concurrent lines.
A \textbf{good triangle} $T$ in $\lines{v_3}$ is the data of three distinct lines $m_2, v_2, u_2$, such that $m_2$ is the minimal line, $v_2$ passes through the minimal vertex $m_1$, and $u_2$ does not pass through $m_1$.
Equivalently, a good triangle corresponds to a good representative $v_3 = \llb f_1, f_2, f_3 \rrb$ with $\deg f_1 > \deg f_2 > \deg f_3$, by putting $m_2 = \llb f_2, f_3 \rrb$, $v_2 = \llb f_1, f_3 \rrb$, $u_2 = \llb f_1, f_2 \rrb$.
If $v_1, v_2, v_3$ is a simplex in $\Comp$, we say that a good triangle $T$ in $\lines{v_3}$  is \textbf{compatible} with this simplex if $v_2$ is one of the lines of $T$, and $v_1$ is the intersection of $v_2$ with another line of $T$.
Each simplex $v_1, v_2, v_3$ admits such a compatible good triangle (not unique
in general): Indeed it corresponds to a choice of good representatives as
given by Lemma \ref{lem:rep of a simplex}.

Let $v_2  \in \linesd{v_3}$ be a vertex with $2$-degree $(\delta_1, \delta_2)$.
We say that $v_2$ has \textbf{inner resonance} if $\delta_2 \in \N\delta_1$.
We say that $v_2$ has \textbf{outer resonance} in $v_3$ if $\deg (v_3 \smallsetminus v_2) \in \N\delta_1 + \N\delta_2$.

\subsection{Elementary reductions}
Let $v_3, v_3'$ be vertices of type 3.
We say that $v_3'$ is a \textbf{neighbor} of $v_3$ if $v_3' \neq v_3$ and there exists a vertex $v_2$ of type 2 such that $v_2 \in \linesd{v_3} \cap \linesd{v_3'}$.
Equivalently, this means that $v_3$ and $v_3'$ are at distance 2 in $\Comp$.
We denote this situation by $v_3' \ne v_3$, or if we want to make $v_2$
explicit, by $v_3' \ne_{v_2} v_3$
We also say that $v_2$ is the \textbf{center} of $v_3 \ne v_3'$.
Recall that the center $v_2$ is uniquely defined, and that we can choose
representatives as in Corollary \ref{cor:type 3 neighbors}.

We say that $v_3'$ is an \textbf{elementary reduction} (resp. a \textbf{weak} elementary reduction)  of $v_3$ with center $v_2$, if $\deg v_3 > \deg v_3' $ (resp. $\deg v_3 \ge \deg v_3'$) and $v_3' \ne_{v_2} v_3$.
Let $v_1$ be the minimal vertex in the line $v_2$.
We say that $v_1, v_2, v_3$ is the \textbf{pivotal simplex} of the reduction,
and that $v_1$ is the \textbf{pivot} of the reduction. \label{def:pivot}
Moreover we say that the reduction is \textbf{optimal} if $v_3'$ has minimal degree among all neighbors of $v_3$ with center $v_2$.
We say that $v_3'$ is a \textbf{simple} elementary reduction (resp. a \textbf{weak simple} elementary reduction) of
$v_3$ if there exist good representatives $v_3 = \llb f_1, f_2, f_3 \rrb$ and
$v_3' = \llb f_1 + P(f_2) + af_3, f_2, f_3 \rrb$ for some $a \in \K$
and non-affine polynomial $P$, satisfying
\begin{align*}
\deg f_1 &> \max\{f_1 + P(f_2), f_1 + P(f_2) + af_3\}; \\
\text{(resp. } f_1 &\ge \max\{f_1 + P(f_2), f_1 + P(f_2) + af_3\}\text{)}.
\end{align*}
In this situation we say that $v_2 = \llb f_2, f_3 \rrb, v_1 = \llb f_2 \rrb$ is the \textbf{simple center} of the reduction, and when drawing pictures we represent this relation by adding an arrow on the edge from $v_2$ to $v_1$ (see Figure \ref{fig:simple}).
Beware that this representation is imperfect, since the arrow does not depend only on the edge from $v_2$ to $v_1$ but indicates a relation between the two vertices $v_3$ and $v_3'$.

\begin{remark} \label{rem:simple}
In the definition of a (weak) simple elementary reduction, if $\deg f_3 = \topdeg v_3$, then we must have $a = 0$.
For instance in the following example:
\begin{align*}
v_3 &= \llb x_1 + x_2^2, x_2, x_3 + x_2^2 + x_2^3 \rrb = \llb f_1, f_2, f_3 \rrb,\\
v_3' &= \llb x_1 - x_3, x_2, x_3 + x_2^2 + x_2^3 \rrb = \llb f_1 + f_2^3 - f_3, f_2, f_3 \rrb,
\end{align*}
we do not want to call $v_3'$ a simple reduction of $v_3$ because $\deg (f_1 +
f_2^3) > \deg f_1$.

On the other hand, consider the following example:
\begin{align*}
v_3 &= \llb x_1 +x_2^2 + x_2^3, x_2, x_3 + x_2^2 \rrb = \llb f_1, f_2, f_3 \rrb ,\\
v_3' &= \llb x_1 - x_3, x_2, x_3 + x_2^2 \rrb = \llb f_1 - f_2^3 - f_3, f_2, f_3 \rrb.
\end{align*}
Here $v_3'$ is a simple elementary reduction of $v_3$, and the coefficient $a = -1$ is necessary to get a good representative.
\end{remark}

\begin{figure}[t]
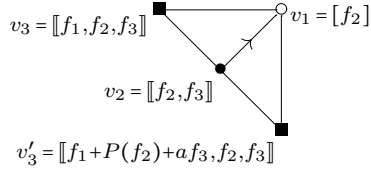

$$
\mygraph{
!{<0cm,0cm>;<0.8cm,0cm>:<0cm,0.8cm>::}
!{(-2,0)}*-{\typethree}="v3"
!{(0,-2)}*-{\typethree}="v3'"
!{(0,0)}*{\typeone}="v1"
!{(-1,-1)}*-{\typetwo}="v2"
"v3"-_<{v_3 \,=\, \llb f_1,f_2,f_3 \rrb}_>{v_2 \,=\, \llb f_2, f_3 \rrb}"v2"-_>{v_3' \,=\, \llb f_1+ P(f_2) +af_3,f_2,f_3  \rrb}"v3'"
"v1"-|@{<}"v2"
"v3'"-_>{v_1 \,=\, [ f_2 ]}"v1"-"v3"
}
$$
\caption{Simple reduction with simple center $v_2,v_1$.}\label{fig:simple}
\end{figure}

\begin{lemma} \label{lem:elem red}
Let $v_3'$ be a neighbor of $v_3 = [ f_1,f_2,f_3 ]$  with center $v_2' = \llb f_1,f_2\rrb$.
Then there exists a non-affine polynomial $P \in \K[y,z]$ such that $v_3'
= \llb
f_1 , f_2,f_3 + P(f_1,f_2)\rrb$.

Moreover:
\begin{enumerate}
\item If $v_3'$ is a weak elementary reduction of $v_3$, then $\deg f_3 \ge \deg P(f_1,f_2)$;
\item If $v_3'$ is an elementary reduction of $v_3$, then $\deg f_3 = \deg P(f_1,f_2)$.
\end{enumerate}
\end{lemma}

\begin{proof}
From Corollary \ref{cor:type 3 neighbors} we know that $v_3'$ has the form $v_3' = [ f_1 , f_2,f_3 + P(f_1,f_2) ]$.
Since by assumption $\deg f_1 \neq \deg f_2$, there exist $a, b \in \K$ such that $(f_1 , f_2,f_3 + P(f_1,f_2) + af_1 +bf_2)$ is a good representative for $v_3'$.
So up to changing $P$ by a linear combination of $f_1$ and $f_2$ we can assume $v_3' = \llb f_1 , f_2,f_3 + P(f_1,f_2)\rrb$.

If $v_3'$ is a weak elementary reduction of $v_3$, then we have
$$\deg f_1 + \deg f_2 + \deg f_3 \ge \deg v_3 \ge \deg v_3'
= \deg f_1 + \deg f_2 + \deg (f_3 + P(f_1, f_2)).$$
So $\deg f_3 \ge \deg (f_3 + P(f_1, f_2))$, which implies $\deg f_3 \ge \deg P(f_1, f_2)$.

Finally if $v_3'$ is an elementary reduction of $v_3$, that is, $\deg v_3 > \deg v_3'$, then the same computation gives
$\deg f_3 > \deg (f_3 + P(f_1, f_2))$, which implies that $\deg f_3 = \deg
P(f_1, f_2)$.
\end{proof}

\begin{lemma}[Square Lemma] \label{lem:square}
Let $v_3, v_3', v_3''$ be three vertices such that:
\begin{itemize}[wide]
\item $v_3' \ne_{v_2'} v_3$ and $v_3'' \ne_{v_2''} v_3$ for some $v_2' \neq
v_2''$ that are part of a good triangle of $v_3$ (this is automatic if $v_2'$
or $v_2''$ is the minimal line of $v_3$);
\item Denoting $v_1$ the common vertex of $v_2'$ and $v_2''$, $v_3''$ is a (possibly weak) simple elementary reduction of $v_3$ with simple center $v_2'', v_1$;
\item $\deg v_3 \ge \deg v_3'$, $\deg v_3 \ge \deg v_3''$, with at least one of the inequalities being strict.
\end{itemize}
Then there exists $u_3$ such that $u_3 \ne v_3'$, $u_3 \ne v_3''$ and $\deg v_3
> \deg u_3$.
\end{lemma}

\begin{proof}
We pick $f_2$ such that $v_1= [f_2]$, and then we take good representatives $v_2' = \llb f_1, f_2 \rrb$, $v_2'' = \llb f_2,f_3 \rrb$.
Since $v_2'$ and $v_2''$ are part of a good triangle, we have $v_3 = \llb f_1, f_2, f_3 \rrb$.
By Lemma \ref{lem:elem red}, and since $v_3''$ is a (possibly weak) simple
reduction of $v_3$, there exist $a \in \K$, $Q \in \K[f_2]$ and $P \in
\K[f_1,f_2]$ such that
\begin{align*}
v_3' &= \llb f_1, f_2, f_3 + P(f_1,f_2)\rrb; \\
v_3'' &= \llb f_1 + a f_3 + Q(f_2), f_2, f_3 \rrb.
\end{align*}
We have
\begin{align*}
\deg f_3 &\ge \deg (f_3 + P(f_1,f_2)), \\
\deg f_1 &\ge \max\{ \deg (f_1 + a f_3 + Q(f_2)),\deg (f_1 + Q(f_2)) \},
\end{align*}
with one of the two inequalities being strict.
We define
$$u_3 := [f_1 + Q(f_2), f_2, f_3 + P(f_1,f_2) ].$$
Observe that $u_3$ is a neighbor of both $v_3'$, with center $\llb f_2, f_3 + P(f_1,f_2) \rrb$, and $v_3''$, with center $[f_1 + Q(f_2), f_2]$ (see Figure \ref{fig:square}).
The inequality on degrees follows from:
\begin{align*}
\deg v_3 &= \deg f_3 + \deg f_2 + \deg f_1 \\
&> \deg (f_3 + P(f_1,f_2)) + \deg f_2 + \deg (f_1 + Q(f_2))\\
&\ge \deg u_3. \qedhere
\end{align*}
\end{proof}

\begin{figure}[t]
$$
\mygraph{
!{<0cm,0cm>;<0.9cm,0cm>:<0cm,0.9cm>::}
!{(2,0)}*-{\typethree}="v3''"
!{(0,2)}*-{\typethree}="v3"
!{(0,-2)}*-{\typethree}="u3"
!{(-2,0)}*-{\typethree}="v3'"
!{(0,0)}*{\typeone}="v1"
!{(1,1)}*-{\typetwo}="v2''"
!{(1,-1)}*-{\typetwo}="u2''"
!{(-1,1)}*-{\typetwo}="v2'"
!{(-1,-1)}*-{\typetwo}="u2'"
"v3"-^<{v_3=\llb f_1, f_2, f_3 \rrb}^>{v_2'= \llb f_1,f_2 \rrb}"v2''"-^>{v_3'= \llb f_1, f_2, f_3 + P(f_1,f_2)\rrb}"v3''"-^>{u_2'' = \llb f_2, f_3 + P(f_1,f_2) \rrb}"u2''"-^>{u_3= [f_1 + Q(f_2), f_2, f_3 + P(f_1,f_2) ]}"u3"-"u2'"-^<{u_2' = [f_1 + Q(f_2),f_2]}"v3'"-^<{v_3''= \llb f_1 + a f_3 + Q(f_2), f_2, f_3\rrb}"v2'"-^<{v_2''= \llb f_2,f_3 \rrb}"v3"
"u2''"-|@{>}"v1"-"u2'"
"u3"-"v1"-"v3"
"v3'"-"v1"-^<(.43){v_1= [f_2]}"v3''"
"v2'"-|@{>}"v1"-"v2''"
}
$$
\caption{Square Lemma \ref{lem:square}.}\label{fig:square}
\end{figure}

\subsection{$K$-reductions} \label{sec:K}

If $f_1, f_2 \in \K[x_1,x_2,x_3]$ are two algebraically independent polynomials with $\deg f_1 > \deg f_2$, we introduce the degree
$$\Delta(f_1,f_2) := \deg f_1 - \deg f_2 + \deg df_1 \wedge df_2 \in \Z^3.$$
Assuming that $v_2 =\llb f_1,f_2 \rrb$ is a vertex of type 2, we
define
$$d(v_2) := \deg df_1 \wedge df_2 \quad\text{ and }\quad \Delta(v_2) :=
\Delta(f_1,f_2).$$
We call $d(v_2)$ and $\Delta(v_2)$ respectively the \textbf{differential degree} and the \textbf{delta degree} of $v_2$.
It is easy to check that these definitions do not depend on a choice of representative.
In fact, for any $\left( \begin{smallmatrix}
                      \alpha & \beta \\ \gamma & \delta
                      \end{smallmatrix} \right) \in \GL_2(\K)$ we have
$$d(\alpha f_1 + \beta f_2) \wedge d( \gamma f_1 + \delta f_2) = (\alpha\delta - \beta\gamma) df_1 \wedge df_2,$$
so in the definition of $d(v_2)$ we could use any representative.
On the other hand in the definition of $\Delta(v_2)$, because of the term $\deg f_1 - \deg f_2$, we really need to work with a good representative.
Observe also that, by definition, for any vertex $v_2$ we have
$$\Delta(v_2) > d(v_2).$$

We now introduce the key concept of $K$-reduction, where we let the reader
decide for himself whether the $K$ should stand for ``Kuroda'' or for
``Kazakh''.

More precisely by a \textbf{$K$-reduction} we shall mean either an elementary
$K$-reduction, or a proper $K$-reduction, two notions that we now define.
Let $v_3$ and $u_3$ be vertices of type 3.

We say that $u_3$ is an \textbf{elementary $K$-reduction} of $v_3$ if $u_3 \ne v_3$ and:
\begin{enumerate}[(K$1$)]
\setcounter{enumi}{-1}
\item\label{def:Kdeg} $\deg v_3 > \deg u_3$;
\item\label{def:Kinn} the center $v_2$ of $v_3 \ne u_3$ has no inner resonance;
\item\label{def:Kout} $v_2$ has no outer resonance in $v_3$;
\item\label{def:Kmin} $v_2$ is not the minimal line in $\lines{v_3}$;
\item\label{def:Kdel} $\Delta(v_2) > \deg (u_3 \smallsetminus v_2 )$.
\end{enumerate}
Denoting by $v_1$ the minimal point in $v_2$, as before (see definition from page \pageref{def:pivot}) we call $v_1$ the \textbf{pivot}, and $v_1, v_2, v_3$ the \textbf{pivotal simplex} of the elementary $K$-reduction (denoted by $\circlearrowleft$ on Figure \ref{fig:redK}).

We say that $u_3$ is a \textbf{proper $K$-reduction} of $v_3$ via the auxiliary vertex $w_3$ if $v_3$ is a weak elementary reduction of $w_3$ with center the minimal line of $w_3$, and $u_3$ is an elementary $K$-reduction of $w_3$.
Formally, this corresponds to the following conditions:
\begin{enumerate}[(K$1'$)]
\setcounter{enumi}{-1}
\item\label{def:KPdeg} $\deg w_3 > \deg u_3$;
\item\label{def:KPinn} the center $w_2$ of $w_3 \ne u_3$  has no inner resonance;
\item\label{def:KPout} $w_2$ has no outer resonance in $w_3$;
\item\label{def:KPmin} $w_2$ is not the minimal line in $\lines{w_3}$;
\item\label{def:KPdel} $\Delta(w_2) > \deg (u_3 \smallsetminus w_2 )$.
\item\label{def:KPdegbis} $\deg w_3 \ge \deg v_3$;
\item\label{def:KPminbis} the center $m_2$ of $w_3 \ne v_3$ is the minimal line in $\lines{w_3}$.
\end{enumerate}
Observe that the pivot $v_1$ of the elementary $K$-reduction from $w_3$ to $u_3$ is the common vertex of the distinct lines $m_2$ and $w_2$ in $\lines{w_3}$.
The simplex $v_1, w_2, w_3$ is still called the \textbf{pivotal simplex} of the proper $K$-reduction.
It will be proved in Proposition \ref{pro:K reduces degree} that the above conditions \ref{def:KPdeg} to \ref{def:KPminbis} imply $\deg v_3 > \deg u_3$, so that the terminology of ``reduction'' is not misleading, even if by no means obvious at this point.

\begin{figure}[t]
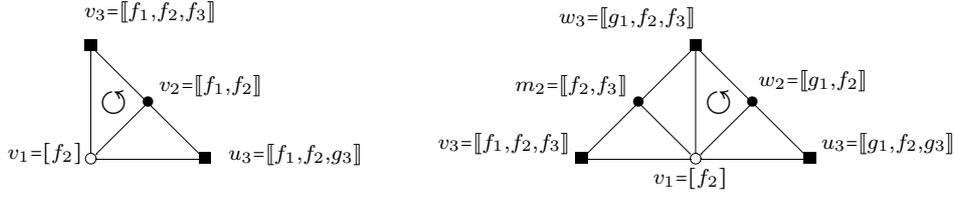

$$
\mygraph{
!{<0cm,0cm>;<0.75cm,0cm>:<0cm,0.75cm>::}
!{(0,2)}*-{\typethree}="v3"
!{(2,0)}*-{\typethree}="v3'"
!{(1,1)}*-{\typetwo}="g1g2"
!{(0,0)}*{\typeone}="v1"
!{(.4,1)}*{\circlearrowleft}
"v3"-^<(-.2){v_3 =\llb f_1, f_2, f_3 \rrb}^>{v_2 = \llb f_1,f_2 \rrb}"g1g2"-^>(1.2){u_3=\llb f_1,f_2,g_3\rrb}"v3'"
"v3"-_>{v_1 = [ f_2 ]}"v1"-"v3'"
"v1"-"g1g2"
}
\qquad
\mygraph{
!{<0cm,0cm>;<0.75cm,0cm>:<0cm,0.75cm>::}
!{(-2,0)}*-{\typethree}="v3"
!{(2,0)}*-{\typethree}="v3'"
!{(0,0)}*{\typeone}="v1"
!{(0,2)}*-{\typethree}="w"
!{(-1,1)}*-{\typetwo}="f2f3"
!{(1,1)}*-{\typetwo}="g1g2"
!{(.4,1)}*{\circlearrowleft}
"f2f3"-_>{v_3=\llb f_1,f_2,f_3\rrb}"v3"-_>{v_1 = [ f_2 ]}"v1"-"f2f3"-^<{m_2 = \llb f_2,f_3 \rrb}^>(1.1){w_3 =\llb g_1,f_2,f_3 \rrb}"w"-"v1"-"v3'"-_<{u_3=\llb g_1,f_2,g_3\rrb}_>{w_2 = \llb g_1,f_2 \rrb}"g1g2"-"w"
"v1"-"g1g2"
}
$$
\caption{Elementary and proper $K$-reductions, with good representatives as in Set-Up \ref{setup:good_rep}.}\label{fig:redK}
\end{figure}

Let $v_1, v_2, v_3$ be a simplex, and let $s \ge 3$ be an odd integer.
We say that the simplex $v_1, v_2, v_3$ has \textbf{Strong Pivotal Form }$\boldsymbol{\PF(s)}$ if
\begin{enumerate}[(\PF$1)$]
\item \label{SFinn} $\deg v_1 = 2\delta$ and 2-$\deg v_2 = (2\delta, s\delta)$ for some $\delta \in \N^3$;
\item \label{SFout} $v_2$ has no outer resonance in $v_3$;
\item \label{SFmin} $v_2$ is not the minimal line in $\lines{v_3}$;
\item \label{SFdel} $\deg(v_3 \smallsetminus v_2) \ge \Delta(v_2)$.
\end{enumerate}

In all the previous definitions we took care of working with vertices, and not with particular representatives.
However for writing proofs it will often be useful to choose representatives.

\begin{setup} \phantomsection \label{setup:good_rep}
\begin{enumerate}[wide]
\item \label{good_rep:1}
Let $u_3$ be an elementary $K$-reduction of $v_3$, with pivotal simplex $v_1,v_2,v_3$.
Then there exist representatives
\begin{align*}
v_1 &= [ f_2 ] & v_3 &= \llb f_1, f_2, f_3 \rrb\\
v_2 &= \llb f_1, f_2 \rrb & u_3 &= \llb f_1, f_2, g_3 \rrb
\end{align*}
such that $g_3 = f_3 + \phi_3(f_1, f_2)$, where $\phi_3 \in \K[x,y]$.
Observe that, by definition, $\deg f_1 > \deg f_2$ and
$\deg f_1 > \deg f_3 > \deg g_3.$

\item \label{good_rep:2}
Let $u_3$ be a proper $K$-reduction of $v_3$, with pivotal simplex $v_1,w_2,w_3$.
Then there exist representatives
\begin{align*}
v_1 &= [ f_2 ] & w_3 &= \llb g_1, f_2, f_3 \rrb \\
m_2 &= \llb f_2, f_3 \rrb & v_3 &= \llb f_1, f_2, f_3 \rrb\\
w_2 &= \llb g_1, f_2 \rrb & u_3 &= \llb g_1, f_2, g_3 \rrb
\end{align*}
such that $g_1 = f_1 + \phi_1(f_2, f_3)$ and $g_3 = f_3 + \phi_3(g_1, f_2)$, where $\phi_1, \phi_3 \in \K[x,y]$.
\end{enumerate}
\end{setup}

\begin{proof}
\begin{enumerate}[wide]
\item
Pick any good representatives $v_1 = [ f_2 ]$, $v_2 =  \llb f_1, f_2 \rrb$, $v_3
= \llb f_1, f_2, f_3 \rrb$ as given by Lemma \ref{lem:rep of a simplex}, and
apply Lemma \ref{lem:elem red} to get $g_3$.
\item
Pick any representative $v_1 = [ f_2 ]$, and then pick $f_3, g_1$ such that $v_2 = \llb f_2, f_3 \rrb$ and $w_2 = \llb g_1, f_2 \rrb$.
Since $m_2$ is the minimal line in $\hat \p^2(w_3)$, we have $\deg g_1 > \deg f_2$ and $\deg g_1 > \deg f_3$, hence $(g_1, f_2, f_3)$ is a good representative for $w_3$.
Now apply Lemma \ref{lem:elem red} twice to get $f_1$ and $g_3$.
 \qedhere
\end{enumerate}
\end{proof}

We establish a first property of a simplex with Strong Pivotal Form.

\begin{lemma} \label{lem:m2 in SPF}
Let $v_1, v_2, v_3$ be a simplex with Strong Pivotal Form $\PF(s)$ for some odd $s \ge 3$.
Then the minimal line $m_2$ in $\lines{v_3}$ has no inner resonance.
\end{lemma}

\begin{proof}
We pick representatives $v_1 = [f_2]$, $v_2 = \llb f_1, f_2 \rrb$, $v_3 = \llb
f_1, f_2, f_3 \rrb$ as given by Lemma \ref{lem:rep of a simplex}.
By \ref{SFinn} we have $\deg f_1 = s\delta > 2\delta = \deg f_2$.
Since $v_2$ is not the minimal line by \ref{SFmin}, the minimal line must be $m_2 = \llb f_2, f_3 \rrb$.
Then by \ref{SFdel} we have $\deg f_3 > (s-2)\delta \ge \delta$, so that $\deg f_2 \not\in \N\deg f_3$.
Since by \ref{SFout} we also have $\deg f_3 \not\in \N\deg f_2$, we conclude that the minimal line $m_2 = \llb f_2, f_3 \rrb$ has no inner resonance.
\end{proof}

We can rephrase results from Corollary \ref{cor:parachute} with the previous definitions (see also Example \ref{exple:noSPF} for some complements):

\begin{proposition} \label{pro:SPF}
Let $v_3$ be a vertex that admits an elementary reduction with pivotal simplex $v_1, v_2, v_3$.
\begin{enumerate}
\item \label{SPF1}
Assume $v_2$ has no inner resonance, and no outer resonance in $v_3$.
Then
$$\deg (v_3 \smallsetminus v_2) > d(v_2).$$
\item \label{SPF2}
If moreover $v_2$ is not the minimal line in $\lines{v_3}$, then
 $v_1, v_2, v_3$ has Strong Pivotal Form $\PF(s)$ for some odd $s \ge 3$.
\end{enumerate}
\end{proposition}

\begin{proof}
\begin{enumerate}[wide]
\item
We pick good representatives $v_1= [ f_2 ]$, $v_2 = \llb f_1, f_2 \rrb$, $v_3 =
\llb f_1, f_2, f_3 \rrb$ as given by Lemma \ref{lem:rep of a simplex}.
By Lemma \ref{lem:elem red}, the elementary reduction has the form
$u_3 = \llb f_1, f_2, f_3 + P(f_1,f_2) \rrb$
with $\deg f_3 = \deg P(f_1,f_2)$.
Since $v_2$ has no outer resonance in $v_3$, we have $\deg f_3 \not\in \N\deg f_1 + \N \deg f_2$, hence
$$\dvirt P(f_1,f_2) > \deg P(f_1,f_2).$$
Since moreover $v_2$ has no inner resonance, we can apply Corollary \ref{cor:parachute}\ref{parachute-ii} to get the inequality $\deg (v_3 \smallsetminus v_2) > d(v_2)$.
\item
By assumption the simplex $v_1, v_2, v_3$ already satisfies conditions
\ref{SFout} and \ref{SFmin}.
The condition that $v_2$ is not the minimal line in $\lines{v_3}$ is equivalent
to
$$\max\{\deg f_1, \deg f_2\} > \deg f_3 = \deg P(f_1,f_2),$$
hence we can apply Corollary \ref{cor:parachute}\ref{parachute-iii},
which yields conditions \ref{SFinn} and \ref{SFdel}.\qedhere
\end{enumerate}
\end{proof}

\subsection{Elementary $K$-reductions} \label{sec:K elementary}

Here we list some properties of an elementary $K$-reduction.
First we have the following corollary from Proposition \ref{pro:SPF}.

\begin{corollary} \label{cor:SPF}
The pivotal simplex of a $K$-reduction has Strong Pivotal Form $\PF(s)$ for some odd $s \ge 3$.
\end{corollary}

\begin{proof}
First, let $v_1, v_2, v_3$ be the pivotal simplex of an elementary $K$-reduction.
We know that,  by \ref{def:Kinn}, $v_2$ has no inner resonance, by
\ref{def:Kout}, $v_2$ has no outer resonance in $v_3$, and by \ref{def:Kmin},
$v_2$ is not the minimal line in $v_3$, so we can apply Proposition
\ref{pro:SPF}\ref{SPF2}.

Now the pivotal simplex of proper $K$-reduction is by definition the pivotal simplex of an elementary $K$-reduction from the auxiliary vertex, so that the above argument applies.
\end{proof}

\begin{lemma} \label{lem:dm2}
Let $u_3$ be an elementary $K$-reduction of $v_3$ with center $v_2$, $m_2$ the
minimal line in $\lines{v_3}$, and $u_2 \in \lines{v_3}$ a line not passing
through the pivot $v_1$ of the reduction.
Then
\begin{enumerate}
\item \label{dm2:1} $d(m_2) \ge \deg (v_3 \smallsetminus m_2) + d(v_2)$;
\item \label{dm2:3} $d(u_2) > d(m_2) > d(v_2)$;
\item \label{dm2:inter} The function
$$t_2 \in \linesd{v_3} \mapsto d(t_2) \in \N^3$$
only takes the three distinct values $d(u_2) > d(m_2) > d(v_2)$, and it takes
its minimal value only at the point $v_2$;
\item \label{dm2:2} $\deg (v_1) + d(u_2) > 2\deg(v_3
\smallsetminus m_2)$.
\end{enumerate}
\end{lemma}

\begin{proof}
The assumption means that $v_2, m_2, u_2$ form a (not necessarily good) triangle.
We use the notation from Set-Up \ref{setup:good_rep}, we therefore have $m_2 =
\llb f_2, f_3 \rrb$, $v_2 = \llb f_1, f_2 \rrb$ and $v_1 = [f_2]$.
\begin{enumerate}[wide]
\item
On the one hand:
$$df_2 \wedge df_3 = df_2 \wedge dg_3 - df_2 \wedge d (\phi_3(f_1,f_2)).$$
On the other hand, the following sequence of inequalities holds, where the first one comes from Corollary \ref{cor:parachute}\ref{parachute-iv}, the second one from \ref{def:Kdel}, and the third one from Lemma \ref{lem:basics forms}:
\begin{equation} \label{eq:bycor(iv)}
\deg df_2 \wedge d (\phi_3(f_1,f_2)) \ge
 \deg f_1 + \deg df_1\wedge df_2 > \deg f_2 + \deg g_3 \ge \deg df_2 \wedge dg_3.
\end{equation}
So $\deg df_2 \wedge df_3 = \deg df_2 \wedge d (\phi_3(f_1,f_2))$, and replacing in (\ref{eq:bycor(iv)}) we obtain the expected inequality:
\begin{equation*}
d(m_2) = \deg df_2 \wedge df_3 \ge \deg f_1 + \deg df_1 \wedge df_2 = \deg (v_3 \smallsetminus m_2) + d(v_2).
\end{equation*}

\item
By the previous point we have
$$\deg f_1 + \deg df_2 \wedge df_3 > \deg f_3 + \deg df_1 \wedge df_2,$$
hence by the Principle of Two Maxima \ref{pro:ptm} we get
\begin{equation} \label{eq:P 2 M}
\deg f_2 + \deg df_1 \wedge df_3 = \deg f_1 + \deg df_2 \wedge df_3.
\end{equation}
Since $\deg f_1 > \deg f_2$ we get $\deg df_1 \wedge df_3 > \deg df_2 \wedge
df_3$, and finally
\begin{equation} \label{eq:3 terms ineq}
\deg df_1 \wedge df_3 > \deg df_2 \wedge df_3 > \deg df_1 \wedge df_2.
\end{equation}
The general form of $u_2$ being $u_2 = \llb f_1 + \alpha f_2, f_3 + \beta
f_2 \rrb$, we have
\begin{equation*}
d(u_2) = \deg (df_1 + \alpha\, df_2) \wedge (df_3 + \beta\, df_2) = \deg df_1
\wedge df_3,
\end{equation*}
so that the expected result is exactly (\ref{eq:3 terms ineq}).

\item We just saw that if $t_2$ is any line not passing through $[f_2]$, then
$d(t_2) = \deg df_1 \wedge df_3 = d(u_2)$.

Now consider $t_2$ passing through $[f_2]$ but not equal to $v_2$.
Then $t_2 = [f_2, f_3 + \alpha f_1]$ for some $\alpha \in \K$ and,
using (\ref{eq:3 terms ineq}):
$$d(t_2) = \deg \big( df_2 \wedge df_3 - \alpha\, df_1 \wedge df_2 \big) = \deg
df_2 \wedge df_3 = \deg (m_2).$$

We obtain that $v_2$ is the unique minimum of the function
$$t_2 \in \linesd{v_3} \mapsto d(t_2).$$

\item 
By (\ref{eq:P 2 M}) we have
$$\deg (v_1) + d(u_2) = \deg (v_3 \smallsetminus m_2) +
d(m_2).$$
Since by assertion \ref{dm2:1} we have $d(m_2) > \deg(v_3 \setminus
m_2)$, the result follows. \qedhere
\end{enumerate}
\end{proof}

As an immediate consequence of Lemma \ref{lem:dm2}\ref{dm2:inter} we get:

\begin{corollary} \label{cor:center of a K red}
Assume that $v_3$ admits an elementary $K$-reduction, and that one of the following holds:
\begin{enumerate}
\item \label{center:1} There exists a (non necessarily good) triangle $u_2,
v_2, w_2 \in \linesd{v_3}$  such that
$$d(u_2) > d(w_2) > d(v_2);$$
\item \label{center:2} $m_2$ is the minimal line in $\lines{v_3}$, and $v_2$ is another line in $\lines{v_3}$ such that
$$d(m_2) > d(v_2).$$
\end{enumerate}
Then $v_2$ is the center of the $K$-reduction.
\end{corollary}

\subsection{Proper $K$-reductions}

In this section we list some properties of proper $K$-reductions, and introduce the concept of a normal $K$-reduction.

\begin{proposition} \label{pro:foldable}
Let $u_3$ be a proper $K$-reduction of $v_3$, via $w_3$.
Then (using notation from Set-Up $\ref{setup:good_rep}$):
\begin{enumerate}
\item \label{foldable:2} $g_1 = f_1 + \phi_1(f_2,f_3)$ with
$\dvirt \phi_1(f_2,f_3) = \deg \phi_1(f_2,f_3).$
\item \label{foldable:3} If the pivotal simplex has Strong Pivotal Form $\PF(s)$
with $s \ge 5$, then $v_3$ is a weak simple elementary reduction of $w_3$, with
simple center $m_2, v_1$, and $\deg v_3 = \deg w_3$.
\end{enumerate}
\end{proposition}

\begin{proof}
First observe that by Corollary \ref{cor:SPF} we know that the pivotal simplex $v_1, w_2, w_3$ has Strong Pivotal Form $\PF(s)$.
In particular by Lemma \ref{lem:m2 in SPF} the minimal line $m_2 = \llb f_2, f_3 \rrb$ of $w_3$ has no inner resonance.

\begin{enumerate}[wide]
\item Assume by contradiction that $\dvirt \phi_1(f_2,f_3) > \deg \phi_1(f_2,f_3)$.

By non resonance of $\llb f_2, f_3 \rrb$ we can apply Corollary \ref{cor:parachute}\ref{parachute-ii} and Lemma \ref{lem:dm2}\ref{dm2:1} to get the contradiction
\begin{equation*}
\deg \phi_1(f_2,f_3) > \deg df_2 \wedge df_3 = d(m_2) >  \deg g_1.
\end{equation*}

\item We just established
$$\deg g_1 \ge  \deg \phi_1(f_2,f_3) = \dvirt \phi_1(f_2,f_3).$$
Since $v_1, w_2,w_3$ has Strong Pivotal Form $\PF(s)$, we have $\deg g_1 = s\delta$, $\deg f_2 = 2\delta$ and  $\deg f_3 > (s-2)\delta$, so that $\deg (f_2f_3) > \deg g_1$.
As soon as $s \ge 5$ we also have $2(s-2) > s$, hence $\deg f_3^2 > \deg g_1$.
This implies that if $s \ge 5$ then $\phi_1$ has the form $\phi_1(f_2,f_3) =  a f_3 + Q(f_2)$, as expected.
Moreover since $\deg Q(f_2) = 2r\delta$ for some $r \ge 2$ and since $s$ is
odd, we have $\deg g_1 > \dvirt \phi_1(f_2,f_3) = \deg \phi(f_2,f_3)$ hence
$\deg f_1 = \deg g_1$ and $\deg v_3 = \deg w_3$.
\qedhere
\end{enumerate}
\end{proof}

In view of the previous proposition, we introduce the following definition.
We say that $u_3$ is a \textbf{normal} $K$-reduction of $v_3$ in any of the two following situations:
\begin{itemize}
\item either $u_3$ is an elementary $K$-reduction of $v_3$;
\item or $u_3$ is a proper $K$-reduction of $v_3$ via an auxiliary vertex $w_3$, and, denoting by $v_1$ the pivot of the reduction and $m_2$ the minimal line in $w_3$, the vertex $v_3$ is not a weak \textit{simple} elementary reduction of $w_3$ with center $m_2, v_1$.
\end{itemize}

Given a proper $K$-reduction, Corollary \ref{cor:SPF} and Proposition
\ref{pro:foldable}\ref{foldable:3} say that if the reduction is normal then the
pivotal simplex has Strong Pivotal Form $\PF(3)$.
We now prove the converse, and give some estimations on the degrees
involved.

\begin{lemma}\label{lem:normal proper}
Assume that $u_3$ is a proper  $K$-reduction of $v_3$, via $w_3$, and that the pivotal simplex has Strong Pivotal Form $\PF(3)$.
Then the reduction is normal, and using representatives as from Set-Up $\ref{setup:good_rep}$, we have:
\begin{align}
\label{NP1} \deg g_1 = 3\delta, \quad \deg f_2 &= 2\delta, \quad \tfrac32\delta \ge \deg f_3 > \delta, \\
\label{NP2} \deg df_1 \wedge df_3 &= \delta + \deg df_2 \wedge df_3\ge 4\delta + \deg dg_1 \wedge df_2,\\
\label{NP3} \deg df_1 \wedge df_2 &= \deg f_3 + \deg df_2 \wedge df_3.
\end{align}
Moreover we have the implications:
\begin{align}
\label{NP4} \deg w_3 > \deg v_3 &\Longrightarrow \deg f_3 = \tfrac32\delta, \quad \deg f_1 > \tfrac52 \delta.\\
\label{NP5} \deg w_3 = \deg v_3 &\Longrightarrow \deg f_1 = \deg g_1 = 3\delta.
\end{align}
In any case we have
\begin{equation}
\label{NP6bis}
\deg f_1 > \deg f_2 > \deg f_3,
\end{equation}
\begin{equation}
\label{NP6} \deg df_1 \wedge df_2 > \deg df_1 \wedge df_3 > \deg df_2 \wedge df_3,
\end{equation}
$m_2 = \llb f_2, f_3 \rrb$ is the minimal line of $v_3$, and for any other line $\ell_2 \in \linesd{v_3}$, we have
\begin{equation}
\label{NP7} d(\ell_2) > d(m_2).
\end{equation}
\end{lemma}

\begin{proof}
The equalities $\deg g_1 = 3\delta$ and $\deg f_2 = 2\delta$ come from the fact that the pivotal simplex has Strong Pivotal Form $\PF(3)$.
Property \ref{SFdel} gives $\deg f_3 > \delta$, and Proposition \ref{pro:foldable}\ref{foldable:2} says that $\dvirt \phi_1(f_2,f_3) = \deg \phi_1(f_2,f_3)$.
Hence apart from $f_2$, $f_3$ and $f_3^2$, any monomial in $f_2, f_3$ has degree strictly bigger than $g_1$.
So there exists $a,b,c \in \K$ such that
$$w_3 = \llb g_1 = f_1 + af_ 3^2 + bf_2 + cf_3, f_2, f_3 \rrb.$$
Moreover, since $w_3 \neq v_3$, we have $a \neq 0$, which implies $3\delta = \deg g_1 \ge 2\deg f_3$.
This proves (\ref{NP1}), and the fact that the $K$-reduction is normal.

By Lemma \ref{lem:dm2}\ref{dm2:1} we have
\begin{equation}\label{eq:df2df3}
\deg df_2 \wedge df_3 = d(m_2) \ge \deg(w_3\smallsetminus m_2) + d(w_2) = \deg g_1 + \deg dg_1 \wedge df_2.
\end{equation}
This implies
$$\deg g_1 + \deg df_2 \wedge df_3 > \deg f_3 + \deg dg_1 \wedge df_2.$$
By the Principle of Two Maxima \ref{pro:ptm}, we get
\begin{equation}\label{eq:2max}
\deg g_1 + \deg df_2 \wedge df_3 = \deg f_2 + \deg dg_1 \wedge df_3.
\end{equation}
Now $dg_1 \wedge df_3 = df_1 \wedge df_3 + bdf_2 \wedge df_3$, and the previous equality implies $\deg dg_1 \wedge df_3 > \deg df_2 \wedge df_3$, so that
$$\deg dg_1 \wedge df_3 = \deg df_1 \wedge df_3.$$
Now combining (\ref{eq:df2df3}) and (\ref{eq:2max}) we get the expected inequality (\ref{NP2}):
\begin{equation*}
\deg df_1 \wedge df_3 = \delta + \deg df_2 \wedge df_3 \ge 4\delta + \deg dg_1 \wedge df_2.
\end{equation*}

Observe that $\deg w_3 > \deg v_3$ is equivalent to $\deg g_1 = \deg f_3^2 > \deg f_1$.
So in this situation $\deg f_3 = \frac32 \delta$, and since by Lemma \ref{lem:basics forms} we have $\deg f_1 + \deg f_3 \ge \deg df_1 \wedge df_3$, from (\ref{NP2}) we also get $\deg f_1 > 4\delta - \frac32 \delta = \frac52 \delta$.
This proves (\ref{NP4}), and (\ref{NP5}) is immediate.
These two assertions imply that we always have $\deg f_1 > \frac52 \delta$, so
that we get (\ref{NP6bis}), and
the minimal line $m_2 =\llb f_2, f_3 \rrb$ of $w_3$ also is the minimal line of
$v_3$.

For the equality (\ref{NP3}) we start again from $g_1 = f_1 + af_3^2 + bf_2 + cf_3$, which gives
$$dg_1 \wedge df_2 = df_1 \wedge df_2 -2 a f_3 df_2 \wedge df_3 - c\, df_2 \wedge df_3 .$$
Since (\ref{eq:df2df3}) implies $\deg (f_3 df_2 \wedge df_3) > \deg dg_1 \wedge df_2$, the first two terms on the right-hand side must have the same degree, which is the expected equality.

Now (\ref{NP2}), (\ref{NP3}) and the inequality $\deg f_3 > \delta$ from
(\ref{NP1}) immediately implies (\ref{NP6}).

Finally, for any line $\ell_2$ distinct from $m_2 = \llb f_2, f_3\rrb$, we have
$$d(\ell_2) = \deg (\alpha\, df_1 \wedge df_2 + \beta\, df_1 \wedge df_3
+\gamma\, df_2 \wedge df_3)$$
with $(\alpha, \beta) \neq (0,0)$, from which we obtain (\ref{NP7}).
\end{proof}

Now we can justify the terminology of ``reduction'', as announced when we gave the definition of a proper $K$-reduction:

\begin{proposition} \label{pro:K reduces degree}
Let $u_3$ be a proper $K$-reduction of a vertex $v_3$.
Then
$$\deg v_3 > \deg u_3.$$
\end{proposition}

\begin{proof}
We use the notation from Set-Up \ref{setup:good_rep}.
Observe that if $\deg w_3 = \deg v_3$, then the proposition is obvious from \ref{def:KPdeg}.

Assume first that the reduction is not normal, that is, $g_1 = f_1 + af_3 + Q(f_2)$.
We know by Corollary \ref{cor:SPF}
that $\deg g_1 = s\delta$, $\deg f_2 = 2\delta$ and $s\delta > \deg
f_3 > (s-2) \delta$, where $s \ge 3$ is odd.
An inequality $\deg g_1 > \deg f_1$ would imply $s\delta = \deg (g_1) = \deg
Q(f_2) = 2r\delta$ for some integer $r$ (the degree of $Q$), a contradiction
with $s$ odd.
Thus we obtain $\deg g_1 = \deg f_1 > \deg (af_3 +
Q(f_2))$, hence $\deg w_3 = \deg v_3$ and we are done.

Now assume we have a normal proper $K$-reduction, and that $\deg w_3 > \deg
v_3$.
By Proposition \ref{pro:foldable}\ref{foldable:3} (see also the discussion
just before Lemme \ref{lem:normal proper}), we are in the setting of Lemma
\ref{lem:normal proper}.
By condition \ref{def:KPdel} we have
$$\delta +  \deg dg_1 \wedge df_2 > \deg g_3 .$$
Adding $3\delta = \deg g_1$, and using (\ref{NP2}) from Lemma \ref{lem:normal proper}, we get
$$\deg df_1 \wedge df_3 \ge 4 \delta +  \deg dg_1 \wedge df_2 > \deg g_1 + \deg g_3.$$
Finally adding $\deg f_2$ we get
\begin{align*}
\deg v_3 &= \deg f_1 + \deg f_2 + \deg f_3 \\
&\ge \deg df_1 \wedge df_3 + \deg f_2 &\text{by Lemma \ref{lem:basics forms}}\\
&> \deg g_1  + \deg g_3 + \deg f_2 \\
&= \deg u_3.\qedhere
\end{align*}
\end{proof}

In the following result we prove that if a vertex admits a non-normal proper
$K$-reduction, then it already admits an elementary (and therefore normal)
$K$-reduction.
It follows that any vertex admitting a $K$-reduction admits a normal
$K$-reduction.

\begin{lemma}[Normalization of a $K$-reduction] \label{lem:normalization}
Let $u_3$ be a non-normal proper $K$-reduction of $v_3$, via $w_3$.
Then there exists $u_3'$ such that
\begin{enumerate}
\item $u_3' \ne v_3$ and $u_3' \ne u_3$;
\item $u_3'$ is an elementary (hence by definition normal) $K$-reduction of $v_3$.
\end{enumerate}
\end{lemma}

\begin{figure}[ht]
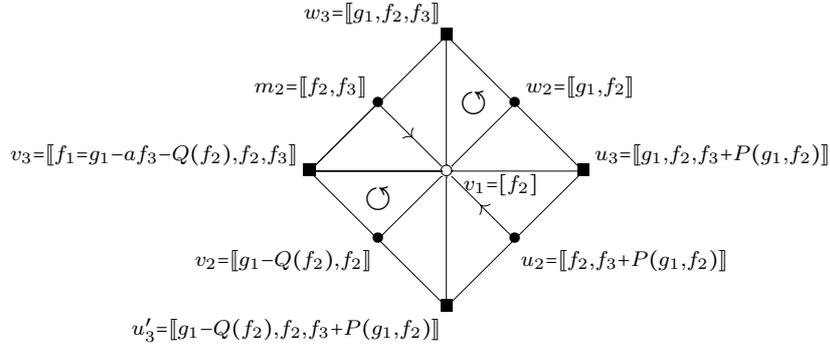

$$
\mygraph{
!{<0cm,0cm>;<0.9cm,0cm>:<0cm,0.9cm>::}
!{(-2,0)}*-{\typethree}="v3"
!{(0,2)}*-{\typethree}="w3"
!{(0,-2)}*-{\typethree}="u3'"
!{(2,0)}*-{\typethree}="u3"
!{(0,0)}*{\typeone}="v1"
!{(-1,1)}*-{\typetwo}="v2"
!{(-1,-1)}*-{\typetwo}="w2'"
!{(1,1)}*-{\typetwo}="w2"
!{(1,-1)}*-{\typetwo}="u2"
!{(.4,1)}*{\circlearrowleft}
!{(-1,-.4)}*{\circlearrowleft}
"w3"-_<{w_3 = \llb g_1, f_2, f_3 \rrb}_>{m_2 = \llb f_2, f_3 \rrb}"v2"-_>{v_3 = \llb f_1 = g_1 - af_3 - Q(f_2), f_2, f_3 \rrb}"v3"-_>{v_2 = \llb g_1 - Q(f_2), f_2 \rrb}"w2'"-_>{u_3' = \llb g_1- Q(f_2), f_2, f_3 + P(g_1,f_2)\rrb}"u3'"-"u2"-_<{u_2 = \llb f_2, f_3 + P(g_1,f_2)\rrb}"u3"-_<{u_3 = \llb g_1, f_2, f_3 + P(g_1,f_2)\rrb}"w2"-_<{w_2 = \llb g_1, f_2 \rrb}"w3"
"v1"-"v3"-"v2"
"u2"-|@{>}"v1"-"w2'"
"u3'"-"v1"-"w3"
"v3"-"v1"-"v3"
"v2"-|@{>}"v1"-_<{\,v_1 = [f_2]}"w2"
"v1"-"u3"
}
$$
\caption{Normalization of a $K$-reduction.}\label{fig:normal}
\end{figure}

\begin{proof}
By Corollary \ref{cor:SPF} the pivotal simplex of the reduction has Strong Pivotal Form $\PF(s)$ for some odd $s$, and by Lemma \ref{lem:normal proper} we have $s \ge 5$.
Note also that by Proposition \ref{pro:foldable} we have $\deg v_3 = \deg w_3$,
and $v_3$ is a weak simple elementary reduction of $w_3$ with simple center
$m_2, v_1$: see the upper-half of Figure \ref{fig:normal}, where we use the
notation of Set-Up \ref{setup:good_rep}.

By the Square Lemma \ref{lem:square}, we get the existence of $u_3'$ with $u_3'
\ne v_3$, $u_3' \ne u_3$ and $\deg w_3 > \deg u_3'$: see Figure
\ref{fig:normal}.
In particular $\deg v_3 > \deg u_3'$, which is \ref{def:Kdeg}.

Since $v_3$ and $w_3$ have the same $3$-degrees, the vertices $v_2$ and $w_2$ also have the same $2$-degrees.
So Properties \ref{def:KPinn} and \ref{def:KPout} for the initial proper $K$-reduction from $v_3$ to $u_3$ imply \ref{def:Kinn} and \ref{def:Kout} for the elementary reduction from $v_3$ to $u_3'$.

Finally $m_2$ is the minimal line of $v_3$, and is distinct from $v_2$, which gives \ref{def:Kmin}, and
 $v_2 = \llb g_1 - Q(f_2), f_2 \rrb$ so that
$d(v_2) = d(w_2)$, which gives \ref{def:Kdel}.
\end{proof}

\begin{corollary} \label{cor:degree of K}
If $v_3 = \llb f_1, f_2, f_3 \rrb$ admits a $K$-reduction,
then one of the following holds:
\begin{enumerate}
\item \label{degreeK:case1} Any line in $\lines{v_3}$ has no inner resonance;
\item \label{degreeK:case2} Up to permuting the $f_i$, we have $\deg f_1 =
2\deg f_3$ and $2\deg f_1 = 3\deg f_2$.
In particular with this numbering $[f_3]$ is the minimal vertex of $v_3$.
\end{enumerate}
\end{corollary}

\begin{proof}
Any line in $\lines{v_3}$ has $2$-degree $(\delta_1, \delta_2)$ with $\delta_i \in \{ \deg f_1, \deg f_2, \deg f_3\}$ for $i = 1,2$.
So there exists a line  in $\lines{v_3}$ with inner resonance if and only if there exist two indexes $i \neq j$ in $\{1,2,3\}$ such that $\deg f_i \in \N \deg f_j$.

If the $K$-reduction is elementary, using the notation from Set-Up
\ref{setup:good_rep}\ref{good_rep:1}, we have $\deg f_1 = s\delta$, $\deg f_2 =
2\delta$ and $s \delta > \deg f_3 > (s-2)\delta$ for some odd $s \ge 3$.
Moreover an inner resonance in $\llb f_2, f_3 \rrb$ would be of the form $\deg f_3 = (s-1) \delta = \frac{s-1}{2} \deg f_2$ for $s \ge 5$, but this is impossible by Property \ref{def:Kout}.
So the only possible resonance is between $\deg f_1$ and $\deg f_3$, in the case $s = 3$, as stated in \ref{degreeK:case2}.

If the $K$-reduction is proper, we use the notation from Set-Up
\ref{setup:good_rep}\ref{good_rep:2}.
Either $\deg g_1 = \deg f_1$ and we are reduced to the previous case; or $\deg
g_1 > \deg f_1$ and by Lemma \ref{lem:normalization} we can assume that the
$K$-reduction is normal (and proper, otherwise again we are reduced to the
previous case).
Then by Lemma \ref{lem:normal proper} we have $3\delta > \deg f_1 > \frac52 \delta$, $\deg f_2 = 2\delta$, $\deg f_3 = \frac32 \delta$, hence there is no relation of the form $\deg f_i \in \N \deg f_j$ for any $i \neq j$
and we are in case \ref{degreeK:case1}.
\end{proof}

\begin{remark}
We shall see later in Corollary \ref{cor:no inner resonance} that in fact Case \ref{degreeK:case2} in the previous corollary never happens.
\end{remark}

\subsection{Stability of $K$-reductions}

Consider $v_3$ a vertex that admits a normal $K$-reduction.
In this section we want to show that most elementary reductions of $v_3$ still admit a $K$-reduction.
First we prove two lemmas that give some constraint on the (weak) elementary reductions that such a vertex $v_3$ can admit.

\begin{lemma} \label{lem:=lemme14}
Let $u_3$ be a normal $K$-reduction of $v_3$, with pivot $v_1$.
Let $u_2$ be any line in $\lines{v_3}$ not passing through $v_1$.
Then $v_3$ does not admit a weak elementary reduction with center $u_2$.
\end{lemma}

\begin{proof}
We start with the notation $v_3 = \llb f_1, f_2, f_3 \rrb$ from Set-Up \ref{setup:good_rep}\ref{good_rep:1} when $u_3$ is an elementary $K$-reduction of $v_3$,
and with the Set-Up \ref{setup:good_rep}\ref{good_rep:2} to which we apply Lemma \ref{lem:normal proper} when $u_3$ is a normal proper $K$-reduction of $v_3$.
It follows that $m_2 = \llb f_2, f_3 \rrb$ is in both cases the minimal line of $v_3$. 
We have $u_2 = \llb h_1, h_3 \rrb$ where $h_1 = f_1 + af_2$, $h_3 = f_3 + bf_2$ for some $a, b \in \K$.
Then $[f_2, h_3]$ is the minimal line $m_2$ in $\lines{v_3}$, however it is possible that $[f_2, h_3]$ is not a good representative of $m_2$.
There are two possibilities:

\begin{itemize}[wide]
\item either $(h_1,f_2,h_3)$ is still a good representative for $v_3$, and we have $\deg h_3 = \deg f_3$, $\deg f_2 = \deg (v_3 \smallsetminus u_2)$;
\item or $\deg f_2 = \deg h_3 > \deg m_1$ where $m_1 = [f_3]$ is the minimal point of $v_3$, and $\deg f_2 > \deg (v_3 \smallsetminus u_2) = \deg f_3$.
\end{itemize}
In both cases we have
\begin{equation*} \label{eq:the true equation}
\deg h_1 = \topdeg v_3, \quad \deg f_2 \ge \deg (v_3 \smallsetminus u_2) \quad \text{and} \quad \deg h_3 \ge \deg f_3.
\end{equation*}

Assume by contradiction that $v_3$ admits a weak elementary reduction with center $u_2$.
By Lemma \ref{lem:elem red} there exists a non-affine polynomial $P\in
\K[y,z]$
such that
$$\deg (v_3 \smallsetminus u_2)\ge \deg P(h_1,h_3).$$
On the other hand we know from Corollary \ref{cor:SPF} that the pivotal simplex
of the $K$-reduction has Strong Pivotal Form $\PF(s)$ for some odd $s \ge 3$, hence
$$\deg h_3 \ge \deg f_3 > (s-2)\delta \ge \delta \;\text{ which implies }\;
2\deg h_3 > 2\delta =  \deg f_2.$$
In consequence, since $\deg h_1 > \deg h_3$, we have 
$$\dvirt P(h_1,h_3) \ge 2\deg h_3 > \deg f_2 \ge \deg (v_3 \smallsetminus u_2),$$ so that
$$\dvirt P(h_1,h_3) > \deg P(h_1,h_3).$$

If $u_2$ has no inner resonance, then we get a contradiction as follows, in both cases of an elementary or a normal proper $K$-reduction:

% ATTENTION : Proposition \ref{pro:SPF}\ref{SPF1} NE MARCHE PAS POUR LES WEAK,
% Il faut vraiment utiliser le parachute \ref{cor:parachute}\ref{parachute-ii}

\begin{align*}
\deg (v_3 \smallsetminus u_2) 
& \ge \deg P(h_1, h_3) \\
&> \deg dh_1 \wedge dh_3 = d(u_2) &\text{by Corollary
\ref{cor:parachute}\ref{parachute-ii}},\\
&> d(m_2) &\text{by Lemma \ref{lem:dm2}\ref{dm2:3} or (\ref{NP7}),}
\\
&> \deg(v_3 \smallsetminus m_2) &\text{by Lemma \ref{lem:dm2}\ref{dm2:1}.}
\end{align*}
More precisely, in the case of a normal proper $K$-reduction the last
inequality comes from $
d(m_2) > \deg (w_3 \smallsetminus m_2) \ge \deg(v_3 \smallsetminus m_2)$ by
Lemma \ref{lem:dm2}\ref{dm2:1} and by \ref{def:KPdegbis}.

Now consider the case where $u_2 = \llb h_1, h_3 \rrb$ has inner resonance.
By Corollary  \ref{cor:degree of K}\ref{degreeK:case2} we have $\deg h_3 = \min \{\deg f_2, \deg f_3\}$, and since by assumption $\deg h_3 \ge \deg f_3$ we get $\deg h_3 = \deg f_3$.
Then Corollary  \ref{cor:degree of K}\ref{degreeK:case2} gives the two relations
\begin{align} \label{eq:v3-m2}
\begin{split}
\tfrac12 \deg(v_3 \smallsetminus m_2) &= \tfrac12 \deg h_1 = \deg h_3, \\
\tfrac23 \deg(v_3 \smallsetminus m_2) &= \tfrac23 \deg h_1 = \deg f_2 \ge \deg
(v_3 \smallsetminus u_2).
\end{split}
\end{align}
In particular we have $\deg f_1 > \deg f_2 > \deg f_3$, and $b=0$, that is, $h_3
= f_3$.
We apply Corollary \ref{cor:parachute}\ref{parachute-i} which gives
$$\deg (v_3 \smallsetminus u_2) \ge \deg P(h_1,h_3) \ge d(u_2) - \deg h_3,$$
which we rewrite as
\begin{equation}\label{eq:intermediate}
\deg h_3 + 2\deg (v_3 \smallsetminus u_2) \ge \deg (v_3 \smallsetminus u_2) +
d(u_2).
\end{equation}
If $u_3$ is an elementary $K$-reduction of $v_3$, and since in our
situation $\deg (v_3 \smallsetminus u_2) = \deg [f_2]$, by Lemma
\ref{lem:dm2}\ref{dm2:2} we have
\begin{equation}\label{eq:dm2}
\deg (v_3 \smallsetminus u_2) + d(u_2) > 2 \deg (v_3 \smallsetminus m_2).
\end{equation}
If on the other hand $u_3$ is a proper $K$-reduction of $v_3$ via $w_3$,
let us prove that (\ref{eq:dm2}) still holds, by using Lemma
\ref{lem:normal proper}.
First note that $\deg (v_3 \smallsetminus u_2) = \deg f_2$, $\deg (v_3 \smallsetminus m_2) = \deg f_1$ and
$$d(u_2) = \deg dh_1 \wedge dh_3 = \deg d(f_1+af_2)\wedge df_3 = \deg df_1 \wedge df_3 \text{ by (\ref{NP2})}.$$
Then, using (\ref{NP1}) and (\ref{NP2}) from Lemma
\ref{lem:normal proper}, we get 
$$\deg f_2 + \deg df_1 \wedge df_3 > 2\delta + 4 \delta = 2\deg g_1 \ge 2 \deg f_1$$
as expected.

Adding the first equality of (\ref{eq:v3-m2}) to twice the second one, and combining with (\ref{eq:intermediate}) and (\ref{eq:dm2}),  we get
the contradiction
\begin{equation*}
(\tfrac12 + \tfrac{2.2}3) \deg(v_3 \smallsetminus m_2) > 2\deg(v_3 \smallsetminus
m_2).\qedhere
\end{equation*}
\end{proof}

\begin{lemma} \label{lem:=lemme15}
Let $u_3$ be a normal proper $K$-reduction of $v_3$, with pivot $v_1$.
Let $v_2' \neq m_2$ be a line in $\lines{v_3}$ passing through $v_1$.
If $v_3'$ is a weak elementary reduction of $v_3$ with center $v_2'$, then this
reduction is simple with center $v_2',v_1$.
\end{lemma}

\begin{proof}
We use the notation from Set-Up \ref{setup:good_rep}\ref{good_rep:2}, and set
$v_2' = \llb h_1,
f_2 \rrb$ with $h_1 = f_1 + a f_3$.
By Lemma \ref{lem:normal proper}, $h_1$ realizes the top degree of $v_3$.
Then $v_3 = \llb f_1, f_2, f_3 \rrb = \llb h_1, f_2, f_3 \rrb$, and by Lemma
\ref{lem:elem red} we have  $v_3' = \llb h_1, f_2, f_3 +P(h_1, f_2)\rrb$ for
some non-affine polynomial $P$.
We want to prove that $P(h_1, f_2) \in \K[f_2]$.
It is sufficient to prove $\deg h_1 > \dvirt P(h_1,f_2)$.
Assume the contrary. Then
$$\dvirt P(h_1,f_2) \ge \deg h_1 > \deg f_3 \ge \deg P(h_1,f_2).$$
By Lemma \ref{lem:normal proper}, we have $\deg f_1 > \deg f_2 > \deg f_3$, so
that by Corollary \ref{cor:degree of K} we have $\deg h_1 = \deg f_1 \not\in
\N\deg f_2$.
Thus we can apply Corollary \ref{cor:parachute}\ref{parachute-ii} to
get
$$\deg f_1 > \deg P(h_1,f_2) > \deg dh_1 \wedge df_2.$$
By (\ref{NP3}) of Lemma \ref{lem:normal proper} we get
$$\deg dh_1 \wedge df_2 = \deg (df_1 \wedge df_2 - adf_2 \wedge df_3)
= \deg df_1 \wedge df_2 > \deg df_2 \wedge df_3.$$
Then by (\ref{NP2}) and (\ref{NP1}) of Lemma \ref{lem:normal proper}
we have
$$\deg df_2 \wedge df_3 \ge 3\delta = \deg g_1 \ge \deg f_1,$$
hence the contradiction $\deg f_1 > \deg dh_1 \wedge df_2 \ge \deg
f_1$.
\end{proof}

\begin{proposition}[Stability of a $K$-reduction]
\label{pro:K stability}
Let $u_3$ be a normal $K$-reduction of $v_3$, and $v_3'$ a weak elementary
reduction of $v_3$ with center $v_2'$.
Denote by $m_2$ the minimal line of $v_3$.
If $u_3$ is an elementary $K$-reduction, assume moreover that the centers of
$v_3' \ne v_3$ and $u_3 \ne v_3$  are distinct.
Then $u_3$ is a $K$-reduction of $v_3'$, and more precisely, we are in one
of the following cases:
\begin{enumerate}
\item \label{Kstability:1} $u_3$ is an elementary $K$-reduction of $v_3$,
and $v_2' = m_2$: then $u_3$ is a (possibly
non-normal) proper $K$-reduction of $v_3'$, via $v_3$;

\item \label{Kstability:2} $u_3$ is an elementary $K$-reduction of $v_3$,
and $v_2' \neq m_2$: then $u_3$ is a (possibly
non-normal) proper $K$-reduction of $v_3'$, via an auxiliary vertex $w_3'$ that
satisfies $\deg w_3' = \deg v_3$;

\item \label{Kstability:3} $u_3$ is a normal proper $K$-reduction of $v_3$ via
$w_3$, and $v_3' = w_3$: then $u_3$ is an elementary $K$-reduction of
$v_3'$;

\item \label{Kstability:4} $u_3$ is a normal proper $K$-reduction of $v_3$ via
$w_3$, and $v_2' = m_2$: then $u_3$
also is a normal proper $K$-reduction of $v_3'$ via $w_3$.
\end{enumerate}
\end{proposition}

\begin{proof}
First assume that $u_3$ is an elementary $K$-reduction of $v_3$.
We denote by $v_1 = [ f_2 ]$,
$v_2 = \llb f_1,f_2 \rrb$, $v_3 = \llb f_1, f_2, f_3 \rrb$ the pivotal simplex
of the $K$-reduction $u_3$ (following Set-Up
\ref{setup:good_rep}\ref{good_rep:1}).
By Lemma \ref{lem:=lemme14}, the line $v_2'$ passes through $v_1$.

\ref{Kstability:1}. If $v_2' = m_2$, since by
assumption $\deg v_3 \ge \deg v_3'$, we directly get that $u_3$ is a proper
$K$-reduction of $v_3'$, via $v_3$.

\ref{Kstability:2}. Now assume that $v_2' \neq m_2$, so
that $v_2' =\llb f_1 + af_3, f_2 \rrb$ for some $a \in \K$, and $a \neq 0$
since we assume $v_2' \neq v_2$.
Then by Lemma \ref{lem:elem red} we can write
$$v_3' =  \llb f_1 + af_3, f_2, f_3 +P( f_1 + af_3, f_2) \rrb \text{ with
} \deg f_3 \ge \deg P( f_1 + af_3, f_2).$$
If we can show that $P$ depends only on $f_2$ we are done: indeed then
$\deg f_3 \neq \deg P(f_2)$, because $m_2 = \llb f_2, f_3 \rrb$ has no inner
resonance by Corollary \ref{cor:degree of K}, hence we have $\deg f_3 =
\deg (f_3 + P(f_2))$.
It follows that $u_3$ is
a proper $K$-reduction of $v_3'$ via $w_3' = \llb f_1, f_2, f_3 + P(f_2) \rrb$,
where $m_2' = \llb f_2, f_3 + P(f_2) \rrb$ is the minimal line of $w_3'$ (see
Figure \ref{fig:stability}, Case \ref{Kstability:2}).

To show that $P$ depends only on $f_2$ it is sufficient to show that $\deg f_3
\ge \dvirt  P( f_1 + af_3, f_2)$.
By contradiction, assume that this is not the case.
Then
$$\dvirt  P( f_1 + af_3, f_2) > \deg f_3 \ge \deg  P( f_1 + af_3, f_2).$$
Since $v_2'$ has the same 2-degree as $v_2$, it has no inner resonance by
\ref{def:Kinn}, and by Corollary \ref{cor:parachute}\ref{parachute-ii} we get
$$\deg  P( f_1 + af_3, f_2) > \deg (df_1 \wedge df_2 - adf_2 \wedge df_3).$$
By Lemma \ref{lem:dm2}\ref{dm2:1} we have $\deg df_2 \wedge df_3 > df_1 \wedge
df_2$ and  $\deg df_2 \wedge df_3 > \deg f_3$, so finally we obtain the
contradiction
$$\deg  P( f_1 + af_3, f_2) > \deg df_2 \wedge df_3 > \deg f_3.$$

\begin{figure}[t]
$$
\xymatrix@R-10px{
\mygraph{
!{<0cm,0cm>;<0.76cm,0cm>:<0cm,0.76cm>::}
!{(-2,-1)}*{\typethree}="v3"
!{(2,-1)}*{\typethree}="v3'"
!{(0,-1)}*{\typeone}="v1"
!{(0,1)}*-{\typethree}="w"
!{(-1,0)}*-{\typetwo}="f2f3"
!{(1,0)}*-{\typetwo}="g1g2"
!{(.4,0)}*{\circlearrowleft}
"f2f3"-_>{v_3'}"v3"-_>{v_1 = [ f_2
]}"v1"-"f2f3"-^<{v_2' = m_2 = \llb f_2,f_3 \rrb}^>(1.1){v_3 =\llb f_1,f_2,f_3
\rrb}"w"-"v1"-"v3'"-_<{u_3}_>{v_2 = \llb f_1,f_2
\rrb}"g1g2"-"w"
"v1"-"g1g2"
}
&
\mygraph{
!{<0cm,0cm>;<1cm,0cm>:<0cm,0.65cm>::}
!{(-1,-1)}*{\typethree}="v3'"
!{(2.2,0.5)}*{\typethree}="u3"
!{(0,2)}*-{\typethree}="v3"
!{(1,-1)}*{\typethree}="w3"
!{(0,0)}*{\typeone}="v1"
!{(-1,1)}*-{\typetwo}="v2'"
!{(1,1)}*{\typetwo}="v2"
!{(0,-2)}*{\typetwo}="m2"
!{(.4,1)}*{\circlearrowleft}
!{(.65,-.2)}*{\circlearrowleft}
"v3'"-^<{v_3' =  \llb f_1 + af_3, f_2, f_3 +P(f_2) \rrb}^>{v_2'=\llb f_1 + af_3,
f_2 \rrb}"v2'"-^>{v_3= \llb f_1, f_2, f_3 \rrb\quad}"v3"-"v2"-^<{\quad v_2= \llb
f_1,f_2 \rrb}^>{u_3}"u3"-"v1"-"v3'"
"v2'"-_>(.8){v_1}|@{>}"v1"-"v2"
"v3"-"v1"
"v3'"-_>{m_2'= \llb f_2, f_3 + P(f_2) \rrb\quad}"m2"-_>{\quad w_3'= \llb f_1,
f_2, f_3 + P(f_2) \rrb}"w3"-"v2" "w3"-"v1"-|@{<}"m2"
}
\\
\text{Case } \ref{Kstability:1} & \text{Case } \ref{Kstability:2} \\
\mygraph{
!{<0cm,0cm>;<0.76cm,0cm>:<0cm,0.76cm>::}
!{(-2,-1)}*{\typethree}="v3"
!{(2,-1)}*{\typethree}="v3'"
!{(0,-1)}*{\typeone}="v1"
!{(0,1)}*-{\typethree}="w"
!{(-1,0)}*-{\typetwo}="f2f3"
!{(1,0)}*-{\typetwo}="g1g2"
!{(.4,0)}*{\circlearrowleft}
"f2f3"-_>{v_3}"v3"-_>{v_1}"v1"-"f2f3"-^<{v_2' =
m_2}^>(1.1){v_3' = w_3}"w"-"v1"-"v3'"-_<{u_3}_>{w_2}"g1g2"-"w"
"v1"-"g1g2"
}
&
\mygraph{
!{<0cm,0cm>;<0.76cm,0cm>:<0cm,0.76cm>::}
!{(-2.5,.4)}*{\typethree}="v3'"
!{(2,-1)}*{\typethree}="v3''"
!{(0,1)}*-{\typethree}="v3"
!{(-2,-1)}*{\typethree}="w'"
!{(-1,0)}*{\typetwo}="v2NO"
!{(1,0)}*-{\typetwo}="v2NE"
!{(0,-1)}*{\typeone}="v1"
!{(.4,0)}*{\circlearrowleft}
"w'"-^<{v_3}^>{w_3}"v3"-^>{u_3}"v3''"
"v2NO"-_<{v_2' = m_2\qquad}_>(1.2){v_3'}"v3'"
"v2NO"-"v1" "v2NE"-"v1"
"w'"-_>{v_1}"v1"-"v3''" "v3"-"v1"
"v3'"-"v1"
}
\\
\text{Case } \ref{Kstability:3} & \text{Case } \ref{Kstability:4}
}
$$
\caption{Stability of a $K$-reduction.} \label{fig:stability}
\end{figure}
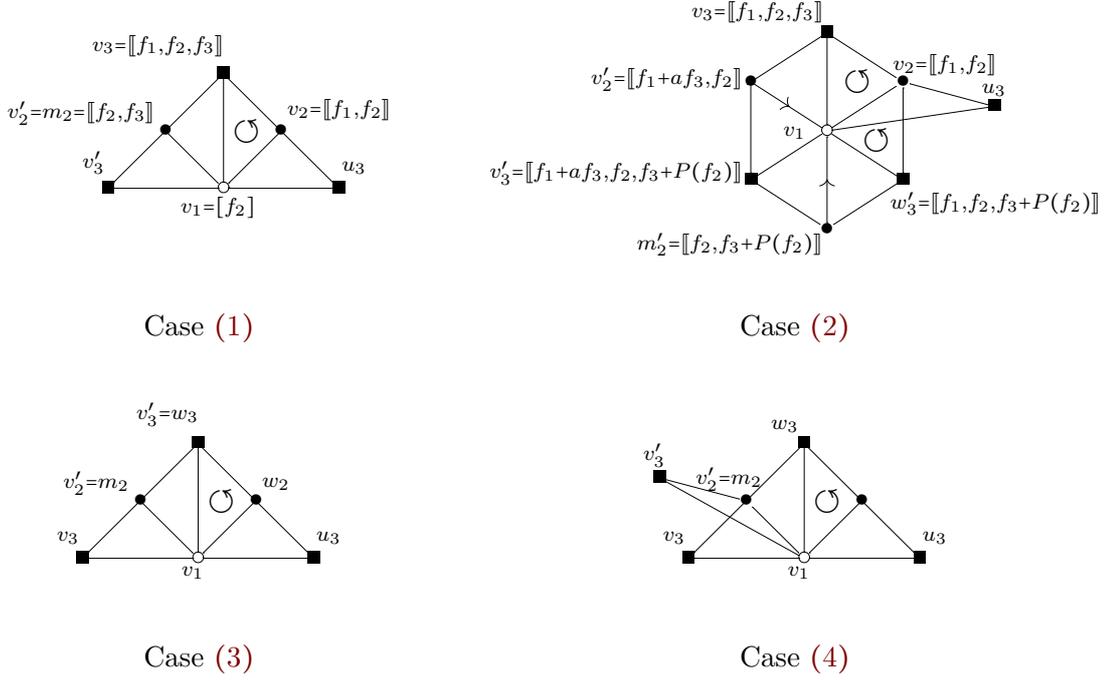

Now assume  that $u_3$ is a normal proper  $K$-reduction of $v_3$, via an
auxiliary vertex $w_3$.
Recall that the minimal line $m_2$ of $v_3$ is also
the minimal line of the intermediate vertex $w_3$ (last assertion of Lemma
\ref{lem:normal proper}).

\ref{Kstability:3}. If $v_3' = w_3$, then by
definition $u_3$ is an elementary $K$-reduction of $v_3'$.

\ref{Kstability:4}. If $v_3' \neq w_3$, but $v_2' = m_2$ , then the conclusion
is also direct, because $\deg w_3 \ge \deg v_3 \ge \deg v_3'$.

Finally we prove that the situation where $v_2' \neq m_2$ leads to a
contradiction.
By Lemma \ref{lem:=lemme15} the reduction from $v_3$ to $v_3'$ is
simple with center $v_2', v_1$.
But then by Remark \ref{rem:simple} (or directly from the proof of Lemma
\ref{lem:=lemme15}) we should have $v_2' = \llb f_1+a f_3, f_2 \rrb$
and $v_3' = \llb f_1+ a f_3, f_2, f_3 + P(f_2) \rrb$.
By Lemma \ref{lem:normal proper} we have
$$\deg f_2 = 2\delta > \tfrac32 \delta \ge \deg f_3,$$
so $\deg P(f_2) > \deg f_3$ and we get a contradiction with $\deg v_3 \ge \deg
v_3'$.
\end{proof}

\section{Reducibility Theorem} \label{sec:reducibility theorem}

In this section we state and prove the main result of this paper, that is, the Reducibility Theorem \ref{thm:reducibility}.

\subsection{Reduction paths}

Given a vertex $v_3$ with a choice of good triangle $T$, we call elementary
$T$-reduction any elementary reduction with center one of the three lines of
$T$.

We now define the notion of a \textbf{reducible} vertex in a recursive manner as follows:
\begin{itemize}[wide]
\item We declare that the vertex $[\id]$ is reducible, where by Lemma \ref{lem:id is min} $[\id]$ is the unique type 3 vertex realizing the minimal degree $(1,1,1)$.
\item Let $\mu > \nu$ be two consecutive degrees, and assume that we have already defined the subset of reducible vertices among type 3 vertices of degree at most $\nu$.
Then we say that a vertex $v_3$ with $\deg v_3 = \mu$ is reducible if for any
good triangle $T$ in $\lines{v_3}$, there exists either a $T$-elementary
reduction or a (proper or elementary) $K$-reduction  from $v_3$ to $u_3$, with
$u_3$ reducible.
\end{itemize}

Let $v_3$, $v_3'$ be vertices of type 3.
A \textbf{reduction path} of length $n \ge 0$ from $v_3$ to $v_3'$ is a sequence
of type 3 vertices $v_3(0), v_3(1), \dots, v_3(n)$ such that:
\begin{itemize}
\item $v_3(0) = v_3$ and $v_3(n) = v_3'$;
\item $v_3(i)$ is reducible for all $i = 0, \dots, n$;
\item For all $i = 0, \dots, n-1$, $v_3(i+1)$ is either an elementary reduction,
or a $K$-reduction, of $v_3(i)$.
\end{itemize}

Observe that, by definition, a reducible vertex $v_3$ admits a reduction path from $v_3$ to the vertex $[\id]$.

In the following sections we shall prove the main result:

\begin{theorem}[Reducibility Theorem] \label{thm:reducibility}
Any vertex of type 3 in the complex $\Comp$ is reducible.
\end{theorem}

Directly from the definition, this theorem has the following consequence: For
any vertex $v_3 \neq [\id]$ of type 3 in the complex $\Comp$, and for any good
triangle $T$ in $\lines{v_3}$, the vertex $v_3$ admits either a $T$-elementary
reduction  or a (proper or elementary) $K$-reduction.

We remark that this result immediately implies that $\Tame(\A^3)$ is a proper
subgroup of $\Aut(\A^3)$:

\begin{corollary} \label{cor:nagata}
The Nagata's automorphism
$$f = (x_1 + 2x_2(x_2^2 - x_1x_3) + x_3(x_2^2 - x_1x_3)^2,\; x_2 + x_3(x_2^2 -
x_1x_3),\; x_3)$$
is not tame.
\end{corollary}

\begin{proof}
Denote $f =(f_1, f_2, f_3)$ the components of $f$.
Assume that $f$ is tame.
Let $v_3 = \llb f_1, f_2, f_3 \rrb$ be the associated
vertex in $\Comp$, and let $T$ be the good triangle associated with this
representative.
We have $\deg f_1 = (2,0,3)$, $\deg f_2 = (1, 0 ,2)$ and $\deg f_3 = (0,0,1)$.

On the one hand, if $f$ admits a $K$-reduction, by Corollary \ref{cor:SPF} one
of the $f_i$ (the pivot of the reduction) should have a degree of the form
$2\delta$: this is not the case.

On the other hand, the degrees of the $f_i$ are pairwise $\Z$-independent, so
for any distinct $i,j \in \{1,2,3\}$ and any polynomial $P$ we have $\dvirt
P(f_i,f_j) = \deg P(f_i,f_j)$.
This implies that if $f$ admits an elementary $T$-reduction, then one of the
$\deg f_i$ should be a $\N$-combination of the other two. Again this is not the
case.

Thus $v_3$ is not reducible, a contradiction.
\end{proof}

We shall prove Theorem \ref{thm:reducibility} in \S\ref{sec:the proof}.
In the next two sections we establish preliminaries technical results.

\subsection{Reduction of a strongly pivotal simplex} \label{sec:SP}

First we describe the set-up that we shall use in this section.

\begin{setup} \label{setup:pivotal}
Let $v_1, v_2, v_3$ be a simplex in $\Comp$ with Strong Pivotal Form $\PF(s)$
for some odd $s \ge 3$.
We choose some good representatives $v_1 =[ f_2 ]$, $v_2 = \llb f_1, f_2
\rrb$ and $v_3 = \llb f_1, f_2, f_3 \rrb$.
Condition \ref{SFinn} means that
$$ \deg f_1 = s\delta, \quad \deg f_2 = 2\delta.$$
By \ref{SFmin} we have $\deg f_1 > \deg f_3$.
 we have
$$ \deg f_3 \ge (s-2) \delta + \deg df_1 \wedge df_2 .$$
Condition \ref{SFout} is equivalent to the condition
$$\deg f_3 \not\in \N\deg f_2.$$
Since $\deg f_3 > (s-2)\delta \ge \delta$, we also obtain
\begin{equation}\label{degf2}
\deg f_3^2 > \deg f_2
\;\text{ and }\;
\deg f_2 \not\in \N \deg f_3.
\end{equation}
In particular, as already noticed in Lemma \ref{lem:m2 in SPF}, the minimal line
$m_2 = \llb f_2, f_3 \rrb$ of $v_3$
has no inner resonance.
Observe also that
\begin{equation}\label{degf1}
\deg f_1 \not\in \N \deg f_3 \text{ except if } s = 3 \text{ and } \deg f_1 = 2
\deg f_3.
\end{equation}
\end{setup}

\begin{lemma} \label{lem:no proper K reduction}
Assume Set-Up $\ref{setup:pivotal}$.
Then $v_3$ does not admit a normal proper $K$-reduction.
\end{lemma}

\begin{proof}
Assume $v_3$ admits a normal proper $K$-reduction, via $w_3$.
Then we get a contradiction as follows:
\begin{align*}
d(m_2) &> \deg(w_3 \smallsetminus m_2) &&\text{by Lemma
\ref{lem:dm2}\ref{dm2:1}}\\
 &\ge \deg(v_3 \smallsetminus m_2) &&\text{by \ref{def:KPdegbis}} \\
 &> \deg(v_3 \smallsetminus v_2) \ge \Delta(v_2) && \text{by \ref{SFmin} and
\ref{SFdel}}\\
 &> d(v_2) > d(m_2) &&\text{by (\ref{NP7}) in Lemma \ref{lem:normal
proper}.}\qedhere
\end{align*}
\end{proof}

\begin{lemma} \label{lem:no outer reduction}
Assume Set-Up $\ref{setup:pivotal}$, and $\deg f_1 \neq 2 \deg f_3$.
Assume that $v_3$ admits a weak elementary reduction $v_3'$ with center $v_2'$.
Then $v_2'$ passes through $v_1$.
\end{lemma}

\begin{proof}
By contradiction, assume that $v_2'$ does not pass through $v_1$.
Recall that $v_1= [ f_2 ]$ and $v_2 = \llb f_1, f_2 \rrb$ with $\deg
f_1 > \deg f_2$.
Up to replacing $f_1$ by $f_1 + af_2$ for some $a \in \K$, we can assume  $v_2
\cap v_2' = [ f_1 ]$ while keeping all the properties stated in Set-Up
\ref{setup:pivotal}.
Then let $[ h_3 ] = [f_3 + af_2]$ be the intersection of $v_2'$ with the
minimal line $\llb f_2, f_3 \rrb$ of $\lines{v_3}$.
Then we have $v_2' = \llb f_1, h_3 \rrb$,  $v_3 = [f_1, f_2, h_3]$, and by
Lemma \ref{lem:elem red} $v_3' = \llb f_1, f_2+\phi_2(f_1,h_3), h_3 \rrb$ for
some non-affine polynomial $\phi_2$, with
$\deg f_1 > \deg f_2 \ge \deg{\phi_2}(f_1,h_3)$.
We have either $\deg f_2 = \deg
h_3$ or $\deg f_3 = \deg h_3$, hence in any case by (\ref{degf2})
$$\dvirt \phi_2(f_1,h_3) > \deg \phi_2(f_1,h_3).$$
In particular $\phi_2(f_1,h_3) \not\in \K[h_3]$, and then the inequality $\deg
f_1 > \deg f_2$ implies
$$\dvirt \phi_2(f_1,h_3) > \deg f_2.$$
By Lemma \ref{lem:p and q} there exist coprime $q > p$ such that
$$q \deg h_3 = p \deg f_1 = ps\delta.$$
Moreover if $p = 1$, we would have $\deg h_3 = \deg f_3$ and $\deg f_1 = 2 \deg
f_3$, in contradiction with our assumption. Hence we have $q > p \ge 2$.

Observe that even if $(f_2, h_3)$ is not a good representative of the minimal
line in $\lines{v_3}$, in any case we have $\deg h_3 \ge \deg(v_3 \smallsetminus
v_2) =   \deg f_3$, and Property \ref{SFdel} gives
$$\deg h_3 \ge (s-2)\delta + \deg df_1 \wedge df_2.$$
Then Corollary \ref{cor:parachute h + phi} yields:
\begin{align*}
2\delta = \deg f_2 \ge \deg (f_2 + \phi_2(f_1,h_3)) &\ge q\deg h_3 -
\deg df_1\wedge df_2 - \deg h_3 \\
& \ge q\deg h_3 - \deg h_3 + (s-2)\delta - \deg h_3.
\end{align*}
Multiplying by $q$ and replacing $q\deg h_3 = ps\delta$ we get:
$$0 \ge (pqs - 2ps + sq-4q )\delta,$$
hence
$$0 \ge ps (q-2) + q(s-4).$$
This implies $s = 3$, and we get the contradiction:
\begin{equation*}
0 \ge 3pq - 6p - q = (3p - 1)(q-2) -2 \ge 5 - 2.\qedhere
\end{equation*}
\end{proof}

\begin{lemma} \label{lem:no minimal reduction}
Assume Set-Up $\ref{setup:pivotal}$.
Assume that $v_3$ admits an elementary reduction $v_3'$ with center $m_2$, the
minimal line of $v_3$.
Assume moreover that $v_3'$ is reducible.
Then $v_3'$ also admits an elementary reduction with center $m_2$.
\end{lemma}

\begin{proof}
By Lemma \ref{lem:elem red} we can write $v_3' = \llb f_1', f_2, f_3 \rrb$,
where $f_1'$ has the form
$$f_1' = f_1+\phi_1(f_2,f_3)$$
for some non-affine polynomial $\phi_1$, with $\deg f_1 = \deg
\phi_1(f_2,f_3) >
\deg f_1'$.
Without loss in generality we can assume that $\phi_1$ has no constant term.
Moreover we can also assume
\begin{equation} \label{eq:f1'notNf_2+Nf_3}
\deg f_1' \not\in \N\deg f_2 + \N \deg f_3,
\end{equation}
otherwise the result is immediate.

Working with the good triangle associated with the representative $(f_1', f_2,
f_3)$, we want to prove that $v_3'$ does not admit a $K$-reduction, nor an
elementary reduction with center $\llb f_1', f_2 \rrb$ or $\llb f_1', f_3 \rrb$:
Indeed since $v_3'$ is reducible by assumption, the only remaining possibility
will be that $v_3'$ admits an elementary reduction with center $m_2 = \llb f_2,
f_3 \rrb$, as expected.
The proof is quite long, so we prove several facts along the way.
The first one is:

\begin{fact} \label{fact:dvirt phi1}
If $\dvirt \phi_1(f_2,f_3) > \deg \phi_1(f_2,f_3)$ then Lemma $\ref{lem:no
minimal reduction}$ holds.
\end{fact}

\begin{proof}
If $\dvirt \phi_1(f_2,f_3) > \deg \phi_1(f_2,f_3)$ then by Lemma \ref{lem:p and
q}, there exist coprime $p,q$  such that $q \deg f_2 = p\deg f_3$.
Observe that \ref{SFout} and (\ref{degf2}) imply $p,q \neq 1$.
We have
\begin{align*}
s\delta = \deg f_1 >\deg f_1' & = \deg (f_1 + \phi_1(f_2,f_3)) &&\text{by Lemma
\ref{lem:elem red}}\\
 &> p\deg f_3 - \deg df_1\wedge df_2 - \deg f_3 &&\text{by Corollary
\ref{cor:parachute h + phi}}\\
 &\ge p\deg f_3 -(\deg f_3 -(s-2)\delta)) - \deg f_3 &&\text{by \ref{SFdel}} \\
 &= (p-2)\deg f_3 - 2\delta + s\delta.
\end{align*}
Multiplying by $p$, recalling that $\deg f_2 = 2\delta$, $p\deg f_3 = q \deg
f_2$ and putting $\delta$ in factor we get:
\begin{equation} \label{eq:2q(p-2)}
0 > 2q(p-2) - 2p
  = (2p-4)(q-1) - 4
  \ge 2p -8.
\end{equation}
It follows that $3 \ge p$.
Now we deduce $p = 3$.
If $p = 2$, then $\deg f_3 = q\delta$.
Condition \ref{SFdel} gives
$$s\delta > \deg f_3 > (s-2)\delta,$$
so $q = s-1$, which contradicts $q$ coprime with $2$.

Replacing $p = 3$ in the first inequality of (\ref{eq:2q(p-2)}) we get $6 > 2q$,
hence $q = 2$.
We obtain $\deg f_3 = \frac{4}{3}\delta$, and the condition $\deg f_3 >
(s-2)\delta$ yields $s = 3$.
Finally
\begin{equation} \label{eq:degf1,f2,f3,f1'}
\deg f_1 = 3\delta, \quad \deg f_2 = 2\delta, \quad \deg f_3 = \tfrac{4}{3}
\delta, \quad 3\delta > \deg f_1' > \tfrac{7}{3} \delta.
\end{equation}

First we observe that these values are not compatible with $v_3'$ admitting a
$K$-reduction.
Indeed, by Corollary \ref{cor:SPF} an elementary $K$-reduction would imply $2
\deg f_1' = s' \deg f_j$ for some odd integer $s' \ge 3$ and $j \in \{2,3\}$,
and one checks from
(\ref{eq:degf1,f2,f3,f1'}) that there is no such relation.
Indeed if $s' = 3$, we have $2\deg f_1 > \frac{14}{3}\delta > 4\delta = s'\deg
f_3$, and if $s' \ge 5$ we have $s' \deg f_3 \ge \frac{20}{3} \delta > 6 \delta
> 2 \deg f_1'$.
Finally, for any $s' \ge 3$ we have $s' \deg f_2 \ge 6 \delta > 2 \deg f_1'$.

Now if $v_3'$ admits a normal proper $K$-reduction, then, noting that $\deg
f_1' > \deg f_2 > \deg f_3$, there should exist $\delta' \in \N^3$ such that all
the conclusions of Lemma \ref{lem:normal proper} hold, with $f_1', f_2, f_3,
\delta'$ instead of $f_1, f_2, f_3, \delta$.
In particular (\ref{NP4}) gives $\deg f_2 = 2 \delta'$, hence $\delta' =
\delta$.
Now since $3\delta > \deg f_1'$, by (\ref{NP4}) and (\ref{NP5}) from Lemma
\ref{lem:normal proper} would imply $\deg f_3 = \frac32 \delta$, incompatible
with $\deg f_3 = \tfrac{4}{3} \delta$.

On the other hand, if $v_3'$ admits an elementary reduction with center $\llb
f_1', f_j \rrb$ with $j = 2$ or $3$,
then, denoting by $k$ the integer such that $\{j,k\} = \{2,3\}$, there would
exist a non-affine polynomial $\phi$ such that $\deg f_k > \deg (f_k +
\phi(f_1', f_j))$, and in particular
$\deg f_1' > \deg f_k = \deg \phi(f_1', f_j)$.
Since $\llb f_2, f_3 \rrb$ has no inner resonance this implies
$\phi(f_1', f_j) \not\in \K[f_j]$, so that $\dvirt \phi(f_1', f_j) \ge \deg
f_1'$, and finally $ \dvirt \phi(f_1', f_j) > \deg \phi(f_1', f_j)$.
Now by (\ref{eq:f1'notNf_2+Nf_3}) we can apply
Corollary \ref{cor:parachute}\ref{parachute-iii} to get  an odd
integer $q'\ge 3$ such that $2 \deg f'_1 = q' \deg f_j$: again
this is not compatible with (\ref{eq:degf1,f2,f3,f1'}).
\end{proof}

From now on we assume $\dvirt \phi_1(f_2,f_3) = \deg \phi_1(f_2,f_3)$.

\begin{fact}
\label{fact:deg f1}
$\deg f_1 = 2\deg f_3$.
\end{fact}

\begin{proof}
By contradiction, assume $\deg f_1 \neq 2\deg f_3$.
Then $\deg f_1 \not\in \N \deg f_3$ by (\ref{degf1}).
Moreover we know that $\deg f_1 \not\in \N \deg f_2$ and $\deg f_2 + \deg f_3 >
\deg f_1$.
This is not compatible with the equalities
\begin{equation*}
\deg f_1 = \deg \phi_1(f_2,f_3) = \dvirt \phi_1(f_2,f_3). \qedhere
\end{equation*}
\end{proof}

We deduce from (\ref{degf1}) and Fact \ref{fact:deg f1} that $s = 3$, so that
$$\deg f_1 = 3\delta, \quad \deg f_2 = 2\delta, \quad \deg f_3 = \tfrac{3}{2}
\delta,$$
and there exist $a,c,e \in \K$ such that (recall that $\phi_1$ has no constant
term):
\begin{equation}
\label{eq:phi1}
\phi_1(f_2,f_3) = af_3^2 + cf_3 + ef_2 \text{ with } a\neq 0.
\end{equation}

Now come some technical facts.

\begin{fact}
\label{fact:df1 df3}
$\deg df_1 \wedge df_3 = \deg df_1' \wedge df_3 = \delta + \deg df_2 \wedge df_3
.$
\end{fact}

\begin{proof}
Recall from Set-Up \ref{setup:pivotal} that we have $\frac{3}{2}\delta = \deg
f_3 > \deg df_1 \wedge df_2 $, so
$$3\delta > \deg f_3 + \deg   df_1 \wedge df_2.$$
Since $\deg f_1 = 3\delta$ we get
$$\deg f_1 + \deg   df_2 \wedge df_3 > \deg f_3 + \deg   df_1 \wedge df_2.$$
By the Principle of Two Maxima \ref{pro:ptm} we have
$$\deg f_2 + \deg   df_1 \wedge df_3 = \deg f_1 + \deg   df_2 \wedge df_3.$$
Passing $\deg f_2$ to the right-hand side we get one of the expected
equalities
\begin{equation*}
\deg   df_1 \wedge df_3 = \deg f_1 - \deg f_2 + \deg   df_2 \wedge df_3 = \delta + \deg df_2 \wedge df_3.
\end{equation*}
From (\ref{eq:phi1}) we get
$$df_1' \wedge df_3 =  df_1 \wedge df_3 + e\,df_2\wedge df_3.$$
By the previous equality we obtain $\deg  df_1 \wedge df_3 = \deg df_1' \wedge df_3$.
\end{proof}

\begin{fact}
\label{fact:df1' df2}
$\deg df_1' \wedge df_2 = \frac{3}{2}\delta + \deg df_2 \wedge df_3.$
\end{fact}

\begin{proof}
From (\ref{eq:phi1}) we get
$$df_1' \wedge df_2 = df_1 \wedge df_2 + 2af_3\, df_3\wedge df_2 + c\, df_3 \wedge df_2.$$
By \ref{SFdel} we have $\deg f_3 > \deg df_1 \wedge df_2$, so $2af_3\, df_3\wedge df_2$ has strictly larger degree than the two other terms of the right-hand side.
Finally,
\begin{equation*}
\deg df_1' \wedge df_2 = \deg f_3\,df_3\wedge df_2 = \tfrac{3}{2}\delta + \deg df_2 \wedge df_3. \qedhere
\end{equation*}
\end{proof}

\begin{fact}
\label{fact:deg f1'}
$\deg f_1' > \delta.$
\end{fact}

\begin{proof}
Consider $P =  f_1 + a y^2 + c y + ef_2 \in \K[f_1,f_2][y]$.
We have
$$\dvirt P(f_3) = \deg f_1 > \deg f_1' = \deg P(f_3).$$
On the other hand $P' = 2ay + c$, so that $\dvirt P'(f_3) =\deg P'(f_3) =  \deg f_3$.
Thus $m(P,f_3) = 1$, and the Parachute Inequality \ref{pro:parachute} yields
\begin{align*}
\deg f_1' = \deg P(f_3) &\ge \dvirt P(f_3) + \deg df_1 \wedge df_2 \wedge df_3 - \deg df_1 \wedge df_2 - \deg f_3\\
            &> \deg f_1  - \deg df_1 \wedge df_2 - \deg f_3\\
            &= \deg f_3  - \deg df_1 \wedge df_2.
\end{align*}
Recall that by \ref{SFdel} we have
$$ \deg f_3 \ge \deg f_1 - \deg f_2  + \deg df_1 \wedge df_2 = \delta + \deg df_1 \wedge df_2.$$
Replacing in the previous inequality we get the result.
\end{proof}

\begin{fact}
\label{fact:non resonant}
The vertices $\llb f_1', f_2 \rrb$ and $\llb f_1', f_3 \rrb$ do not have outer resonance in $v_3' = \llb f_1', f_2, f_3 \rrb$.
\end{fact}

\begin{proof}
We have (the last inequality is Fact \ref{fact:deg f1'}):
$$\deg f_2 = 2\delta, \quad \deg f_3 = \tfrac32 \delta, \quad \deg f_1' >
\delta.$$
Moreover these degrees are pairwise distinct, because $(f_1', f_2, f_3)$ is a
good representative of $v_3'$.
This implies that $\deg f_3$ is not a $\N$-combination of $\deg f_1'$ and $\deg
f_2$, and $\deg f_2$  is not a $\N$-combination of $\deg f_1'$ and $\deg f_3$.
\end{proof}

Now we are ready to finish the proof of Lemma \ref{lem:no minimal reduction}.
Observe that by Lemma \ref{lem:normalization}, to prove that $v_3'$ does not
admit a $K$-reduction it is sufficient to exclude normal $K$-reductions.
Recall also that from Facts \ref{fact:df1 df3} and \ref{fact:df1' df2} we have
\begin{equation} \label{eq:three 2-forms}
\deg f_1' \wedge f_2 > \deg f_1' \wedge f_3 > \deg f_2 \wedge f_3.
\end{equation}

\begin{fact} \label{fact:elementary K}
$v_3'$ does not admit an elementary $K$-reduction.
\end{fact}

\begin{proof}
By contradiction, assume that $v_3'$ admits an elementary $K$-reduction.
By (\ref{eq:three 2-forms}) and Corollary \ref{cor:center of a K
red}\ref{center:1}, it follows that the center of the reduction is $\llb f_2,
f_3 \rrb$.
But then by Corollary \ref{cor:SPF} we should have $2\deg f_2 = s' \deg f_3$ for some odd integer $s' \ge 3$, and this is not compatible with $\deg f_2 = 2\delta$ and $\deg f_3 = \frac32 \delta$.
\end{proof}

\begin{fact} \label{fact:proper K}
$v_3'$ does not admit a normal proper $K$-reduction.
\end{fact}

\begin{proof}
By contradiction, assume that $v_3'$ admits a normal proper $K$-reduction $u_3$ via $w_3$.
By comparing (\ref{eq:three 2-forms}) with (\ref{NP6}) from Lemma
\ref{lem:normal proper}, we obtain that the center $v_2'$ of $v_3' \ne w_3$ is
equal to $v_2' = \llb f_2, f_3 \rrb$, that $v_2'$ is the minimal line of
$v_3'$, and that $\deg w_3 = 3\delta + \deg f_2 + \deg f_3 = \deg v_3$.
Since $\llb f_2, f_3 \rrb$ is also the minimal line of $v_3$, we
get that $u_3$ is also a proper $K$-reduction of $v_3$ via $w_3$.
This is a contradiction with Lemma \ref{lem:no proper K reduction}.
\end{proof}

Now we are left with the task of proving that $v_3'$ does not admit an
elementary reduction with center $\llb f_1', f_2 \rrb$ or $\llb f_1', f_3 \rrb$.

\begin{fact} \label{fact:no f1'f2}
$v_3'$ does not admit an elementary reduction with center $\llb f_1', f_2 \rrb$.
\end{fact}

\begin{proof}
By contradiction, assume there exists $\phi_3(f_1', f_2)$ such that
$$\deg{\phi_3(f_1', f_2)} = \deg{f_3}.$$
By Fact \ref{fact:non resonant} we have $\dvirt \phi_3(f_1',f_2) > \deg \phi_3(f_1',f_2)$, and on the other hand Proposition \ref{pro:SPF}\ref{SPF1} gives
$$\tfrac{3}{2}\delta = \deg f_3 = \deg \phi_3(f_1',f_2) > \deg df_1'\wedge df_2.$$
This is a contradiction with Fact \ref{fact:df1' df2}.
\end{proof}

\begin{fact} \label{fact:no f1'f3}
$v_3'$ does not admit an elementary reduction with center $\llb f_1', f_3 \rrb$.
\end{fact}

\begin{proof}
By contradiction, assume there exists $\phi_2(f_1', f_3)$ such that
$\deg f_2 > \deg (f_2 +  \phi_2(f_1', f_3))$, which implies
$$\deg{\phi_2(f_1', f_3)} = \deg{f_2}.$$
By Fact \ref{fact:non resonant} we have $\dvirt \phi_2(f_1',f_3) > \deg \phi_2(f_1',f_3)$.
By Lemma \ref{lem:p and q} there exist coprime $p,q \in \N^*$ and $\gamma \in
\N^3$ such that
$$\deg f_1' = p\gamma, \qquad \deg f_3 = q\gamma.$$
Corollary \ref{cor:parachute}\ref{parachute-i} then yields
$$ 2\delta = \deg f_2 = \deg \phi_2(f_1',f_3) \ge
pq\gamma + \deg df_1' \wedge df_3 - p\gamma - q\gamma.$$
By Fact \ref{fact:df1 df3} we know that $ \deg df_1' \wedge df_3 = \delta + \deg df_2\wedge df_3$, so that
\begin{equation}
\label{eq:pq-p-q}
\delta \ge \deg df_2\wedge df_3 + (pq - p - q) \gamma.
\end{equation}
We know that  $q\gamma = \deg f_3 = \frac{3}{2}\delta > \delta$, and by Fact \ref{fact:deg f1'} we have $p\gamma = \deg f_1' > \delta$, so
$$ \min\{p,q\} > pq - p - q .$$
By Fact \ref{fact:non resonant} we know that $p \neq 1$ and $q \neq 1$.
The only possibilities for the pair $(p,q)$ are then $(2,3)$ or $(3,2)$.

If $(p,q) = (2,3)$ then Fact \ref{fact:deg f1'} gives the contradiction
$$ 3\delta = 2 \deg f_3 = 3\deg f_1' > 3\delta.$$

If $(p,q) = (3,2)$, the equalities $\deg f_1' = 3 \gamma$ and $\deg f_3 = 2
\gamma = \frac32 \delta$  yield
$$\gamma = \tfrac{3}{4}\delta, \quad \deg f_1' = \tfrac{9}{4} \delta,
\quad pq - p-q = 1.$$
The inequality (\ref{eq:pq-p-q}) becomes
$$ \tfrac{\delta}{4} \ge \deg df_2\wedge df_3.$$
Corollary \ref{cor:parachute h + phi} then yields
\begin{align*}
\deg (f_2 + \phi_2(f_1',f_3)) &> 2\deg f_1' - \deg d f_2 \wedge d f_3 - \deg
f_1' \\
	&\ge \left( \tfrac{9}{4} - \tfrac{1}{4} \right) \delta \\
	&= 2\delta.
\end{align*}
This is a contradiction with $2\delta = \deg f_2 > \deg (f_2 +  \phi_2(f_1',f_3)).$
\end{proof}

This finishes the proof of Lemma \ref{lem:no minimal reduction}.
\end{proof}

We now introduce an induction hypothesis that will be the corner-stone for the proof of the Reducibility Theorem \ref{thm:reducibility}.

\begin{IH}[for degrees $\nu, \mu$ in $\N^3$] \label{IH}
Let $v_3$ be a reducible vertex such that $\nu \ge \deg v_3$.
Then any neighbor $u_3$ of $v_3$ with $\mu \ge \deg u_3$ also is reducible.
\end{IH}

We shall refer to this situation by saying that we can apply the $(\nu, \mu)$-Induction Hypothesis \ref{IH} to $v_3 \ne u_3$.

\begin{lemma} \label{lem:optimal normal}
Let $\mu > \nu \in \N^3$ be two consecutive degrees, and assume Induction Hypothesis $\ref{IH}$ for degrees $\nu, \nu$.
Let $v_3$ be a reducible vertex distinct from $[\id]$, with $\mu \ge \deg v_3$,
and let $T$ be a good triangle for $v_3$.
Then there exists a reducible vertex $u_3$ such that
\begin{itemize}
\item either $u_3$ is an optimal elementary $T$-reduction or an optimal
elementary $K$-reduction of $v_3$;
\item or $u_3$ is a normal proper $K$-reduction of $v_3$.
\end{itemize}
\end{lemma}

\begin{proof}
Let $v_3(1)$ be the first step of a reduction path from $v_3$, with respect to
the triangle $T$.
By this we mean that $v_3(1)$ is a reducible vertex that is either a
$T$-elementary reduction of $v_3$, or a $K$-reduction of $v_3$.

First assume that $v_3(1)$ is an elementary reduction of $v_3$.
If this reduction is optimal, we can take $u_3 = v_3(1)$ and we are done.
If the reduction is non-optimal, denote by $v_2$ the center of $v_3 \ne
v_3(1)$.
Then let $u_3$ be an optimal elementary reduction of $v_3$ with the same center
$v_2$.
Since $\mu > \nu $ are two consecutive degrees, the inequalities
$$\mu \ge \deg v_3, \quad \deg v_3 > \deg v_3(1) \quad \text{ and } \quad
\deg v_3 > \deg u_3$$
imply
$$\nu \ge \deg v_3(1)  \quad  \text{ and } \quad \nu \ge \deg u_3.$$
By the $(\nu,\nu)$-Induction Hypothesis \ref{IH} applied to $v_3(1) \ne u_3$
we get that $u_3$ is reducible, as expected.
Moreover since $u_3$ is an elementary reduction with center $v_2$, by
construction $u_3$ is an optimal elementary $T$-reduction or an optimal
elementary $K$-reduction of $v_3$.

Now assume that $v_3(1)$ is a proper $K$-reduction of $v_3$.
If this reduction is normal, we are done.
Otherwise, by Lemma \ref{lem:normalization}, there exists $u_3'$ an elementary
$K$-reduction of $v_3$, such that $u_3' \ne v_3(1)$.
By the $(\nu, \nu)$-Induction Hypothesis \ref{IH} applied to $v_3(1) \ne u_3'$
we get that $u_3'$ is reducible.
So we can take $u_3'$ as the first step of a reduction path from $v_3$,
and we are reduced to the first case of the proof.
\end{proof}

\begin{proposition} \label{pro:slide}
Let $\mu > \nu \in \N^3$ be two consecutive degrees, and assume Induction Hypothesis $\ref{IH}$ for degrees $\nu, \nu$.
Let $v_3$ be a vertex that is part of a simplex $v_1, v_2, v_3$ with Strong Pivotal Form $\PF(s)$ for some odd $s \ge 3$.
Let $T$ be any good triangle compatible with the simplex $v_1$, $v_2$, $v_3$.
Assume that $v_3$ is reducible, and that $\mu \ge \deg v_3$.
Then
\begin{enumerate}
\item \label{slide1} There exists a reduction path from $v_3$ to $[\id]$ that
starts with an elementary $T$-reduction or an elementary $K$-reduction;
\item \label{slide2} Any elementary $T$-reduction or elementary $K$-reduction
of $v_3$ admits $v_2$ as a center;
\item \label{slide3} Any optimal elementary $T$-reduction of $v_3$ is an elementary $K$-reduction;
\item \label{slide4} There exists an elementary $K$-reduction $u_3$ of $v_3$ with center $v_2$, such that $u_3$ is reducible.
\end{enumerate}
\end{proposition}

\begin{proof}
Let $v_3(1)$ be the first step of a reduction path from $v_3$ to $[\id]$,
with respect to the choice of good triangle $T$.
By Lemma \ref{lem:optimal normal} we can assume that if this first step is a $K$-reduction, then it is a normal $K$-reduction.
Moreover we know from Lemma \ref{lem:no proper K reduction} that $v_3$
does not admit any normal proper $K$-reduction.
This gives \ref{slide1}.

We write $v_3 = \llb f_1, f_2, f_3 \rrb$ as in Set-Up \ref{setup:pivotal}, such
that the triangle $T$ corresponds to the choice of representative $(f_1, f_2,
f_3)$.
The remaining possibilities for the first step $v_3(1)$ of the
reduction path are:
\begin{enumerate}[(i)]
\item \label{slide:case1} An elementary $T$-reduction with center $\llb f_2,
f_3 \rrb$;
\item \label{slide:case0} An elementary $T$-reduction with center $\llb f_1,
f_3 \rrb$;
\item \label{slide:case2} An elementary $K$-reduction;
\item \label{slide:case3} An elementary $T$-reduction with center $v_2 = \llb
f_1, f_2 \rrb$.
\end{enumerate}

In case \ref{slide:case1}, by Lemma \ref{lem:optimal normal} we can moreover assume that the elementary reduction is optimal.
Then Lemma \ref{lem:no minimal reduction} gives a contradiction.

In case \ref{slide:case0}, Lemma \ref{lem:no outer reduction}
implies that $\deg f_1 = 2\deg f_3$, so there exists $a \in \K$ such that $w_3
:= [f_1 + a f_3^2, f_2, f_3]$ satisfies $\deg v_3 > \deg w_3$.
By the Square Lemma \ref{lem:square}, there exists $u_3$ such that $u_3 \ne
w_3$, $u_3 \ne v_3(1)$ and $\deg v_3 > \deg u_3$.
By the $(\nu,\nu)$-Induction Hypothesis \ref{IH} applied successively to $v_3(1) \ne
u_3$ and $u_3 \ne w_3$, we obtain that $w_3$ is reducible, and so could be
chosen as the first step of a reduction path.
We are reduced to case \ref{slide:case1}, which leads to a contradiction.

In case \ref{slide:case2}, by Lemma \ref{lem:dm2}\ref{dm2:1}  we have
$$d(m_2) > \topdeg v_3.$$
Moreover by \ref{SFdel} we have
$$\topdeg v_3 > \deg (v_3 \smallsetminus v_2) \ge \Delta(v_2) > d(v_2).$$
By Corollary \ref{cor:center of a K red}\ref{center:2} we conclude that $v_2$ is the center of the $K$-reduction, hence we also are in case \ref{slide:case3}, which gives \ref{slide2}.

Now assume that $u_3$ is an optimal elementary $T$-reduction of $v_3$.
By \ref{slide2}, we know that the center of this reduction is $v_2$,
and by the $(\nu,\nu)$-Induction Hypothesis \ref{IH} applied to $v_3(1) \ne_{v_2} u_3$,
we get that $u_3$ is reducible.
We want to prove that $\Delta(v_2) > \deg(u_3 \smallsetminus v_2)$, that is, Property \ref{def:Kdel}, which will imply that $u_3$ is a $K$-reduction of $v_3$.
By contradiction, assume $\deg(u_3 \smallsetminus v_2) \ge \Delta(v_2)$, which
is condition \ref{SFdel} for the simplex $v_1,v_2, u_3$.
Moreover conditions \ref{SFinn} and \ref{SFmin} for the simplex $v_1,v_2, u_3$
directly follow from the analogous conditions for the simplex $v_1,v_2, v_3$,
and condition \ref{SFout} follows from the optimality of the reduction $u_3$.
Thus $v_1,v_2, u_3$ has Strong Pivotal Form $\PF(s)$, and from assertions
\ref{slide1} and \ref{slide2} we conclude that $u_3$ admits an elementary
reduction with center $v_2$.
This contradicts the optimality of the reduction from $v_3$ to $u_3$, and so we obtain \ref{slide3}.

Finally to prove \ref{slide4}, consider a first step of a reduction path from
$v_3$ to $[\id]$,  with respect to the good triangle $T$, as given by Lemma
\ref{lem:optimal normal}.
By Lemma \ref{lem:no proper K reduction} this first step
is not a normal proper $K$-reduction, so that it is an optimal elementary
reduction, hence by
\ref{slide3} this is an elementary $K$-reduction, as expected.
\end{proof}

\begin{corollary} \label{cor:minimal or K}
Let $\mu > \nu \in \N^3$ be two consecutive degrees, and assume Induction Hypothesis $\ref{IH}$ for degrees $\nu, \nu$.
Let $v_3$ be a reducible vertex with $\mu \ge \deg v_3$, and assume that $u_3$ is an optimal elementary reduction of $v_3$ with pivotal simplex $v_1,v_2,v_3$.
If $v_2$ has no inner resonance, and no outer resonance in $v_3$, then either $v_2$ is the minimal line of $v_3$, or $u_3$ is an elementary $K$-reduction of $v_3$.
\end{corollary}

\begin{proof}
Assume $v_2$ is not the minimal line of $v_3$.
By Proposition \ref{pro:SPF}\ref{SPF2}, the simplex $v_1, v_2, v_3$ has Strong Pivotal Form.
Let $T$ be a good triangle compatible with the simplex $v_1$, $v_2$, $v_3$.
In particular $u_3$ is an optimal elementary $T$-reduction of $v_3$, so we can
conclude by Proposition \ref{pro:slide}\ref{slide3}.
\end{proof}

\subsection{Vertex with two low degree neighbors} \label{sec:4C}

\begin{setup} \label{setup:2 neighbors}
Let $\mu >\nu \in \N^3$ be two consecutive degrees, and assume the Induction Hypothesis \ref{IH} for degrees $\nu, \mu$.
Let $v_3, v_3', v_3''$ be vertices such that
\begin{itemize}
\item $\mu \ge \deg v_3$, $\deg v_3 > \deg v_3'$ (hence $\nu \ge \deg v_3'$), $\deg v_3 \ge \deg v_3''$;
\item $v_3' \ne_{v_2'} v_3$ and $v_3'' \ne_{v_2''} v_3$ with $v_2' \neq v_2''$;
\item $v_3'$ is reducible (hence $v_3$ also is by the $(\nu,\mu)$-Induction Hypothesis);
\item $v_2'$ is minimal, in the sense that if $u_3$ is an elementary reduction of $v_3$ with center $u_2$, which is the first step of a reduction path, then $\deg u_2 \ge \deg v_2'$.
\end{itemize}

We denote by $v_1 = [ f_2 ]$ the intersection point of the lines $v_2'$ and $v_2''$.
We fix choices of $f_1, f_3$ such that $v_2' = \llb f_1, f_2 \rrb$ and $v_2'' = \llb f_2, f_3 \rrb$.
Observe that it is possible that $\deg f_1 = \deg f_3$, and in this case $(f_1, f_2, f_3)$ is not a good representative of $v_3$.
In any case by Lemma \ref{lem:elem red} there exist some non-affine
polynomials in two variables $P_1, P_3$ such that (see Figure
\ref{fig:setup}):
\begin{align*}
v_3 &= [f_1, f_2, f_3]; \\
v_3' &=  \llb f_1, f_2, f_3 + P_3(f_1, f_2) \rrb; \\
v_3'' &= \llb f_1 + P_1(f_2, f_3), f_2, f_3 \rrb.
\end{align*}
\end{setup}

\begin{figure}[ht]
$$\mygraph{
!{<0cm,0cm>;<0.8cm,0cm>:<0cm,0.8cm>::}
!{(2,0)}*{\typethree}="v3'"
!{(-2,0)}*{\typethree}="v3''"
!{(0,0)}*{\typeone}="v1"
!{(0,2)}*-{\typethree}="v3"
!{(1,1)}*-{\typetwo}="v2'"
!{(-1,1)}*-{\typetwo}="v2''"
"v2'"-^>{v_3' =\llb f_1, f_2,  f_3 + P_3(f_1,f_2) \rrb}"v3'"-^>{v_1 = [ f_2 ]}"v1"-"v2'"-_<{v_2' =\llb f_1,f_2 \rrb}_>(1.1){v_3 =[f_1,f_2,f_3]}"v3"-"v1"-"v3''"-^<{v_3'' =\llb f_1 + P_1(f_2,f_3), f_2,f_3\rrb}"v2''"-^<{v_2''=  \llb f_2,f_3 \rrb}"v3"
"v1"-"v2''"
}
$$
\caption{Set-Up \ref{setup:2 neighbors}.}\label{fig:setup}
\end{figure}

In this section we shall prove:

\begin{proposition} \label{pro:2 neighbors}
Assume Set-Up $\ref{setup:2 neighbors}$.
Then $v_3''$ is reducible.
\end{proposition}

We divide the proof in several lemmas.
The proposition will be a direct consequence of Lemmas \ref{lem:outer
resonance}, \ref{lem:v3' K reduction}, \ref{lem:easy} and \ref{lem:hard}.
We start with a consequence from the minimality of $v_2'$.

\begin{lemma} \label{lem:inner gives minimal}
Assume Set-Up $\ref{setup:2 neighbors}$.
If $v_2'$ has inner resonance, then $v_2'$ is the minimal line of $v_3$.
\end{lemma}

\begin{proof}
Assume first that there exists $a \in \K$ and $r \ge 2$ such that $\deg f_2 > \deg (f_2 + a f_1^r)$.
Then consider the vertex $w_3 = [f_1, f_2 + af_1^r, f_3]$, which satisfies $\deg v_3 > \deg w_3$.
Assume by contradiction that $v_2'$ is not the minimal line of $v_3$.
Then we have $\deg f_2 > \deg f_3$, hence $m_2 = [f_1, f_3]$ is the minimal line of $v_3$.
By the Square Lemma \ref{lem:square}, we find $u_3$ such that $u_3 \ne w_3$, $u_3 \ne v_3'$ and $\deg v_3 > \deg u_3$.
By applying the $(\nu,\mu)$-Induction Hypothesis \ref{IH} successively to $v_3' \ne u_3$ and $u_3 \ne w_3$, we find that $w_3$ is reducible.
So we can take $w_3$ as the first step of a reduction path from $v_3$, which contradicts the minimality of $v_2'$.

Now assume that there exist $a \in \K^*$ and $r \ge 2$ such that $\deg f_1 >
\deg (f_1 + a f_2^r)$.
If $v_2'$ is not the minimal line, we have $\deg f_1 \ge \deg f_3$.

If $\deg f_1 = \deg f_3$, there exists $b \in \K^*$ such that $\deg f_1 > \deg
(f_1 + b f_3)$.
Then $[f_2, f_1 +b f_3]$ is the minimal line of $v_3$, and we consider $w_3' = [f_1 + af_2^r, f_2,f_1 +b f_3]$, which satisfies $\deg v_3 > \deg w_3'$.
As above, by the Square Lemma \ref{lem:square} and the $(\nu,\mu)$-Induction
Hypothesis \ref{IH}, we obtain that $w_3'$ can be chosen to be the first step of
a reduction path from $v_3$.
This contradicts the minimality of $v_2'$.

If $\deg f_1 > \deg f_3$, $v_2'' = \llb f_2, f_3 \rrb$ is the minimal line of $v_3$.
We consider $w_3'' = [f_1 + af_2^r, f_2, f_3]$ which satisfies $\deg v_3 > \deg w_3''$.
As before $w_3''$ can be chosen to be the first step of a reduction path from $v_3$, with center $v_2''$: again this contradicts the minimality of $v_2'$.
\end{proof}

\begin{lemma} \label{lem:outer resonance}
Assume Set-Up $\ref{setup:2 neighbors}$.
If $v_2'$ is not the minimal line in $v_3$, and has outer resonance in $v_3$, then $v_3''$ is reducible.
\end{lemma}

\begin{proof}
Since $v_2' = \llb f_1, f_2\rrb$ is not the minimal line in $v_3$, one of the
degrees $\deg f_1$ or $\deg f_2$ must realize the top degree of $v_3$.

First consider the case $\deg f_2 = \topdeg v_3$.
In this case, we have $\deg f_1 \neq \deg f_3$.
Indeed, assuming by contradiction that $\deg f_1 = \deg f_3$, there would exist
$\alpha \in \K^*$ such that $f_3' = f_3 - \alpha f_1$ satisfies $\deg f_3 >
\deg f_3'$.
Then, we would have $\deg f_2 > \deg f_1 > \deg f_3'$, and the line $v_2' = \llb
f_1, f_2 \rrb$ could not have outer resonance in $v_3$.
Therefore, $(f_1, f_2, f_3)$ is a good representative of $v_3$, and the
assumption on outer resonance means that $\deg f_3 = \deg f_1^r$ for some $r \ge
2$.
Recall that by Lemma \ref{lem:elem red}, the non-affine polynomial $P_1 \in \K[y,z]$  associated with the weak elementary reduction $v_3''$ of $v_3$ satisfies $\deg f_1 \ge P_1(f_2,f_3)$.
We have $\deg f_2 \not\in \N \deg f_3$, otherwise we would have $\deg f_2 \in \N\deg f_1$, that is, $v_2'$ would have inner resonance, and by Lemma \ref{lem:inner gives minimal} this would contradict $v_2' \neq m_2$.
Corollary \ref{cor:parachute}\ref{parachute-iii} then gives $2\deg f_2 = 3 \deg
f_3$, and $\deg (v_3 \smallsetminus v_2'') \ge \Delta(v_2'')$.
But then the existence of $\delta \in \N$ such that $\deg f_2 = 3\delta, \deg
f_3 = 2\delta$ and  $\deg f_1 > \delta$ is not compatible with the existence of
$r \ge 2$ such that $\deg f_3 = r \deg f_1$: contradiction.

Now consider the case $\deg f_1 = \topdeg v_3$.
In particular $\deg f_1 > \deg f_2$ and $v_1 =[ f_2 ]$ is on the minimal line of $v_3$.
Now we distinguish two subcases (see Figure \ref{fig:outer resonance}):
\begin{itemize}[wide]
\item Case $\deg f_1 > \deg f_3$.
Then there exist $a \in \K$ and $r \ge 2$ such that $\deg f_3 >  \deg
(f_3 +af_2^r)$, and $v_2''$ is the minimal line of $v_3$.
We apply the Square Lemma \ref{lem:square} to get $u_3$ a common neighbor of $v_3''$ and $w_3' = [f_1, f_2, f_3 + af_2^r]$ satisfying $\deg v_3 > \deg u_3$, and then we conclude by the $(\nu,\mu)$-Induction Hypothesis \ref{IH} applied successively to $w_3' \ne u_3$ and $u_3 \ne v_3''$.
\item Case $\deg f_1 = \deg f_3$. Then there exists $b \in \K^*$ such that
$\deg f_1 > \deg (bf_1 + f_3)$.
As in the first case of the proof, the assumption on outer resonance implies
that $(f_1, f_2, f_3 + bf_1)$ is a good representative of $v_3$,
and there exists $r \ge 2$ such that $\deg (bf_1 + f_3) = \deg f_2^r$.
Then there exists $a \in \K$ such that $u_3' = [f_1,f_2, bf_1 + f_3 +af_2^r]$ and $u_3'' = [bf_1 + f_3 +af_2^r,f_2,f_3]$ are (simple) elementary reductions of $v_3$ with respective centers $v_2'$ and $v_2''$.
Moreover $u_3'$ and $u_3''$ are neighbors, with center $[bf_1 + f_3 +af_2^r,f_2]$.
Again we conclude by applying the $(\nu,\mu)$-Induction Hypothesis \ref{IH} to $v_3' \ne u_3'$, $u_3' \ne u_3''$ and $u_3'' \ne v_3''$ successively. \qedhere
\end{itemize}
\end{proof}

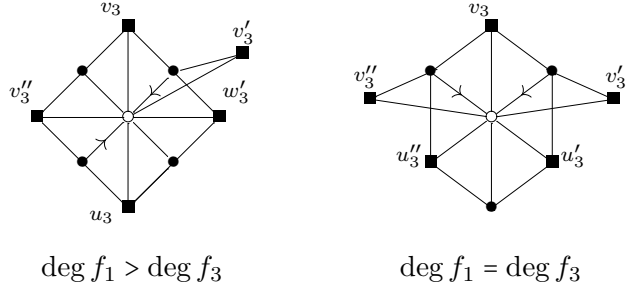
\begin{figure}[t]
$$
\xymatrix@R-20px{
\mygraph{
!{<0cm,0cm>;<0.6cm,0cm>:<0cm,0.6cm>::}
!{(2.5,1.4)}*{\typethree}="v3''"
!{(2,0)}*{\typethree}="w''"
!{(0,2)}*-{\typethree}="v3"
!{(-2,0)}*{\typethree}="v3'"
!{(0,-2)}*{\typethree}="w"
!{(-1,1)}*-{\typetwo}="v2NO"
!{(-1,-1)}*-{\typetwo}="v2SO"
!{(1,1)}*{\typetwo}="v2NE"
!{(1,-1)}*{\typetwo}="v2SE"
!{(0,0)}*{\typeone}="v1"
"w"-^<{u_3}"v3'"-^<{v_3''}^>{v_3}"v3"-^>{w_3'}"w''"-"w"
"v2NE"-^>(1.2){v_3'}"v3''"
"v2NO"-"v1"-|@{<}"v2SO" "v2NE"-|@{>}"v1"-"v2SE"
"w''"-"v1"-"v3'" "v3"-"v1"-"w"-"v2SE"
"v3''"-"v1"
}
&
\mygraph{
!{<0cm,0cm>;<0.8cm,0cm>:<0cm,0.6cm>::}
!{(2,0.4)}*{\typethree}="v3'"
!{(0,2)}*-{\typethree}="v3"
!{(-2,0.4)}*{\typethree}="v3''"
!{(0,-2)}*-{\typetwo}="u2"
!{(-1,1)}*-{\typetwo}="v2''"
!{(-1,-1)}*{\typethree}="u3''"
!{(1,1)}*-{\typetwo}="v2'"
!{(1,-1)}*{\typethree}="u3'"
!{(0,0)}*{\typeone}="v1"
"v3''"-^<{v_3''}"v2''"-^>{v_3}"v3"-"v2'"-^>{v_3'}"v3'"
"v2''"-|@{>}"v1"-|@{<}"v2'" "u3''"-"v1"-"u3'" "v3"-"v1"-"u2"
"v2''"-_>{u_3''}"u3''"-"u2"-"u3'"-_<{u_3'}"v2'"
"v3''"-"v1"-"v3'"
}
\\
\deg f_1 > \deg f_3 & \deg f_1 = \deg f_3
}
$$
\caption{Lemma \ref{lem:outer resonance}, case $\deg f_1 = \topdeg v_3$.}\label{fig:outer resonance}
\end{figure}

We conclude the case where $v_2'$ is not equal to the minimal line $m_2$ of $v_3$ with the following:

\begin{lemma} \label{lem:v3' K reduction}
Assume Set-Up $\ref{setup:2 neighbors}$, $v_2'\neq m_2$ (where $m_2$ is the minimal line in $v_3$), and has no outer resonance in $v_3$.
Then there exists $u_3$ an elementary $K$-reduction of $v_3$ with center $v_2'$.
Moreover, $v_3''$ is reducible, and there exists a reduction path starting with a proper $K$-reduction from $v_3''$ to $u_3$.
\end{lemma}

\begin{proof}
By Lemma \ref{lem:inner gives minimal}, we know that $v_2'$ has no inner resonance.
Then Corollary \ref{cor:minimal or K} says that any optimal elementary reduction $u_3$ of $v_3$ with center $v_2'$ is an elementary $K$-reduction.
By the $(\nu,\mu)$-Induction Hypothesis \ref{IH} applied to $v_3' \ne u_3$, we get that $u_3$ is reducible.
The last assertion follows by the Stability of $K$-reductions \ref{pro:K
stability}, Case \ref{Kstability:1} or \ref{Kstability:2}.
\end{proof}

Now we treat the situation where $v_2' = m_2$ is the minimal line of $v_3$, and first we identify some cases that we can handle with the Square Lemma \ref{lem:square}.

\begin{lemma} \label{lem:easy}
Assume Set-Up $\ref{setup:2 neighbors}$, and $v_2' = m_2$.
In the following cases, $v_3''$ is reducible:
\begin{enumerate}
\item \label{case1:easy} $v_3$ admits a simple elementary reduction with simple center $v_2', v_1$;
\item \label{case2:easy} $v_3$ admits a simple elementary reduction with simple center $v_2'', v_1$;
\item \label{case3:easy} $v_3''$ is a simple weak elementary reduction of $v_3$ with simple center $v_2'', v_1$.
\end{enumerate}
\end{lemma}

\begin{figure}[ht]
$$
\xymatrix@R-20px{
\mygraph{
!{<0cm,0cm>;<0.55cm,0cm>:<0cm,0.55cm>::}
!{(.6,-.2)}*{\text{\scriptsize $v_1$}}
!{(2.5,1.4)}*{\typethree}="v3''"
!{(2,0)}*{\typethree}="w''"
!{(0,2)}*{\typethree}="v3"
!{(-2,0)}*{\typethree}="v3'"
!{(0,-2)}*{\typethree}="w"
!{(-1,1)}*{\typetwo}="v2NO"
!{(-1,-1)}*{\typetwo}="v2SO"
!{(1,1)}*{\typetwo}="v2NE"
!{(1,-1)}*{\typetwo}="v2SE"
!{(0,0)}*{\typeone}="v1"
"w"-^<{u_3}"v3'"-^<{v_3''}^>{v_3}"v3"-^>{w_3'}"w''"-"w"
"v2NE"-^<{v_2'}^>(1.2){v_3'}"v3''"
"v2NO"-"v1"-|@{<}"v2SO" "v2NE"-|@{>}"v1"-"v2SE"
"w''"-"v1"-"v3'" "v3"-"v1"-"w"-"v2SE"
"v3''"-"v1"
}
&
\mygraph{
!{<0cm,0cm>;<0.55cm,0cm>:<0cm,0.55cm>::}
!{(.6,-.2)}*{\text{\scriptsize $v_1$}}
!{(-2.5,1.4)}*{\typethree}="v3'"
!{(2,0)}*{\typethree}="v3''"
!{(0,2)}*{\typethree}="v3"
!{(-2,0)}*{\typethree}="w'"
!{(0,-2)}*{\typethree}="w"
!{(-1,1)}*{\typetwo}="v2NO"
!{(-1,-1)}*{\typetwo}="v2SO"
!{(1,1)}*{\typetwo}="v2NE"
!{(1,-1)}*{\typetwo}="v2SE"
!{(0,0)}*{\typeone}="v1"
"w"-^<{u_3}"w'"-^<{w_3''}^>{v_3}"v3"-^>{v_3'}"v3''"-"w"
"v2NO"-_<{v_2''}_>(1.2){v_3''}"v3'"
"v2NO"-|@{>}"v1"-"v2SO" "v2NE"-"v1"-|@{<}"v2SE"
"w'"-"v1"-"v3''" "v3"-"v1"-"w"-"v2SE"
"v3'"-"v1"
}
&
\mygraph{
!{<0cm,0cm>;<0.55cm,0cm>:<0cm,0.55cm>::}
!{(.6,-.2)}*{\text{\scriptsize $v_1$}}
!{(2,0)}*{\typethree}="v3''"
!{(0,2)}*{\typethree}="v3"
!{(-2,0)}*{\typethree}="w'"
!{(0,-2)}*{\typethree}="w"
!{(-1,1)}*{\typetwo}="v2NO"
!{(-1,-1)}*{\typetwo}="v2SO"
!{(1,1)}*{\typetwo}="v2NE"
!{(1,-1)}*{\typetwo}="v2SE"
!{(0,0)}*{\typeone}="v1"
"w"-^<{u_3}"w'"-^<{v_3''}^{v_2''}^>{v_3}"v3"-^>{v_3'}"v3''"-"w"
"v2NO"-|@{>}"v1"-"v2SO" "v2NE"-"v1"-|@{<}"v2SE"
"w'"-"v1"-"v3''" "v3"-"v1"-"w"-"v2SE"
}
\\
\text{Case } \ref{case1:easy} & \text{Case } \ref{case2:easy}  & \text{Case } \ref{case3:easy}
}
$$
\caption{Lemma \ref{lem:easy}.}\label{fig:easy}
\end{figure}

\begin{proof}
Since $m_2 = \llb f_1, f_2 \rrb$, we have $v_3 = \llb f_1, f_2, f_3 \rrb$ with $\deg f_3 = \topdeg v_3$.
\begin{enumerate}[wide]
\item
Denote by $w_3'$ a simple elementary reduction of $v_3$ with simple center
$v_2', v_1$.
By the Square Lemma \ref{lem:square}, there exists $u_3$ a neighbor of both $w_3'$ and $v_3''$ such that $\deg v_3 > \deg u_3$ (see Figure \ref{fig:easy}).
Then we conclude by applying the $(\nu,\mu)$-Induction Hypothesis \ref{IH} successively to $v_3' \ne w_3'$, $w_3' \ne u_3$ and $u_3 \ne v_3''$.

\item
Denote by $w_3''$ a simple elementary reduction of $v_3$ with simple center
$v_2'', v_1$.
By the Square Lemma \ref{lem:square}, there exists $u_3$ a neighbor of both $w_3''$ and $v_3'$ such that $\deg v_3 > \deg u_3$.
Then we conclude by applying the $(\nu,\mu)$-Induction Hypothesis \ref{IH} successively to $v_3' \ne u_3$, $u_3 \ne w_3''$ and $w_3'' \ne v_3''$.

\item
By the Square Lemma \ref{lem:square}, there exists $u_3$ a neighbor of both $v_3''$ and $v_3'$ such that $\deg v_3 > \deg u_3$.
Again the $(\nu,\mu)$-Induction Hypothesis \ref{IH} applied successively to $v_3' \ne u_3$ and $u_3 \ne v_3''$ yields that $v_3''$ is reducible.
\qedhere
\end{enumerate}
\end{proof}

The last logical step to finish the proof of Proposition \ref{pro:2 neighbors} is the following lemma.
However, we should mention that we will be able to prove \textit{a posteriori} that this case never happens: see Corollary \ref{cor:no hard}.

\begin{lemma} \label{lem:hard}
Assume Set-Up $\ref{setup:2 neighbors}$, $v_2' = m_2$, and that we are not in one of the cases covered by Lemma $\ref{lem:easy}$.
Then there exists $u_3''$ an elementary $K$-reduction of $v_3$ with center $v_2''$ such that $u_3''$ is reducible.
In particular, by the $(\nu,\mu)$-Induction Hypothesis $\ref{IH}$, $v_3''$ is reducible.
\end{lemma}

\begin{proof}
It is sufficient to check that the simplex $v_1, v_2'', v_3$ has Strong Pivotal
Form $\PF(s)$ for some odd $s\ge 3$: Indeed then one can apply Proposition
\ref{pro:slide}\ref{slide4} to get the result.

On the one hand $\deg f_3 > \deg f_1 \ge \deg P_1(f_2,f_3)$, and on the other hand since we are not in the situation of Lemma \ref{lem:easy}\ref{case3:easy} we have $P_1(f_2,f_3) \not\in \K[f_2]$, so that
$$\dvirt P_1(f_2,f_3) > \deg P_1(f_2,f_3).$$
We also have $\deg f_3 \not\in \N \deg f_2$ since otherwise we could apply Lemma \ref{lem:easy}\ref{case1:easy}.
So we are in the hypotheses of Corollary \ref{cor:parachute}\ref{parachute-iii}, and there exist a degree $\delta \in \N^3$ and an odd integer $s \ge 3$ such that $\deg f_2 = 2\delta$, $\deg f_3 = s\delta$ and
$$s\delta > \deg P_1(f_2,f_3) \ge \Delta(f_3,f_2).$$
It remains to check \ref{SFout}: if $v_2'' =\llb f_2, f_3\rrb$ had outer resonance in $v_3$ then we would have $\deg f_1 \in \N \deg f_2$, and we could apply Lemma \ref{lem:easy}\ref{case2:easy}, contrary to our assumption.
\end{proof}

\subsection{Proof of the Reducibility Theorem} \label{sec:the proof}

Clearly Theorem \ref{thm:reducibility} is a corollary of

\begin{proposition} \label{pro:reduction path}
If a vertex $v_3$ of type 3 in the complex $\Comp$ is reducible, then any neighbor of $v_3$ also is reducible.
\end{proposition}

\begin{proof}
We plan to prove Proposition \ref{pro:reduction path} by induction on degree: we need to prove that for any $\nu \in \N^3$, the Induction Hypothesis \ref{IH} holds for degrees $\nu,\nu$.
Clearly when $\nu =(1,1,1)$ this is true (because empty!).

Let $\mu > \nu$ be two consecutive degrees in $\N^3$.
It is sufficient to prove the two following facts.

\begin{fact} \label{fact:induction 1}
Assume the Induction Hypothesis $\ref{IH}$ for degrees $\nu,\nu$.
Then it also holds for degrees $\nu,\mu$.
\end{fact}

\begin{fact} \label{fact:induction 2}
Assume the Induction Hypothesis $\ref{IH}$ for degrees $\nu,\mu$.
Then it also holds for degrees $\mu,\mu$.
\end{fact}

To prove Fact \ref{fact:induction 1}, consider $v_3'$ a reducible vertex with $\nu \ge \deg v_3'$, and let $v_3$ be a neighbor of $v_3'$, with center $v_2'$, and with $\deg v_3 = \mu$ (otherwise there is nothing to prove).
We want to prove that $v_3$ is reducible.

If $v_2'$ is the minimal line of $v_3$, then $v_3'$ is a $T$-reduction of $v_3$ for any good triangle $T$ and we are done.

If $v_2'$ is not the minimal line of $v_3$, and has no inner or outer resonance in $v_3$, then by Corollary \ref{cor:minimal or K} any optimal reduction $u_3$ of $v_3$ with respect to the center $v_2'$ is an elementary $K$-reduction of $v_3$.
Since $u_3$ is a neighbor of $v_3'$ and $\nu \ge \deg u_3$, we conclude by the $(\nu,\nu)$-Induction Hypothesis \ref{IH} that $u_3$ is reducible.
Hence $v_3$ also is reducible, with first step of a reduction path the $K$-reduction to $u_3$.

Finally assume that $v_2'$ has resonance, and that $v_2'$ is not the minimal line of $v_3$.
By Lemma \ref{lem:elem red} we can write
$$v_3 = \llb f_1, f_2, f_3 \rrb, \quad v_3' = \llb f_1, f_2, g_3 \rrb, \quad v_2' = \llb f_1, f_2 \rrb,$$
with $\deg f_1 > \max \{\deg f_3,\deg f_2 \}$ and $g_3 = f_3 + P(f_1, f_2)$ for some polynomial $P$.

If $v_2'$ has inner resonance, then $\deg f_1 = r \deg f_2$ for some $r \ge 2$.
There exists $a \in \K$ such that $v_3'' = [f_1 + af_2^r, f_2, f_3]$ is a simple elementary reduction of $v_3$ with center $\llb f_2, f_3 \rrb$, which is the minimal line of $v_3$, hence belongs to any good triangle $T$.
Then we can apply the Square Lemma \ref{lem:square} to get $u_3$ with $\nu \ge \deg u_3$ and $u_3 \ne v_3'$, $u_3 \ne v_3''$.
We conclude by the $(\nu,\nu)$-Induction Hypothesis \ref{IH}, applied successively to $v_3' \ne u_3$ and $u_3 \ne v_3''$, that $v_3''$ is reducible, hence $v_3$ also is.

If $v_2'$ has no inner resonance, but has outer resonance in $v_3$, then $\deg f_1 > \deg f_3 > \deg f_2$ and $\deg f_3 = r \deg f_2$ for some $r \ge 2$.
There exists $Q(f_2)$ such that $\deg f_3 > \deg (f_3 + Q(f_2))$ and $\deg (f_3 + Q(f_2)) \not \in \N\deg f_2$.
Let $T$ be a good triangle of $v_3$.
One of the lines in $T$ has the form $u_2 = \llb f_1 + af_3, f_2\rrb$.
Set $u_3 = \llb f_1 + af_3, f_2, f_3 + Q(f_2) \rrb$, which is an
elementary $T$-reduction of $v_3$ with center $u_2$, and a neighbor of $w_3 =
\llb f_1, f_2, f_3 + Q(f_2) \rrb$.
By the $(\nu,\nu)$-Induction Hypothesis \ref{IH} applied successively to $v_3' \ne w_3$ and $w_3 \ne u_3$, we obtain that $u_3$ is  reducible, hence $v_3$ also is.
\\

To prove Fact \ref{fact:induction 2}, consider $v_3$ a reducible vertex with $\deg v_3 = \mu$, and let $v_3''$ be a neighbor of $v_3$, with center $v_2''$, such that
$$\deg v_3 \ge \deg v_3''.$$
We want to prove that $v_3''$ is reducible.

First assume that $v_3$ admits a reduction path such that the first step $v_3'$ is an elementary reduction, with center $v_2'$.
Moreover we assume that $v_2'$ has minimal degree between all possible such first center of a reduction path.
If $v_2' = v_2''$, then $v_3''$ is a neighbor of $v_3'$ and by the $(\nu,\mu)$-Induction Hypothesis \ref{IH} we are done.
If on the contrary $v_2'$ and $v_2''$ are two different lines in $\lines{v_3}$, we are in the situation of Set-Up \ref{setup:2 neighbors}.
Then we conclude by Proposition \ref{pro:2 neighbors}.

Finally assume that $v_3$ admits a reduction path such that the first step is a proper $K$-reduction $v_3'$.
If this proper $K$-reduction is normal, then by Stability of $K$-reduction \ref{pro:K stability}, Case
\ref{Kstability:3} or \ref{Kstability:4}, we obtain that $v_3'$ is also a
$K$-reduction of $v_3''$, and we are done.
On the other hand if $v_3'$ is a non-normal proper $K$-reduction of $v_3$, then by Lemma \ref{lem:normalization} we get the existence of a vertex $u_3$ that is both an elementary $K$-reduction of $v_3$ and a neighbor of $v_3'$.
By applying the $(\nu,\mu)$-Induction Hypothesis \ref{IH} to $v_3' \ne u_3$ we get that $u_3$ is reducible, and we are reduced to the previous case of the proof.
\end{proof}

\section{Simple connectedness} \label{sec:1-connected}

In this section we prove that the complex $\Comp$ is simply connected, which
amounts to saying that the group $\TA$ is the amalgamated product of three
subgroups along their pairwise intersections.

\subsection{Consequences of the Reducibility Theorem}

Now that the Reducibility Theorem \ref{thm:reducibility} is proved, all previous results that were dependent of a reducibility assumption become stronger.
This is the case in particular for:
\begin{itemize}
\item The Induction Hypothesis \ref{IH}, which is always true;
\item Lemma \ref{lem:no minimal reduction}, which now implies that if $v_3$ is part of a simplex with Strong Pivotal Form, then $v_3$ does not admit an elementary reduction with center $m_2$ the minimal line of $v_3$: see the proof of Proposition \ref{pro:no K} below;
\item Proposition \ref{pro:slide} and Corollary \ref{cor:minimal or K};
\item Set-Up \ref{setup:2 neighbors}, hence also all results in \S\ref{sec:4C}.
\end{itemize}
In particular we single out the following striking consequences of the Reducibility Theorem \ref{thm:reducibility}.

\begin{proposition} \label{pro:no K}
Let $u_3$ be an elementary $K$-reduction of $v_3$, with center $v_2$.
Then $v_3$ does not admit any elementary reduction with center distinct from $v_2$.
\end{proposition}

\begin{proof}
By contradiction, assume that $v_3'$ is an elementary reduction of $v_3$, with center $v_2'$ distinct from $v_2$.
Then by Stability of $K$-reduction \ref{pro:K stability},
Case \ref{Kstability:1} or \ref{Kstability:2}, there exists $w_3'$ with $\deg
w_3' = \deg v_3$ such that $u_3$ is a proper $K$-reduction of $v_3'$ via $w_3'$.
In particular, $v_3'$ is an elementary reduction of $w_3'$ with center $m_2$,
the minimal line of $w_3'$.
Moreover, by Corollary \ref{cor:SPF}, the pivotal simplex of this proper
$K$-reduction has Strong Pivotal Form.
Now consider $v_3''$ an optimal elementary reduction of $w_3'$ with center
$m_2$: Lemma \ref{lem:no minimal reduction} gives a contradiction.
\end{proof}

\begin{corollary} \label{cor:v3 w3 same degree}
Let $u_3$ be a proper $K$-reduction of $v_3$, via $w_3$.
Then $\deg w_3 = \deg v_3$.
\end{corollary}

\begin{proof}
By Proposition \ref{pro:no K} we cannot have $\deg w_3 > \deg v_3$.
\end{proof}

\begin{corollary} \label{cor:no inner resonance}
Let $v_3$ admitting a $K$-reduction, then any line $u_2$ in $\lines{v_3}$ has no inner resonance.
In other words, Case \ref{degreeK:case2} in Corollary \ref{cor:degree of K} never  happens.
\end{corollary}

\begin{proof}
We use the notation $v_3 = \llb f_1, f_2, f_3 \rrb$ from Corollary
\ref{cor:degree of K}, which says that if there exists a line in $v_3$ with
inner resonance, then $\deg f_1 = 2 \deg f_3$ and $\deg f_1 = \frac32 \deg f_2$.
But in this case we would have an elementary
reduction $[f_1 + af_3^2, f_2, f_3]$ of $v_3$ with center the minimal line $\llb f_2, f_3 \rrb
\neq v_2$, in contradiction with Proposition \ref{pro:no K}
\end{proof}

We also obtain that the situation of Lemma \ref{lem:hard} never happens:

\begin{corollary} \label{cor:no hard}
Assume Set-Up $\ref{setup:2 neighbors}$, and $v_2' = m_2$.
Then we are in one of the cases covered by Lemma $\ref{lem:easy}$.
\end{corollary}

\begin{proof}
Otherwise by Lemma \ref{lem:hard} there would exist an elementary $K$-reduction of $v_3$ with center $v_2''$, in addition to the elementary reduction with center $v_2'$: This is not compatible with Proposition \ref{pro:no K}.
\end{proof}

\subsection{Local homotopies}

To prove the simple connectedness of $\Comp$, the following terminology will be
convenient.

We call \textbf{combinatorial path} a sequence of vertices $v_3(i)$, $i = 0, \dots, n$,
such that for all $i = 0, \dots, n-1$, $v_3(i) \ne v_3(i+1)$.
Denoting by $v_2(i)$ the center of $v_3(i) \ne v_3(i+1)$, we think of such a sequence as equivalent to a path $\gamma\colon [0,2n] \to \Comp$ where for each $i = 0, \dots, n-1$, the interval $[2i, 2i + 2]$ is mapped isometrically onto the union of the two edges $v_3(i),v_2(i)$ and $v_2(i),v_3(i+1)$.
In particular we have these parameterizations in mind when we say that two such
combinatorial paths are homotopic in $\Comp$.
We say that a combinatorial path is \textbf{locally geodesic} if for all $i$, $v_2(i) \neq v_2(i+1)$ (and $v_3(i) \neq v_3(i+ 1)$, but this is already contained in the definition of $v_3(i) \ne v_3(i+ 1)$).
Observe that starting from any combinatorial path, by removing some vertices we can always obtain a
locally geodesic one.
If $v_3(0) = v_3(n)$ we say that the path is a \textbf{combinatorial loop} with base
point $v_3(0)$.
If $v_3(0) = [\id]$, the \textbf{maximal vertex} of such a loop is defined as
the vertex $v_3(i_0)$ that realizes the maximum $\max \deg v_3(i)$, with $i_0$
maximal.
In particular, we have
$$\deg v_3(i_0) > \deg v_3(i_0+1) \quad \text{ and } \quad \deg v_3(i_0) \ge
\deg v_3(i_0-1).$$

\begin{lemma} \label{lem:local homotopy}
Let $v_3(i)$, $i=0,\dots, n$, be a locally geodesic loop with base point at
$[\id]$, and let $v_3(i_0)$ be the maximal vertex.
Then the combinatorial path $v_3(i_0-1)$, $v_3(i_0)$, $v_3(i_0+1)$ is homotopic in
$\Comp$ to a combinatorial path from $v_3(i_0-1)$ to $v_3(i_0+1)$ where all
type 3 intermediate vertices have degree strictly less than $\deg v_3(i_0)$.
\end{lemma}

\begin{proof}
Since $\deg v_3(i_0) > \deg v_3(i_0+1)$, we know that $v_3(i_0)$ admits an
elementary reduction, so it makes sense to choose $v_2'$ a vertex of
minimal degree such that $v_3 = v_3(i_0)$ admits an elementary reduction $v_3'$
with center $v_2'$.
Then we are going to apply Set-Up \ref{setup:2 neighbors} twice, taking $v_3''$ to be
successively $v_3(i_0-1)$ and $v_3(i_0+1)$.
It is sufficient to prove that in both cases the combinatorial path $v_3'', v_3, v_3'$
is homotopic to a combinatorial path from $v_3''$ to $v_3'$ where all
type 3 intermediate vertices have degree strictly less than $\deg v_3$.
Observe that there are degenerate cases which are easy to handle.
First if $v_3' = v_3''$, we just take the combinatorial path with one single vertex $v_3'$.
Second, if $v_3'$ and $v_3''$ share the same center with $v_3$, we can just discard $v_3$ to obtain a combinatorial path from $v_3''$ to $v_3'$ without any intermediate vertex, so we can indeed
assume that the centers $v_2'$ and $v_2''$ are distinct as required in Set-Up
\ref{setup:2 neighbors}.

If $v_2' = m_2$ then by Corollary \ref{cor:no hard} we are in one of the cases
covered by Lemma \ref{lem:easy}, and the homotopy is clear in all cases (see
Figure \ref{fig:easy}).
For instance in Case \ref{case1:easy} we replace the path $v_3'', v_3, v_3'$ by
$v_3'', u_3, w_3', v_3'$, and the other cases are similar.

If $v_2' \neq m_2$, and $v_2'$ has outer resonance in $v_3$, then we are in one of the two cases covered by Lemma \ref{lem:outer resonance}, and again the homotopy is clear (see Figure \ref{fig:outer resonance}).

If $v_2' \neq m_2$, and $v_2'$ has no outer resonance in $v_3$, then by Lemma
\ref{lem:v3' K reduction} there exists $u_3$ an elementary $K$-reduction of
$v_3$ with center $v_2'$.
Therefore, in that case, up to replacing $v_3'$ by $u_3$ we may assume without loss in generality that $v_3'$ is an elementary $K$-reduction of $v_3$.
By Proposition \ref{pro:no K}, this can only happen in the case $v_3'' = v_3(i_0
- 1)$, and $\deg v_3(i_0) = \deg v_3(i_0 - 1)$.
By Stability of $K$-reduction \ref{pro:K stability}, Case \ref{Kstability:1} or
\ref{Kstability:2}, $v_3'$ is a proper
$K$-reduction of $v_3(i_0 - 1)$, and up to a local homotopy (see Figure
\ref{fig:stability}, Case \ref{Kstability:2}) we can
assume that the intermediate vertex is $v_3 = v_3(i_0)$.
If this proper $K$-reduction is not normal, we obtain the expected homotopy from
the normalization process of Lemma \ref{lem:normalization} (see Figure
\ref{fig:normal}).
Otherwise, By Stability of $K$-reduction \ref{pro:K stability},  Case \ref{Kstability:3} or
\ref{Kstability:4}, we get that
$v_3'$ is a normal proper $K$-reduction of $v_3(i_0 - 2)$, with $v_2(i_0 -2) =
v_2(i_0-1)$ the minimal line of $v_3(i_0)$: This contradicts our assumption
that we started with a locally geodesic loop.
\end{proof}

We need one last ingredient before proving the simple connectedness of the complex $\Comp$.

\begin{lemma} \label{lem:link1connected}
The link $\LL(v_1)$ of a vertex of type 1 is a connected graph.
\end{lemma}

\begin{proof}
By transitivity of the action of $\TA$ on vertices of type 1, it is sufficient
to work with the link $\LL([x_3])$.
Let $v_3 = \llb f_1, f_2, x_3 \rrb$ be a vertex of type 3 in $\LL([x_3])$,
where we choose our representative such that $\deg f_1 > \deg f_2$.

First observe that $v_3$ does not admit any elementary $K$-reduction.
Indeed by Corollary \ref{cor:SPF} the pivotal simplex of such a reduction should have Strong Pivotal Form $\PF(s)$ for some odd $s \ge 3$.
In particular there exist $\delta \in \N^3$ and a reordering $\{g_2, g_3\} = \{f_2, x_3\}$ such that
$$\deg f_1 = s\delta, \quad \deg g_2 = 2\delta \quad \text{ and } \quad s\delta > \deg g_3 > (s-2)\delta.$$ 
But since $\deg x_3 = (0,0,1)$ is the minimal possible degree of a component of an automorphism, both cases $g_2 = x_3$ or $g_3 = x_3$ are impossible.

It follows that $v_3$ also does not admit a proper $K$-reduction: such a reduction would be via $w_3 = \llb g_1, f_2, x_3 \rrb$, but we just proved that such a $w_3$ cannot admit an elementary $K$-reduction.

By Theorem \ref{thm:reducibility}, we conclude that $v_3$ admits an elementary
reduction $v_3'$, which clearly must admit $[x_3]$ as pivot, since there is no non-constant
polynomial $P \in \K[x_1,x_2,x_3]$ with $\deg x_3 > \deg P$.
In particular, $v_3'$ also is in $\LL([x_3])$, and by induction on degree, we obtain the existence of a reduction path from $v_3$ to $[\id]$ that stays in $\LL([x_3])$.
\end{proof}

Now we recover a result of \cite{U} and \cite{wright}, about relations in $\Tame(\A^3)$.
Precisely, Umirbaev gives an algebraic description of the relations, and Wright shows that this result can be rephrased in terms of an amalgamated product structure over three subgroups.
Our proof follows the same strategy as in \cite[Proposition 3.10]{BFL}.

\begin{proposition}\label{pro:1connected}
The complex $\Comp$ is simply connected.
\end{proposition}

\begin{proof}
Let $\gamma$ be a loop in $\Comp$. We want to show that it is homotopic to a trivial loop.
Without loss in generality, we can assume that the image of $\gamma$ is contained in the 1-skeleton of the square complex, and that $\gamma(0) = \llb x_1,x_2,x_3 \rrb$ is the vertex of type 3 associated with the identity.

A priori (the image of) $\gamma$ is a sequence of edges of arbitrary type.
By Lemma \ref{lem:link1connected}, we can perform a homotopy to avoid each vertex of type 1.
So now we assume that vertices in $\gamma$ are alternatively of type 2 and 3.
Precisely, up to a reparametrization we can assume that for each $i$, $\gamma(2i)$ has type 3 and $\gamma(2i+1)$ has type 2, so that $\gamma$ defines a combinatorial path by setting $v_3(i) = \gamma(2i)$ for each $i$.
By removing some of these vertices we can also assume that $\gamma$ is a locally geodesic loop.

Let $v_3(i_0)$ be the maximal vertex of the loop, and $\delta_0$ its degree.
Then by Lemma \ref{lem:local homotopy} we can conclude by induction on the couple $(\delta_0, i_0)$, ordered with lexicographic order.
\end{proof}

Since $\TA$ acts on a simply connected 2-dimensional simplicial complex with fundamental domain a simplex, we can recover the group from the data of the stabilizers of each type of vertex.
This is a simple instance of the theory of developable complexes of groups in the sense of Haefliger (see \cite[III.$\mathcal C$]{BH}).
Following Wright we can phrase this fact as follows:

\begin{corollary}[{\cite[Theorem 2]{wright}}] \label{cor:product}
The group $\TA$ is the amalgamated product of the three following subgroups along their pairwise intersections:
\begin{align*}
\Stab([x_1, x_2, x_3]) &= A_3; \\
\Stab([x_1,x_2]) &= \{(ax_1 + bx_2 + c, a'x_1 + b'x_2 + c', \alpha x_3 + P(x_1, x_2))\}; \\
\Stab([x_1]) &= \{(ax_1 + b, f_2(x_1, x_2, x_3), f_3(x_1, x_2, x_3) \}.
\end{align*}

\end{corollary}

\section{Examples} \label{sec:examples}

We gather in this last section a few examples of interesting reductions.

\begin{example}[Elementary $K$-reduction with $s = 3$] \label{exple:s=3}
Let
\begin{align*}
g &= (x_1,\, x_2,\, x_3 + x_1^2 - x_2^3), \\
t_1 &= (x_1 + \alpha x_2 x_3 + x_3^3,\, x_2 + x_3^2,\, x_3).
\end{align*}
Clearly in the composition $g \circ t_1$ the terms of degree 6 cancel each other.
Moreover, if we choose $\alpha = \tfrac{3}{2}$ this is also the case for terms of degree 5:
$$g \circ t_1 = \left(x_1 + \tfrac{3}{2} x_2 x_3 + x_3^3,\, x_2 + x_3^2,\,
x_3 + x_1^2 - x_2^3 + 3x_1x_2x_3 + \tfrac{x_3^2}4(8x_1x_3-3x_2^2) \right).$$
Consider now a triangular automorphism preserving the quadratic form $8x_1x_3-3x_2^2$ that appears as a factor:
$$ t_2 = (x_1,\, x_2 + x_1^2,\, x_3 + \tfrac34 x_1x_2 + \tfrac{3}{8} x_1^3).$$
A direct computation shows that the components of $f = g \circ t_1 \circ t_2 = (f_1,f_2,f_3)$ admits the following degrees:
$$(9,0,0), \; (6,0,0), \; (7,0,1).$$
Finally, $u_3 = [t_1 \circ t_2]$, whose degrees of components are
$$(9,0,0), \; (6,0,0), \; (3,0,0),$$
is an elementary $K$-reduction of $v_3 = [f_1,f_2,f_3]$.
Following notation from the definition of Strong Pivotal Form, we have $s = 3$ and $\delta = (3,0,0)$.
Moreover
\begin{multline*}
df_1 \wedge df_2 = -\tfrac32 (x_1^2 - x_2)dx_2 \wedge dx_3 + (\tfrac{27}{16}x_1^2x_2 + \tfrac98 x_2^2 + \tfrac32 x_1 x_3 + 1)dx_1 \wedge dx_2 \\
+ (-\tfrac94 x_1^3 - \tfrac32 x_1 x_2 + 2x_3) dx_1 \wedge dx_3.
\end{multline*}
so that $\deg df_1 \wedge df_2 = (4,0,1)$, from the contribution of the factor $x_1^3 dx_1 \wedge dx_3$.

\end{example}

\begin{example}[Elementary $K$-reduction with $s = 5$, {see also \cite{Ku:type1}}] \label{exple:s=5}
One can apply the same strategy to produce examples of $K$-reduction with $s \ge 3$ an arbitrary odd number.
We give the construction for $s = 5$, and leave the generalization to the
reader.
Let
\begin{align*}
g &= (x_1,\, x_2,\, x_3 + x_1^2 - x_2^5), \\
t_1 &= (x_1 + \alpha x_2^2 x_3 + \beta x_2 x_3^3 + x_3^5,\, x_2 + x_3^2,\, x_3).
\end{align*}
Observe that $\alpha x_2^2 x_3 + \beta x_2 x_3^3 + x_3^5$ is homogeneous of degree 5, by putting weight 1 on $x_3$ and weight 2 on $x_2$.
By choosing $\alpha = \tfrac{15}{8}$, $\beta = \tfrac{5}{2}$, we minimize the degree of the composition:
$$ g \circ t_1 = \left(
x_1 + \tfrac{15}{8} x_2^2 x_3 + \tfrac{5}{2} x_2 x_3^3 + x_3^5,\, x_2 + x_3^2,\,
x_3 + \tfrac{1}{8}x_3^4(16x_1x_3 - 5 x_2^3) + \cdots \right).$$
Now take the following triangular automorphism, which preserves the polynomial $16x_1x_3 - 5 x_2^3$:
$$ t_2 = \left(x_1,\, x_2 + x_1^2,\, x_3 + \tfrac{5}{16}(3x_1x_2^2 + 3x_1^3x_2 + x_1^5)\right).$$
We compute the degrees of the components of $f = g \circ t_1 \circ t_2 = (f_1, f_2, f_3)$:
$$(25,0,0), \; (10,0,0), \; (20,3,0).$$
Finally, $u_3 = [t_1 \circ t_2]$, whose degrees of components are
$$(25,0,0), \; (10,0,0), \; (5,0,0),$$
is an elementary $K$ reduction of $v_3 = [f_1, f_2, f_3]$.
Here we have $s = 5$ and $\delta = (5,0,0)$.
Moreover
\begin{multline*}
df_1 \wedge df_2 = -\tfrac{15}8 \left( x_1^2 + x_2 \right)^2 dx_2 \wedge dx_3
+ \left( 2x_3- \tfrac{5}{8}(5 x_1^5 - 9 x_1^3x_2 - 3 x_1x_2^2)  \right) dx_1 \wedge dx_3 \\
+ \left( \tfrac{75}{128}(5 x_1^4 x_2^2 + 9 x_1^2 x_2^3 + 3 x_2^4) + \tfrac{15}8 ( x_1^3x_3 + 2 x_1 x_2 x_3) + 1 \right)dx_1 \wedge dx_2,
\end{multline*}
so that $\deg df_1 \wedge df_2 = (5,3,0)$, from the contribution of the factor $x_1^4x_2^2 dx_1 \wedge dx_2$.
\end{example}

\begin{example}[Elementary reduction without Strong Pivotal Form] \label{exple:noSPF}
We give examples of elementary reduction that show that the three assumptions in Proposition \ref{pro:SPF} are necessary to get Strong Pivotal Form.
\begin{enumerate}[wide]
\item Let $f = (f_1, f_2, f_3) \in \TA$ and $r \ge 2$ such that
$$\deg f_1 > \deg f_3 = r\deg f_2.$$
In particular there exists $a \in \K$ such that $\deg f_3 > \deg (f_3 + a f_2^r)$.
For instance $f= (x_1 + x_3^3, x_2, x_3 + x_2^2)$ is such an automorphism, for $r = 2$ and $a = -1$.
Then $u_3 = \llb f_1, f_2, f_3 + a f_2^r \rrb$ is an elementary reduction of $v_3 = \llb f_1, f_2, f_3 \rrb$, and the pivotal simplex does not have Strong Pivotal Form.
Observe that $v_2 = \llb f_1, f_2 \rrb$ has outer resonance in $v_3$.

\item Let $u_3 = \llb f_1, f_2, f_3 + P(f_1,f_2)\rrb$ be an elementary $K$-reduction of $w_3 = \llb f_1, f_2, f_3 \rrb$, with $2\deg f_1 = s\deg f_2$ for some odd $s \ge 3$.
For instance we can start with one of the examples \ref{exple:s=3} or \ref{exple:s=5}.
Pick any integer $r \ge \frac{s+1}{2}$.
Then $v_3' = \llb f_1+f_2^r, f_2, f_3 + P(f_1-f_2^r,f_2) \rrb$ is an elementary reduction of $v_3 = \llb f_1 + f_2^r, f_2, f_3 \rrb$, and the pivotal simplex does not have Strong Pivotal Form.
Observe that the center $v_2 = \llb f_1 + f_2^r, f_2 \rrb$ has inner resonance.

\item With the notation of Example \ref{exple:s=3} or \ref{exple:s=5}, the elementary reduction from $\llb g \circ t_1 \rrb$ to $\llb t_1 \rrb$ gives an example of an elementary reduction where the center $m_2$, which is the minimal line, does not have inner or outer resonance, and again the pivotal simplex does not have Strong Pivotal Form.
\end{enumerate}
\end{example}

\begin{example}[Non-normal proper $K$-reduction]
Pick the elementary $K$-reduction from Example \ref{exple:s=5}, and set $v_3' = [f_1 + f_2^2, f_2, f_3]$, which is a weak elementary reduction of $v_3$.
Then $u_3$ is a non-normal proper $K$-reduction of $v_3'$, via $v_3$.
This corresponds to Case \ref{Kstability:1} of Stability of a $K$-reduction
\ref{pro:K stability}.
We mention again that it is an open question whether there exists any normal proper $K$-reduction.
\end{example}

\begin{nonexample}[Hypothetical type II and type III reductions] \label{nonexample}
From Corollary \ref{cor:SPF} we know that if $v_3$ admits an elementary $K$-reduction, then the pivotal simplex  
has Strong Pivotal Form $\PF(s)$ for some odd $s \ge 3$.
In particular if $s = 3$, then $v_3 = \llb f_1, f_2, f_3 \rrb$ with $\deg f_1 = 3\delta$, $\deg f_2 = 2 \delta$ and $\deg f_3 > \delta$ for some $\delta \in \N^3$.
It is not clear if there exists an example of such a reduction with $\frac32\delta > \deg f_3$, or even $2\delta > \deg f_3$.
Observe that such an example would be the key for the existence  of the following reductions (for the definition of a reduction of type II or III, see the original paper \cite{SU:main}, or \cite[\S7]{Ku:main}):
\begin{enumerate}
\item If $\frac32\delta > \deg f_3$, then $v_3' = \llb f_1 + f_3^2, f_2, f_3 \rrb$ would admit a normal proper $K$-reduction, via $v_3$: This would correspond to a type III reduction.
\item If $2\delta > \deg f_3 > \frac32\delta$, then $v_3$ would admit an elementary $K$-reduction such that the pivot $[f_2]$ is distinct from the minimal vertex $[f_3]$: This would correspond to a type II reduction.
\item If $\frac32\delta > \deg f_3$, then $v_3'' = \llb f_1 , f_2 + f_3^2, f_3 \rrb$ would be an example of a vertex that admits a reduction along a center with outer resonance in $v_3''$, but that does not admit a reduction with center the minimal line of $v_3''$ (see Lemma \ref{lem:outer resonance}).
\end{enumerate}
\end{nonexample}

\bibliographystyle{myalpha}
\bibliography{biblio}

\end{document}